\documentclass{article}
\usepackage[T1]{fontenc}
\usepackage[utf8]{inputenc}

\usepackage{geometry}
\usepackage[dvipsnames]{xcolor}
\definecolor{myred}{HTML}{c20014}
\definecolor{mygreen}{HTML}{008000}

\usepackage{amsmath}
\usepackage{amsfonts}
\usepackage{caption}
\usepackage{array,multirow,makecell}
\usepackage{amsthm}
\usepackage{amssymb}

\usepackage{verbatim}
\usepackage{dsfont}
\usepackage[inline]{enumitem}
\setlist{noitemsep}
\usepackage{csquotes}
\usepackage{cite}
\usepackage{sectsty}
\usepackage{tikz-cd}

\usepackage{accents}

\usepackage[multiple]{footmisc}

\usepackage{url}
\usepackage[colorlinks=true,linkcolor=myred,citecolor=mygreen]{hyperref}
\hypersetup{linktocpage}
\usepackage[capitalise]{cleveref}

\crefname{equation}{}{}
\Crefname{equation}{Equation}{Equations}

\numberwithin{equation}{section}

\crefrangeformat{enumi}{#3#1#4 through~#5#2#6}

\newtheorem{theo}{Theorem}[section]
\newtheorem{prop}[theo]{Proposition}
\newtheorem{coro}[theo]{Corollary}
\newtheorem{lemma}[theo]{Lemma}

\theoremstyle{definition}
\newtheorem{defi}[theo]{Definition}

\theoremstyle{remark}
\newtheorem{rem}[theo]{Remark}

\crefname{theo}{Theorem}{Theorems}
\crefname{claim}{Claim}{Claims}
\crefname{defi}{Definition}{Definitions}
\crefname{enumi}{}{}
\Crefname{enumi}{Item}{Items}
\crefname{rem}{Remark}{Remarks}
\crefname{prop}{Proposition}{Propositions}

\newcommand{\dt}[1]{\frac{\mathrm{d}#1}{\mathrm{d}t}}
\newcommand{\dl}[2]{\frac{\partial#1}{\partial#2}}

\newcommand{\D}{\mathcal{D}}

\renewcommand{\d}{\mathrm{d}}

\renewcommand{\L}{\mathcal{L}}

\renewcommand{\S}{\mathcal{S}}

\newcommand{\R}{\mathbb{R}}

\newcommand{\N}{\mathbb{N}}
\newcommand{\C}{\mathcal{C}}
\newcommand{\Z}{\mathbb{Z}}

\newcommand{\Op}{\mathcal{O}}

\renewcommand{\i}{\mathrm{i}}

\DeclareMathOperator{\id}{id}
\DeclareMathOperator{\re}{Re}
\DeclareMathOperator{\im}{Im}

\newcommand{\Co}{\mathbb{C}}
\newcommand{\F}{\mathcal{F}}

\newcommand{\inte}{\mathrm{in}}
\newcommand{\fra}{\mathrm{fr}}
\newcommand{\dom}{\operatorname{dom}}
\newcommand{\ran}{\operatorname{ran}}

\newcommand{\sg}[1]{\{#1_t\}_{t \geq 0}}

\newcommand{\hl}[1]{(0, +\infty; #1)}

\newcommand{\para}[1]{\medskip \noindent {\bfseries #1.}}

\newcommand{\U}{\mathcal{U}}

\newcommand{\Chi}{\mathrm{X}}
\newcommand{\Rho}{\mathrm{P}}

\newcommand{\vnabla}{\vec{\nabla}}
\newcommand{\vn}{{\vec{n}}}

\providecommand{\keywords}[1]
{
  \small	
  \textbf{{Keywords.}} #1
}
\providecommand{\AMS}[1]
{
  \small	
  \textbf{{AMS subject classifications.}} #1
}

\renewcommand{\leq}{\leqslant}
\renewcommand{\geq}{\geqslant}
\renewcommand{\epsilon}{\varepsilon}
\newcommand{\eps}{\varepsilon}

\title{Admissibility theory in abstract Sobolev scales and transfer function growth at high frequencies\thanks{This work was supported by the Research Council of Finland Grant number 349002.}}
\author{Lassi Paunonen\thanks{Mathematics Research Centre, Faculty of Information Technology and Communication Sciences, Tampere University, P.O. Box 692, 33101 Tampere, Finland. Emails: \nolinkurl{name.surname@tuni.fi}.}%
\and David Seifert\thanks{School of Mathematics, Statistics and Physics, Newcastle University, Herschel Building, Newcastle upon Tyne, NE1 7RU, United Kingdom. Email: \nolinkurl{david.seifert@ncl.ac.uk}.}%
\and Nicolas Vanspranghe\footnote{Corresponding author.}
\footnotemark[2]
}

\begin{document}

\maketitle

\abstract{

For strongly continous semigroups on  Hilbert spaces, we investigate admissibility properties of control and observation operators shifted along continuous scales of spaces built by means of either interpolation and extrapolation or functional calculus. Our results show equivalence of 
 admissibility in, on the one hand, a fractional domain of the generator and, on the other hand, a (different, in general) quadratic  interpolation space of the same ``Sobolev order''. Furthermore, such properties imply quantified resolvent bounds in the original state space topology. When the semigroup is a group, the resulting frequency-domain estimates are in fact equivalent to the aforementioned time-domain properties. In the case of systems with both control and observation, we are able to translate input-output regularity properties into high-frequency growth rates of operator-valued transfer functions. As an application, based on results by Lasiecka, Triggiani and Tataru on interior and boundary regularity of the wave equation under Neumann control, we derive optimal asymptotics for the Neumann-to-Dirichlet wave transfer function. With that in hand, we establish non-uniform energy decay rates for the wave equation posed in a rectangle and subject to Neumann  damping on an arbitrary open subset of the boundary.

}


\AMS{47D06, 34G10, 93B28 (35L90, 35L05, 35B40).}

\tableofcontents

\section{Introduction}

\subsection{Orientation and sample results}
In the control theory of linear semigroups on Hilbert spaces,
\emph{admissible} control  operators are those through which $L^2$-inputs produce well-defined solutions in the original state space. On the other hand, admissible observation operators give outputs that are well-defined in the $L^2$-sense for \emph{all} semigroup solutions. The question of admissibility is relevant for controls and observations that cannot be modelled as bounded linear maps into or from the space on which the semigroup operates. Most notably this applies to boundary control problems for partial differential equations. For more insight into classical admissibility theory, the reader is referred to  \cite{Wei89, Wei89b,TucWei09book} or the survey article \cite{JacPar04survey}; also, proper definitions are recalled in \cref{sec:main-res} below.
Despite their ubiquitous character in semigroup- and system-theoretic approaches (see, e.g., \cite{RamTak05,Mil05,CurWei06,TucWei14survey}), admissible operators exclude a number of classical  models in partial differential equations. A prime example 
is the multidimensional wave equation with Neumann boundary control in the space of finite-energy data \cite{LasTri90,Tat98}, which we return to in greater detail in \cref{sec:intro-edp}.

The goal of the present paper is to investigate, in the semigroup setting, relaxed admissibility-type conditions formulated in terms of abstract Sobolev scales, and derive quantitative connections with frequency-domain properties.
We now, slightly informally, present some of our main results;
 precise definitions and statements will follow in \cref{sec:preliminary,sec:main-res}.

Let $A$ be the infinitesimal generator of a strongly continuous semigroup $\sg{S}$ on a Hilbert space $X$. 
Filling the gaps
 in the standard two-sided discrete Sobolev tower $\{X_n\}_{n \in \Z}$ based on $A$, we
define two continuous Hilbert scales $\{X^\fra_s\}_{s \in \R}$ and $\{X^\inte_s\}_{s \in \R}$ composed of fractional domains of $-A$ and quadratic interpolation spaces, respectively.
Consider an $A$-bounded control operator $B$, that is, a bounded linear map $B$ from some other Hilbert space $U$ into the extrapolation space $X_{-1}$.
\begin{theo}
\label{th:B-sample}
Let $0 \leq \eta \leq 1$. Consider the inhomogeneous Cauchy problem
\begin{equation}
\label{eq:sample-cauchy}
\dot{x} = Ax + Bu, \quad x(0) = 0,
\end{equation}
and let $T >0$.
The following are equivalent: 
\begin{enumerate}[label=(\roman*)]
  \item  For all controls $u \in L^2(0, T; U)$, the final state $x(T)$ lies in $X_{-\eta}^\fra$;
  \item  For all controls $u \in L^2(0, T; U)$, the final state $x(T)$ lies in $X_{-\eta}^\inte$;
  \item  For all controls $u \in H^\eta(0, T; U)$, the final state $x(T)$ lies in $X$.
\end{enumerate}
\end{theo}
It is worth noting that the two first conditions of \cref{th:B-sample} are equivalent even though the spaces $X_{-\eta}^\inte$ and $X_{-\eta}^\fra$ need not coincide in general; see \cref{rem:int-fra} below. Both have their own benefits: the scaling operators associated with $\{X^\fra_s\}_{s \in \R}$ lie in the functional calculus of the generator $A$ and, as such, commute with the semigroup; on the other hand, $\{X_s^\inte\}_{s\in\R}$ is an (exact) interpolation scale and, in practice, identification with standard function spaces might be easier.
 \Cref{th:B-sample} also means that one may return to the original state space $X$ at the cost of $\eta$-smoother controls in the usual Sobolev sense. 
\begin{theo}
\label{theo:B-sample-FQ}
Let $\sigma$ be any real number larger than the growth bound of $\sg{S}$. Given $0 \leq \eta \leq 1$, if one of the equivalent properties of \cref{th:B-sample} holds, then
\begin{equation}
\label{eq:B-sample-FQ}
\|(\sigma + \i \omega - A)^{-1}B\|_{\L(U, X)} = O(|\omega|^{\eta}), \quad \omega \in \R, \quad |\omega| \to + \infty.
\end{equation}
If $\sg{S}$ is right-invertible, then \cref{eq:B-sample-FQ} is equivalent to the properties of \cref{th:B-sample}.
\end{theo}
In the right-invertible case (typically, when the semigroup is a group), \cref{theo:B-sample-FQ} establishes an exact correspondence between the lack of admissibility as measured by the equivalent criteria of \cref{th:B-sample} and the growth rate at high frequencies of the operator-valued function $p \mapsto (p - A)^{-1}B$. Note here that the estimate \cref{eq:B-sample-FQ} is given in the  uniform operator topology from $U$ to $X$: with \cref{theo:B-sample-FQ}, one can trade boundedness properties in ``weaker'' spaces for resolvent growth at the original $X$-level. The precise statements of \cref{th:B-sample,theo:B-sample-FQ} are contained in \cref{theo:input-group} below. Their analogues for $A$-bounded observation operators $C$, involving positive order spaces for ``smoother'' initial data and negative order spaces for distributional outputs, are given in \cref{theo:inter-group}.

When simultaneously dealing with $A$-bounded control and observation operators $B$ and $C$, we are moreover able to translate input-output regularity properties into high-frequency growth rates of the transfer function $p \mapsto C(p - A)^{-1}B$.  The next result is valid under a (mild) technical assumption guaranteeing that $C(p-A)^{-1}B$ is a  bounded linear map from the input space $U$ into the output space $Y$; see \cref{eq:hyp-Z} below.
\begin{theo}
\label{theo:sample-IO}
Let $0 \leq \eta_1, \eta_2 \leq 1$. Assume that $B$ and $C$ satisfy the conditions of \cref{th:B-sample} and their observation operator analogues with parameters $\eta_1$ and $\eta_2$, respectively.
Choose $\sigma$ as in \cref{theo:B-sample-FQ} and let $T > 0$. The following are equivalent:
\begin{enumerate}[label=(\roman*)]
  \item For all $u \in L^2(0, T; U)$, the solution $x$ to \cref{eq:sample-cauchy} satisfies $Cx \in H^{-\eta_1 - \eta_2}(0, T; Y)$;
  \item We have
\begin{equation}
  \label{eq:TF-sample-FQ}
  \|C(\sigma + \i \omega - A)^{-1}B\|_{\L(U, Y)} = O(|\omega|^{\eta_1 + \eta_2}), \quad \omega \in \R, \quad |\omega| \to + \infty.
  \end{equation}
\end{enumerate}
\end{theo}
 \Cref{theo:sample-IO}, whose detailed version is \cref{theo:transfer-function} below, is designed with applications to partial differential equations in mind. It allows us to turn sharp time-domain regularity available in the literature into frequency-domain estimates that, in many cases of interest, improve on those obtained solely by means of functional calculus and trace theory.

\subsection{Motivation and applications to partial differential equations}
\label{sec:intro-edp}

Having reviewed our contributions in the abstract setting, we now 
turn to partial differential equations. 
Admissible control and observation operators often translate ``hidden regularity'' properties; see, e.g., \cite[Remark 2.10] {Cor07book}.
 For instance, given a domain $\Omega$ in $\R^d$, $d \geq 2$, consider the wave equation $(\partial_t^2 - \Delta)w = 0$ posed in $\Omega \times (0, T)$  with homogeneous Dirichlet boundary condition $w = 0$ on $\partial \Omega \times (0, T)$. A relatively elementary differential multiplier argument reveals that all finite-energy solutions $w$ have 
a Neumann boundary trace
 $\partial_{\vec{n}} w $ that is well-defined in $L^2(\partial \Omega \times (0, T))$;
see \cite[Theorem 2.2]{Kom94book}. In other words, the Neumann trace defines an admissible observation operator for the wave group on $H^1_0(\Omega) \times L^2(\Omega)$. It follows from a duality argument that Dirichlet boundary data $u \in L^2(\partial \Omega \times (0, T)$) generate solutions in the (weaker but natural) state space $L^2(\Omega) \times H^{-1}(\Omega)$. For a full account of the Dirichlet problem, the reader is referred to \cite{LasLio86}.

\para{Boundary and interior regularity of waves under Neumann control} The case of inhomogeneous Neumann boundary conditions is more delicate.
It was first shown by Lasiecka and Triggiani that, in general,  Neumann boundary data $u \in L^2(\partial \Omega \times (0, T))$ fail to produce finite-energy solutions; see in particular  \cite[Theorem 2.3]{LasTri90} for a counterexample. In a series of papers \cite{LasTri81,LasTri90,LasTri91}, they investigated the regularity of solutions to the wave equation and their boundary traces under Neumann input in the case of a sphere or a rectangle (using direct eigenfunction expansions) and a general smooth domain (using pseudodifferential calculus). In the smooth case, these results were further refined by Tataru in \cite{Tat98}.
Define the ``loss of derivative'' parameter $\eta$ as follows:
\begin{equation}
\label{eq:cases-eta}
\eta \triangleq 1/3 ~\mbox{for general (smooth)}~\Omega, \quad \eta \triangleq 1/4 ~\mbox{if $\partial\Omega$ is flat}, \quad
\eta \triangleq 1/6 ~\mbox{if $\partial \Omega$ is concave}.
\end{equation}
For the precise meaning of ``flat'' and ``concave'', we refer the reader to \cite{Tat98}. What follows is a slightly modified\footnote{See \cref{rem:reg-wave} below.} version of \cite[Theorem 9]{Tat98}.
\begin{theo}[\!\!\cite{Tat98}]
\label{theo:wave-neumann} Assume that $\Omega$ is smooth. For all $u \in L^2(\partial \Omega \times (0, T))$, the corresponding solution $w$ to the initial boundary value problem
\begin{equation}
\label{eq:wave-neumann}
(\partial_t^2 - \Delta)w = 0 \quad \mbox{in}~ \Omega \times (0, T), \quad \partial_{\vec{n}}w = u \quad \mbox{on}~\partial \Omega \times (0, T), \quad (w, \partial_t w)|_{t = 0} = 0,
\end{equation}
satisfies
\begin{equation}
\label{eq:wave-neumann-regularity}
(w, \partial_t w)|_{t = T} \in H^{1-\eta}(\Omega) \times H^{-\eta}(\Omega), \quad w|_{\partial \Omega} \in H^{1-2\eta}(\partial \Omega \times (0, T)).
\end{equation}
\end{theo}
Tataru's theorem is sharp in the sense that, in each case described by \cref{eq:cases-eta}, one can construct geometries where \cref{eq:wave-neumann-regularity} would fail to hold for any better  exponent $\eta$. By \cite[Theorem A]{LasTri91},
when $\Omega$ is a rectangle in $\R^2$,
 the results of \cref{theo:wave-neumann} also hold with $\eta = 1/4 + \eps$ for arbitrary $\eps > 0$.

\para{Transfer function asymptotics for waves}
We shall combine our semigroup-theoretic apparatus with the above regularity results to obtain new high-frequency bounds for the Neumann-to-Dirichlet wave transfer function.  In what follows, $\Delta_N$ denotes the $L^2(\Omega)$-realisation of the Laplacian with zero Neumann boundary condition, $\gamma$ is the trace operator and $\gamma^\ast$ is its adjoint.

\begin{theo} \label{theo:wave-neumann-FQ}
Assume that $\Omega$ is either a rectangle, a smooth bounded domain or a suitable\footnote{Here, largely for simplicity, we take $\Omega$ to be either a half-space or an unbounded domain with bounded boundary.} smooth unbounded domain.
As  $|\lambda| \to + \infty$, $\lambda \in \R$,
\begin{subequations}
\label{eq:neumann-dirichlet-FQ}
\begin{align}
\label{eq:neumann-dirichlet-L2}
&\|\gamma ((1 + \i \lambda)^2 - \Delta_N)^{-1}\gamma^\ast\|_{\L(L^2(\partial \Omega))} = O(|\lambda|^{2\eta - 1}), \\
&\|\gamma ((1 + \i \lambda)^2 - \Delta_N)^{-1}\gamma^\ast\|_{\L(L^2(\partial \Omega), H^1(\partial \Omega))} = O(|\lambda|^{2\eta}).
\end{align}
\end{subequations}
\end{theo}

The exponents in \cref{eq:wave-neumann-regularity} are optimal, at least in the smooth case; as a matter of fact, one could use our results to deduce \cref{theo:wave-neumann} from \cref{theo:wave-neumann-FQ}. On the other hand, a ``geometry-agnostic'' approach relying solely on continuity of the trace map from $H^{1/2+\eps}(\Omega)$ into $L^2(\partial\Omega)$, $\eps > 0$,  would lead to a strictly worse parameter $\eta = 1/2 + \eps$ in \cref{eq:neumann-dirichlet-FQ}.
The reader may notice a resemblance between
 \cref{eq:neumann-dirichlet-FQ} and
 wavenumber-explicit Neumann-to-Dirichlet bounds for  Helmhotz-type equations, as considered in scattering theory; see, e.g., \cite{Spe14,BasSpe15}.
The transfer function asymptotics of \cref{theo:wave-neumann-FQ}, and especially \cref{eq:neumann-dirichlet-L2}, have a wide range of applications, some of which are presented in this manuscript.
They are especially useful in, e.g., applying ``semi-abstract'' techniques in control and stability  problems, as we now illustrate.



\para{Energy decay rates for waves with Neumann feedback} 
For a class of ``well-posed, conservative'' systems, Russell's principle \cite{Rus78survey} states that exact \emph{observability} of the uncontrolled dynamics is \emph{equivalent} to  uniform exponential stability  under a natural output feedback. In boundary stabilisation problems this control often takes the form of a dissipative boundary condition, and the well-posedness condition essentially amounts to continuity on $L^2(\partial \Omega \times (0, T))$ of a certain  ``boundary to boundary'' (input-output) operator. For \cref{eq:wave-neumann}, this operator would be the map $u \mapsto \partial_t w|_{\partial \Omega}$, which, as we have just seen,  fails to meet that requirement; see also the discussion in \cite{LasTri03}.
In the context of \emph{non-uniform} stability, the analysis in \cite{AmmTuc01,AnaLea14,ChiPau23,KleWan23arxiv} reveals that Russell's principle remains partially effective:  relaxed but quantified observability properties yield (possibly non-optimal) energy decay rates for smooth feedback solutions. As shown in \cite{ChiPau23}, in the case of an unbounded control operator, this procedure requires knowledge of transfer function bounds such as \cref{eq:neumann-dirichlet-FQ}.
Following that spirit, we combine our results  with observability properties established in \cite{RamTak05,TenTuc09} to derive explicit decay rates for the boundary-stabilised wave equation on a rectangle.
\begin{theo}
\label{theo:wave-neumann-stab} Let $\Omega$ be a rectangle of $\R^2$ and let $b \in L^\infty(\partial \Omega)$, $b \geq 0$, be such that, for some nonempty open subset $\Op$ of $\partial \Omega$ and some positive constant $b_0$, $b \geq b_0$ a.e.\  on $\Op$. Let $(w_0, w_1) \in H^1(\Omega) \times H^1(\Omega)$ be initial data satisfying $\Delta w_0 \in L^2(\Omega)$ and the compatibility condition $\partial_\vn w_0 = -b^2 w_1$ on $\partial \Omega$, and let $w$ be the corresponding solution to the initial boundary value problem
\begin{equation}
\label{eq:wave-neumann-stab}
(\partial_t^2 - \Delta)w = 0 \quad \mbox{in}~ \Omega \times (0, +\infty), \quad \partial_{\vec{n}}w = - b^2 \partial_t w \quad \mbox{on}~\partial \Omega \times (0, +\infty), \quad (w, \partial_t w)|_{t = 0} = (w_0, w_1).
\end{equation}
Then, for any fixed $\eps > 0$,
\begin{equation}
E(t; w) \triangleq
\frac{1}{2} \int_\Omega |\partial_t w(x, t)|^2 + \|\vec{\nabla} w(x, t)\|^2 \, \d x  = o(t^{-2/3 + \eps}), \quad t \to + \infty.
\end{equation}
\end{theo}
The energy decay in \cref{theo:wave-neumann-stab} is unlikely to be sharp and the exponent remains $1/3 + \eps$ off the \emph{a priori} rate $o(t^{-1})$ that one may expect under Schr\"odinger group observability; see \cite[Part II]{AnaLea14}, and also \cite[Remark 2.6]{AbbNic15}. The latter reference suggests that, specifically when $b$ is constant and nonzero on some side of the rectangle and zero elsewhere, the energy decay rate $o(t^{-1}$) can be established through separation of variable and explicit Fourier expansion of the feedback solutions.
However, for general control function $b$ the result of \cref{theo:wave-neumann-stab} is, to our knowledge, new and also the best possible among currently available blackbox methods. Note that, for general (smooth) bounded domains $\Omega$, \cref{theo:apriori-decay-waves} below gives \emph{a priori} energy decay rates under the so-called non-uniform Hautus test.

\subsection{Additional background and outline}

In the context of abstract linear systems on Hilbert spaces, admissible control and observation operators in their current form have been introduced in \cite{Wei89,Wei89b}, but the concept can be traced back to \cite{HoRus83} and also appears in, e.g., \cite{Sal87}. Admissibility in the Banach space setting is also a topic of interest \cite{HaaKun07,JacSch19,AroGlu24arxiv,PreSch24arxiv}; note however that most work in that direction relies on various special features of the semigroup, the state space or the space of controls. Several  relaxed admissibility-type properties can be found in the literature: for instance, a weighted  property called ``$\alpha$-admissibility'', which is tailored to analytic semigroups, is considered in \cite{HaaLeM05,HaaKun07}; also, \cite{LatRan05} introduces ``$\beta$-admissibility'' in order to deal with controls and observations that cannot be modelled as $A$-bounded operators.

In \cite{GuiOpm23,GuiLog24}, the authors investigate causal translation-invariant linear operators between vector-valued Bessel potential spaces and their relationship with polynomially-growing  multipliers in the frequency domain.
Their results bear some similarity with our \cref{theo:sample-IO}, although the setting and the techniques required differ substantially.
Our framework should make it possible to combine regularity properties in \cref{theo:sample-IO} and semigroup stability  in order to verify the input-output ``Sobolev stability'' property introduced in \cite{GuiLog24}; see also \cite{GuiLog24b} for related system-theoretic applications.




The remainder of the paper is organised as follows. \Cref{sec:preliminary} contains notation, definitions and, most importantly, the precise construction of the semigroup-theoretic Sobolev scales under consideration in our work, along with their basic properties. In \Cref{sec:main-res} we state the complete versions of our main results. \Cref{sec:proof} is devoted to their proofs. In \cref{sec:second-order} we apply our main results to a special class of  second-order  systems, which in particular encompasses the wave model \cref{eq:wave-neumann}. In \cref{sec:wave}, we carry out our case study of the wave equation with Neumann boundary data \cref{eq:wave-neumann}. In addition to proving \cref{theo:wave-neumann-FQ}, we also give sharp regularity results for the Schr\"odinger equation posed in the half-space and subject to a Neumann boundary condition. \Cref{sec:decay} is devoted to observability-based energy decay results for second-order systems under dissipative feedback; there, we prove \cref{theo:wave-neumann-stab}, which pertains to the wave equation with Neumann boundary damping, but the abstract results of \cref{sec:stability} may be of independent interest to the reader. Finally, \cref{sec:conc} concludes our paper, and \cref{sec:add-hilb,sec:cutoff} contain various technical results used throughout the manuscript.

\section{Preliminaries}
\label{sec:preliminary}

\subsection{Notation and definitions}
\label{sec:def}



\para{Basic notation} In this paper, $\R^+$ and $\N$ stand for the sets of nonnegative reals and integers, respectively. We make standard use of asymptotic ``big-O'' and ``small-o'' notation. The letter $\C$ indicates various spaces of continuous (often vector-valued) functions. Also, unless stated otherwise, $K, K', \dots$ denote generic positive constants that may change from line to line but remain independent of all variables of interest in a given situation.

\para{Vector spaces}
Throughout this work, all vector spaces are over the ground field $\Co$ of complex numbers.
The norm of a normed vector space $E$ is denoted by $\|\cdot \|_E$. We use brackets $\langle \cdot, \cdot \rangle$ to write sesquilinear forms: if $E$ is a Hilbert space, $\langle \cdot, \cdot \rangle_E$ denotes its scalar product; if $E$ is a Banach space, $E^\ast$ stands for its topological antidual (equipped with the dual norm) and $\langle \cdot, \cdot \rangle_{E^\ast, E}$ indicates the antiduality pairing between $E^\ast$ and $E$. If $E$ and $F$ are Banach spaces, $\L(E, F)$ denotes the space of bounded linear operators between $E$ and $F$, equipped with the operator norm, which makes it a Banach space. Given a Banach space $E$, $\L(E)$ denotes the Banach algebra $\L(E, E)$. If $A : \dom(A) \subset E \to E$ is an (in general unbounded) linear operator, $\dom(A)$, $\ker(A)$ and $\ran(A)$ denote its domain, kernel and range, respectively. The resolvent set $\rho(A)$ of such an operator $A$ is the set of all complex numbers $\mu$ such that $\mu - A$ possesses a two-sided inverse in $\L(E)$. Here and in the sequel,
when $\mu$ is a complex number we shall use the same symbol to denote  the linear operator $\mu \id$, where $\id$ is the identity.

\para{Interpolation spaces} Let $E_0$ and $E_1$ be Hilbert spaces such that $E_1$ is a subspace of $E_0$ with continuous and dense embedding.
Given $0 < \theta <1$, $[E_1, E_0]_\theta$ denotes the (unique) \emph{geometric} interpolation space of exponent $\theta$ between $E_0$ and $E_1$; see \cite{McC92} in the separable case and \cite[Section 3]{ChaHew15}\footnote{
A corrigendum available online completes this reference.
} in general. As explained in \cite{ChaHew15},
this is, up to renormalisation, the same Hilbert space as produced by standard quadratic interpolation techniques, namely:
\begin{itemize*}
  \item the $K$- and $J$-methods with parameter $q = 2$ \cite[Chapter 3]{Ber76book};
  \item the methods of trace spaces with parameters $p = q = 2$ \cite{Lio61};
  \item complex interpolation \cite[Chapter 4]{Ber76book}; 
  \item diagonalisation of a positive operator $\Lambda$ on $E_0$ such that $\dom(\Lambda) = E_1$, see \cite[Section 2.1, Chapter 1]{LioMag68book} or \cite{ChaHew15}.
\end{itemize*}  For convenience of notation, we also let $[E_1, E_0]_1 \triangleq E_0$ and $[E_1, E_0]_0 \triangleq E_1$. 
Given two other Hilbert spaces $F_0$ and $F_1$ with continuous and dense embedding $F_1 \hookrightarrow F_0$ and a linear map  $S$ that is continuous from $E_0$ into $F_0$ and also from $E_1$ into $F_1$, then $S$ continuously maps $[E_1, E_0]_{\theta}$ into $[F_1, F_0]_\theta$ for any $0 \leq \theta \leq1$, and moreover
\begin{equation}
\label{eq:geo-inter}
\|S\|_{\L([E_1, E_0]_{\theta}, [F_1, F_0]_\theta)} \leq \|S\|_{\L(E_1, F_1)}^{1-\theta} \|S\|_{\L(E_0, F_0)}^\theta.
\end{equation}
Note the absence of indeterminate constant in \cref{eq:geo-inter}. When using operator norm estimates of this kind, we will write for short that we interpolate between $E_0 \to  F_0$ and $E_1 \to F_1$.



\para{Vector-valued Sobolev spaces} Let $E$ be a Hilbert space and let $- \infty \leq a < b \leq + \infty$.
The space of $E$-valued test functions, i.e., infinitely differentiable functions $\varphi : (a, b) \to E$ with compact support, is denoted by $\D(a, b; E)$.
 The space $\D'(a, b; E)$ of $E$-valued distributions is the space of linear mappings from $\D(a, b; E)$ into $E$ that are continuous in the following sense: $u \in \D'(a, b; E)$ if and only if for every compact set $K \subset (a, b)$, there exist a real number $M > 0$ and an integer $m \geq 0$ such that the inequality $\|u[\varphi]\|_E \leq M \sum_{j=0}^{m} \sup_{t \in K} \|(\d^j / \d t^j) \varphi(t)\|_E$ holds for all $\varphi \in \D(K, E)$.
Having identified (classes of) locally Bochner-integrable $E$-valued functions on $(a, b)$ with \emph{regular} distributions (see, e.g., \cite[Section 1.1, Chapter III]{Ama95book}), for $m \in \N$, we define the $L^2$-based $E$-valued Sobolev space 
\begin{equation}
\label{eq:def-sob}
H^m(a, b; E) \triangleq \left \{ u \in \D'(a, b; E) : \frac{\d^ju}{\d t^j} \in L^2(a, b; E),~ j = 0, \dots, m \right \},
\end{equation}
endowed with its natural Hilbert norm.
In \cref{eq:def-sob}, the derivatives are \emph{a priori} understood in the distributional sense: given $u \in \D'(a, b; E)$, $(\d / \d t)u$ is defined in $\D'(a, b; E)$ by $(\d / \d t)u[\varphi] = - u[(\d / \d t)\varphi]$ for all $\varphi \in \D(a, b; E)$. Following \cite[Theorem 1.2.2, Chapter III]{Ama95book} or \cite[Lemma 3.1, Chapter II]{Tem12book},
we can also characterise $H^1(a, b; E)$ as the space of all $u \in L^2(a, b; E)$ such that, for some $v \in L^2(a, b; E)$ and $u_0 \in E$, $u(\tau) = u_0 + \int_0^t v(t) \, \d t$ for a.e.\  $\tau \in (a, b)$. For such $u$, $(\d / \d t)u = v$ in $\D'(a, b; E)$ and we may identify $u$ with its unique continuous representative.
For reals $s = m + \theta$, $m \in \N$, 
$0 < \theta < 1$,
we let
\begin{equation}
\label{eq:def-inter}
H^{s}(a, b; E) \triangleq [H^{m+1}(a, b; E), H^m(a, b; E)]_{1 - \theta}.
\end{equation}
Up to equivalence of norms, the space defined by \cref{eq:def-inter} coincides with the $E$-valued Sobolev-Slobodeckii space $W^{s,2}(a, b; E)$ as used in, e.g., \cite{Sim90,Ama97}. For $s \geq 0$, we define $H^s_0(a, b; E)$ as the closure of $\D(a, b; E)$ in  $H^s(a, b; E)$; we then let
\begin{equation}
H^{-s}(a, b; E) \triangleq (H^s_0(a, b; E))^\ast.
\end{equation}
We may (and will) identify $L^2(a, b; E)$ with its antidual through the Riesz isomorphism. Furthermore, given real numbers $0 \leq s_1 \leq s_2$, we recall that $H_0^{s_2}(a, b; E) \hookrightarrow H_0^{s_1}(a, b; E) \hookrightarrow L^2(a, b; E)$ with continuous and dense embeddings. This results in a chain of continuous and dense embeddings
\begin{equation}
\label{eq:chain-H-spaces}
H^{s_2}_0(a, b; E) \hookrightarrow H_0^{s_1}(a, b; E) \hookrightarrow L^2(a, b; E) \hookrightarrow H^{-s_1}(a, b; E) \hookrightarrow H^{-s_2}(a, b; E), \quad 0 \leq s_1 \leq s_2,
\end{equation}
where the antiduality pairing between $H^{-s_2}(a, b; E)$ and $H^{s_2}_0(a, b; E)$ extends that between $H^{-s_1}(a, b; E)$ and $H^{s_1}_0(a, b; E)$, which in turn extends the scalar product in $L^2(a, b; E)$.
In particular, 
\begin{equation}
\label{eq:dual-norm}
\|u\|_{H^{-s}(a, b; E)} 
 = \sup_{\varphi \in \D(a, b; E)\setminus \{0\}} \frac{|\langle u, \varphi\rangle_{L^2(a, b; E)}|}{\|\varphi\|_{H^s(a, b; E)}}, \quad u \in L^2(a, b; E), \quad s \geq 0.
\end{equation}
Finally, Sobolev spaces enjoy the \emph{restriction property}: given $s \in \R$ and two nonempty open intervals $I$ and $J$ of $\R$ with $J \subset I$, the restriction map $u \to u|_{J}$ is continuous from $H^s(I, E)$ into (in fact, onto) $H^s(J, E)$.

\begin{rem}
It is more common to define negative order Sobolev spaces as \emph{duals} of $H^s_0$-spaces instead of antiduals. The distinction is insignificant in practice as antidual and dual spaces are trivially isomorphic via the antilinear map $\langle u, \cdot \rangle \mapsto \overline{\langle u, \cdot\rangle}$. We find that antiduals are more convenient when dealing with Gelfand triples and adjoint operators.
\end{rem}

 \para{The Fourier transform} The Schwartz space of smooth rapidly decreasing $E$-valued functions on $\R$  is denoted by $\S(\R, E)$ and endowed with its canonical Fr\'echet space topology.
The space $\S'(\R, E)$ of $E$-valued tempered distributions is the space of continuous linear maps from $\S(\R, E)$ into $E$. It is equipped with the topology of uniform convergence on bounded sets.
The Fourier transform
\begin{equation}
\F[\varphi](\omega) \triangleq \int_\R e^{- \i \omega t} \varphi(t) \, \d t, \quad \omega \in \R, \quad \varphi \in \S(\R, E),
\end{equation}
is an isomorphism from $\S(\R, H)$ onto itself. Furthermore its extension to $\S'(\R, E)$, defined by
$
\F[u]\varphi \triangleq u \F[\varphi]
$ for all $u \in \S'(\R, E)$ and $\varphi \in \S(\R, E)$,
is also an isomorphism from $\S'(\R, E)$ onto itself. 
 With the Fourier transform  we can characterise the spaces $H^{s}(\R, E)$ as  Bessel potential spaces.
More precisely, using, e.g.,
 \cite[Theorems 4.4.2 and 4.5.1, Chapter VII]{Ama19book} 
we see that, for all $s \in \R$,
\begin{equation}
\label{eq:bessel-spaces}
H^{s}(\R, E) = \left \{ u \in \S'(\R, E) :   (1 + \cdot^2)^{s/2} \F[u] \in L^2(\R, E) \right \}
\end{equation}
 and $(\int_\R (1 + \omega^2)^s \|\F[\cdot](\omega)\|^2_E \, \d \omega)^{1/2}$ is an equivalent norm on $H^{s}(\R, E)$. Note that, for  \cref{eq:bessel-spaces} to make sense when $s < 0$, we implicitely rely on the identification between $\S'(\R, E)$ and the topological antidual of $\S(\R, E)$; see \cite[Corollary 1.4.10 and Theorem 1.7.5, Appendix]{Ama19book}.





\para{The Laplace integral} Finally, let $u : (0, +\infty) \to E$ be a Bochner-measurable function that is integrable on $(0, \tau)$ for all (finite) $\tau > 0$. We define the Laplace transform $\hat{u}$ of $u$
by
\begin{equation}
\hat{u}(p) \triangleq \lim_{\tau \to +\infty} \int_0^\tau e^{-pt} u(t) \, \d t
\end{equation}
for any $p \in \Co$ such that the limit exists in $E$. The reader is referred to \cite[Part A, Section 1.4]{AreBat01book} for more details on the convergence of the Laplace integral. Denote by $\L$ the operator sending $u$ to $\hat{u}$.
 If $u \in L^1(0, + \infty; E)$, having extended $u$ with zero for negative time, we have
\begin{equation}
\hat{u}(\i \omega) = \L[u](\i \omega) =  \F[u](\omega), \quad \omega \in \R.
\end{equation}


\subsection{Semigroup-theoretic framework}
\label{sec:adm-semi}
\label{sec:pre-sta}

Let $A : \mathrm{dom}(A) \to X$ be the infinitesimal generator of a strongly continuous semigroup $\{S_t\}_{t \geq 0}$ on the Hilbert space $X$. In that follows, $\mu$ is a fixed element of the resolvent set $\rho(A)$ of $A$; if possible, it is convenient to choose $\mu = 0$.
 The domain $\dom(A)$ of $A$ is equipped with the norm $\|(\mu - A) \cdot \|_X$, which is equivalent to the graph norm, and the resulting Hilbert space is denoted by $X_1$. The first extrapolation space $X_{-1}$ is the completion of $X$ with respect to the norm $\| (\mu - A)^{-1}\cdot\|_{X}$ and is itself a Hilbert space; see, e.g., \cite[Section II.5.a]{EngNag00book}. We have the chain of continuous and dense embeddings 
\begin{equation}
X_{1} \hookrightarrow X \hookrightarrow X_{-1}.
\end{equation}
We use the same symbols to denote the unique extensions of $A$ in $\L(X, X_{-1})$ and $A^{-1}$ in $\L(X_{-1}, X)$.
Furthermore,
$\{S_t \}_{t \geq 0}$ uniquely extends to a strongly continuous semigroup on $X_{-1}$, which we likewise denote by the same symbol.
More generally, we consider the standard discrete Sobolev tower $\{X_n\}_{n\in \Z}$ (i.e., the extrapolated discrete power scale generated by $A$): if $n \in \N$, $X_n$ denotes $\dom(A^{n})$ equipped with the norm $\|(\mu - A)^n\cdot\|_X$, while $X_{-n}$ stands for the completion of $X$ with respect to the norm $\|(\mu - A)^{-n}\cdot\|_X$. The semigroup $\sg{S}$ and its generator $A$, after extension or restriction, act naturally on $\{X_n\}_{n\in \Z}$; we refer the reader to \cite[Section II.5.a]{EngNag00book} and the 
 diagrams therein for more details. In particular, for any $n \in \Z$, we have $X_n \hookrightarrow X_{n-1}$ with continuous and dense embedding, and $\mu - A$  is an isometric isomorphism from $X_n$ onto $X_{n-1}$.

In order to fill the gaps in $\{X_n\}_{n\in\N}$, we now introduce two continuous scales of intermediate spaces
 based on the functional calculus of $\mu -A$ and interpolation in Hilbert spaces, respectively. 
\begin{enumerate}[label=(\roman*)]
\item
The operator $\mu-A$ possesses fractional powers
 $(\mu -A)^{s}$, $s \in \R$; see for instance \cite[Section II.2.9]{Ama95book} or \cite[Section II.5.c]{EngNag00book}.
For $s \geq 0$,
 we let
\begin{equation}
X_s^\fra \triangleq \dom((\mu -A)^{s}),
\end{equation}
equipped with the norm $\|(\mu -A)^s \cdot \|_X$ (the superscript ``fr'' stands for ``fractional''). For $s < 0$, we define $X_s^\fra$ as the completion of $X$ with respect to the norm $\|(\mu-A)^{s} \cdot\|_X$.
 This produces Hilbert spaces. 
For all real numbers $s_1 \leq s_2$, we have $X_{s_2}^\fra \hookrightarrow X_{s_1}^\fra$ with continous and dense embedding, and
\begin{equation}
(\mu -A)^{s_2 - s_1}: X_{s_2}^\fra \to X_{s_1}^\fra
~\mbox{is an isometric isomorphism.}
\end{equation}

\item Let $s \in \R$ and write $s = n + \theta$ with $n \in \Z$ and $0 \leq \theta < 1$. We define $X_{s}^\inte$ by
\begin{equation}
X_s^{\mathrm{in}} \triangleq [X_{n+1}, X_n]_{1 - \theta}
\end{equation}
(here, ``in'' stands for ``interpolation''). Again, for all reals $s_1 \leq s_2$, we have $X_{s_2}^\inte \hookrightarrow X_{s_1}^\inte$ with continous and dense embedding. Interpolating between discrete values, we see that, for any $s \in \R$, $\mu - A$ is an isometric isomorphism between $X^\inte_{s +1}$ and $X^\inte_{s}$. This allows us to show, with the reiteration property, that
$\{X^\inte_s\}_{s \in \R}$ is an \emph{exact interpolation scale}, i.e., for $s_2 \geq s_1$ and $0 < \theta < 1$, having set $s \triangleq (1- \theta)s_1 + \theta s_2$, we have
\begin{equation}
X_s^\inte = [X_{s_2}^\inte, X_{s_1}^\inte]_{1 -\theta}
\end{equation}
with \emph{equality} of norms.
\end{enumerate}
\begin{rem}
A different choice of $\mu \in \rho(A)$ leads to the same spaces as above up to equivalence of norms.
\end{rem}

Let $X_s$ denote either $X_s^\fra$ or $X_s^\inte$. Just as in the discrete setting, the semigroup $\sg{S}$ behaves well on the continuous scale $\{X_s\}_{s \in \R}$.  Given $s \in \R$, after extension or restriction, $\sg{S}$ as a collection of operators in $\L(X_s)$ is well-defined, satisfies the semigroup property and is strongly continuous. Furthermore, the domain of its infinitesimal generator as a semigroup on $X_\theta$ is precisely $X_{\theta+1}$, that is,
\begin{equation}
\left\{ x \in X_s: \lim_{t \to 0^+} \frac{S_t x - x}{t} ~\mbox{exists in}~X_s \right\} = X_{s+1}.
\end{equation}
In  the case $X_s = X_s^\fra$, these facts readily follow from
commutativity  of the semigroup with any fractional power of $\mu -A$. When $X_s = X_s^\inte$, one can again use interpolation between discrete values; see also \cite[Section V.2]{Ama00} for a more detailed and  general study of semigroups in interpolation-extrapolation scales.
\begin{rem}
\label{rem:int-fra}
The spaces $X_s^\inte$ and $X_s^\fra$ need not coincide. However, they are equal (with equivalence of norms) when $A$ generates a contraction semigroup \cite[Corollary 4.30]{Lun09book}.  Another sufficient condition for $X_s^\inte$ and $X_s^\fra$ to coincide is that $-A$ possesses a bounded $H^\infty$-calculus \cite[Theorem 16.3]{Yag10book}. On the other hand,  \cite[Chapter 16, Section 1.5]{Yag10book} exhibits generators $A$ of analytic semigroups for which the spaces $X_{1/2}^\inte$ and $X_{1/2}^\fra$ are not equal.
\end{rem}



Finally, denote by $\sigma_0$ the growth bound of $\sg{S}$, that is, the infimum of all reals $\sigma$ such that there exists $K > 0$ with $\|S_t\|_{\L(X)} \leq K e^{\sigma t}$ for all $t \geq 0$.
The resolvent $p \mapsto (p - A)^{-1}$ is bounded in operator norm on closed half-planes of the form $\{ p \in \Co,~ \re p \geq \sigma \}$, where $\sigma > \sigma_0$.
It follows from the resolvent identity that, for each $\sigma > \sigma_0$,
\begin{equation}
\label{eq:est-res}
\|(\sigma + \i \omega - A)^{-1}\|_{\L(X, X_1)} =  O(|\omega|), \quad \omega \in \R, \quad |\omega| \to + \infty.
\end{equation}
 We can then deduce from \cref{eq:est-res} that, for $\sigma > \sigma_0$ and $0 < s < 1$,
\begin{equation}
\label{eq:est-res-theta}
\|(\sigma + \i \omega - A)^{-1}\|_{\L(X, X_s)} = O(|\omega|^s), \quad \omega \in \R, \quad |\omega| \to + \infty.
\end{equation}
In the case $X_s = X_s^\fra$, \cref{eq:est-res-theta} is obtained by using the moment inequality \cite[Proposition 6.6.4]{Haa06book}. If $X_s = X_s^\inte$, then the estimate can be derived by interpolating between the pairs $\{X, X\}$ and $\{X, X_1\}$. Note also that, on \emph{bounded} subsets of such half-planes, the resolvent is bounded in $\L(X_{s_2}, X_{s_1})$ for all $s_1 \geq s_2$. We shall make use of this fact without further comment and focus on high-frequency asymptotics.



\section{Main results for operator semigroups}
\label{sec:main-res}

\subsection{Admissibility of observation operators}


We consider the  Cauchy problem associated with the semigroup $\sg{S}$:
\begin{equation}
\label{eq:cauchy-pb-stat}
\dot{x} = Ax, \quad x(0) = x_0.
\end{equation}
Recall that \emph{classical} solutions $x$ to \cref{eq:cauchy-pb-stat} are those originating from initial data $x_0$ in 
$X_1$.
 Such solutions have regularity $x \in \C(\R^+, X_1) \cap \C^{1}(\R^+, X)$.

Let $Y$ be another Hilbert space  and consider a map $C \in \L(X_1, Y)$, which we will interpret as an $A$-bounded observation operator.
Having noted that $Cx$ is well-defined in $\C(\R^+, Y)$ for any solution to \cref{eq:cauchy-pb-stat} with initial data in $X_1$,
let us recall the definition of admissibility of $C$ as an observation operator for the semigroup $\sg{S}$.
\begin{defi}[Admissible observation operators]\label{def:admissible-observation} The  operator $C$ is \emph{admissible} if there exist $K, T > 0$ such that solutions $x$ to \cref{eq:cauchy-pb-stat} satisfy
\begin{equation}
\label{eq:std-admi}
\|Cx\|_{L^2(0, T; Y)} \leq K \|x_0\|_X, \quad x_0 \in X_1.
\end{equation}
\end{defi}
If \cref{eq:std-admi} holds for some $T > 0$ then it holds for all $ T > 0$, up to a change of constant $K$. In particular, if $C$ is admissible, for any $T > 0$, the map $x_0 \mapsto Cx$ uniquely extends as a bounded linear operator from $X$ into $L^2(0, T; Y)$. 

Our first  theorem 
shows the equivalence of three conditions generalising \cref{eq:std-admi} and  establishes a connection with frequency-domain growth of the operator-valued function $p \mapsto C(p - A)^{-1}$. We recall that $\sigma_0$ is the growth bound of $\sg{S}$.
\begin{theo}[Admissibility of observation operators] 
\label{theo:inter-group}Let $0 \leq \eta \leq 1$. 
Consider solutions $x$ to the Cauchy problem \cref{eq:cauchy-pb-stat} with initial data $x_0$.
The following  are equivalent:
\begin{enumerate}[label=(\roman*),series=inter-group]
  \item \label{it:adm-eta-fra} \emph{(Smoother data, fractional power.)} 
There exist $K, T > 0$ such that
  \begin{equation}
\|Cx\|_{L^2(0, T; Y)}\leq K \|x_0\|_{X^\fra_\eta}, \quad x_0 \in X_1;
  \end{equation}
  \item \label{it:adm-eta} \emph{(Smoother data, interpolation.)} There exist $K, T > 0$ such that
  \begin{equation}
\|Cx\|_{L^2(0, T; Y)}\leq K \|x_0\|_{X^\inte_\eta}, \quad x_0 \in X_1;
  \end{equation}
  \item \label{it:alt-adm} \emph{(Distributional outputs.)} There exist $K, T > 0$ such that
    \begin{equation}
  \|Cx\|_{H^{-\eta}(0, T; Y)} \leq K \|x_0\|_X, \quad x_0 \in X_1.
  \end{equation}
\end{enumerate}
Furthermore, if \cref{it:adm-eta-fra,it:adm-eta,it:alt-adm} are satisfied then:
\begin{enumerate}[resume*=inter-group]
  \item \label{it:FQ-cond} \emph{(Frequency-domain condition.)} For some (equivalently, all) $\sigma > \sigma_0$,
 \begin{equation} 
  \label{eq:FQ-cond-C}
  \|C(\sigma + \i \omega - A)^{-1}\|_{\L(X, Y)} =
  O(|\omega|^\eta),
  \quad \omega \in \R, \quad |\omega| \to + \infty.
  \end{equation}
\end{enumerate}
Finally, under the additional assumption that $\sg{S}$ is left-invertible,\footnote{
The semigroup $\sg{S}$ is \emph{left-invertible} (resp. \emph{right-invertible}) if for some (hence all) $t>  0$, the operator $S_t$ possesses a left inverse (resp. right inverse) in $\L(X)$.
} 
\cref{it:adm-eta-fra,it:adm-eta,it:alt-adm,it:FQ-cond} are all equivalent.
\end{theo}
\begin{rem}
If any of the conditions \cref{it:adm-eta-fra,it:adm-eta,it:alt-adm} holds for some $T > 0$, it must hold for all $T > 0$. This is straightforward for \cref{it:adm-eta-fra,it:adm-eta} and follows from our proof for \cref{it:alt-adm}.
\end{rem}
The case $\eta = 0$ corresponds to classical admissibility:  \cref{it:adm-eta-fra,it:adm-eta,it:alt-adm} are identical statements and \cref{it:FQ-cond} reads as uniform boundedness of $p \mapsto C(p - A)^{-1}$ on  vertical lines $\sigma + \i \R$. The fact that admissibility implies this boundedness property, and the converse implication in the left-invertible case, was already known; see, e.g., \cite[Theorem 4.3.7 and Corollary 5.2.4]{TucWei09book}. On the other hand, if $\eta = 1$ then the properties in \cref{theo:inter-group} are always true due to $A$-boundedness of $C$. This can be seen directly: \cref{it:adm-eta-fra,it:adm-eta} are trivially satisfied, one can check \cref{it:alt-adm} similarly as in \cref{lem:output-H-1} below and \cref{it:FQ-cond} 
follows from \cref{eq:est-res}, which is a direct consequence of the resolvent identity.
To our knowledge, the results of \cref{theo:inter-group} in the case $0 < \eta < 1$ are completely new. In the light of \cref{sec:pre-sta}, \cref{it:adm-eta-fra,it:adm-eta} amount to admissibility of $C$ for the semigroup $\sg{S}$ \emph{restricted} to the spaces (of ``smoother'' initial data) $X^\fra_\eta$ and $X^\inte_\eta$, respectively. We stress again that those spaces are different in general. \Cref{it:alt-adm} means that, even for solutions at the $X$-level, i.e., \emph{mild} solutions, the output $Cx$ is well-defined in a distributional sense in $H^{-\eta}$-spaces.
\begin{rem}
It is interesting to compare \cref{theo:inter-group} with \cite[Theorem 6.7.3]{Haa06book}, which, specialised to our setting, reads as follows: for $0 < \theta < 1$, the condition $\|C(\lambda -A)^{-1}\|_{\L(X, Y)} = O(\lambda^{\theta-1})$, $\lambda \in \R$, $\lambda \to + \infty$, is \emph{equivalent} to $C$ being bounded from the \emph{real} interpolation space $(X_1, X)_{1 - \theta, 1}$ into $Y$. Thus  real line resolvent growth characterises a certain form of relative boundedness of $C$ with respect to $A$. In contrast,
growth conditions on vertical lines such as \cref{eq:FQ-cond-C} need not imply additional boundedness properties of $C$.
\end{rem}

\begin{rem}
To some extent, one may shift the equivalent conditions of \cref{theo:input-group} up and down across the Sobolev scale: see for instance the proof of \cref{prop:frac} and the tools therein for the spaces $\{X^\fra_s\}_{s \in \R}$, and \cref{lem:reg-shift} below for $\{X^\inte_s\}_{s \in \R}$.
\end{rem}

Before moving on to the next section,
let us introduce some terminology, mostly for the sake of simplifying further statements.
\begin{defi}[$\eta$-admissible observation operators]
\label{def:eta-output} Let $0 \leq \eta \leq 1$. We say that the operator $C$ is \emph{$\eta$-admissible}  if it satisfies any of the conditions \cref{it:adm-eta-fra,it:adm-eta,it:alt-adm} of \cref{theo:inter-group}.
\end{defi}



\subsection{Admissibility of control operators
}
\label{sec:adm-input}

Let $U$  be another Hilbert space and $B \in \L(U,X_{-1})$. We will say that $B$ is a $A$-bounded control operator.
Consider  the inhomogeneous Cauchy problem
\begin{equation}
\label{eq:cauchy-u}
\dot{x} = Ax + Bu, \quad x(0) = 0.
\end{equation}
Given $T > 0$,
input functions $u$ in (say) $L^2(0, T; U)$ produce solutions $x$ in $\C([0, T], X_{-1})$, with
\begin{equation}
\label{eq:conv}
x(t) =  \int_0^t S_{t - s}Bu(s) \, \d s, \quad  0 \leq t \leq T.
\end{equation}
Let $\sigma > \sigma_0$.
If $u$ has further regularity, say $u \in H^1(0, T; U)$, integrating by parts in \cref{eq:conv} yields
\begin{equation}
\label{eq:conv-ipp}
x(t) =
e^{-\sigma t}(A - \sigma)^{-1}(S_t Bu(0)
- Bu(t)) + \int_0^t S_{t-s}(A^{-1}- \sigma)^{-1}B(\dot{u}(s) - \sigma u(s)) \, \d s, \quad 0 \leq t \leq T,
\end{equation}
which shows that for such inputs, $x$ belongs to $\C(\R^+, X)$. With that in mind, we recall the definition of admissibility of $B$ as a control operator for the semigroup $\sg{S}$.
\begin{defi}[Admissible control operators] The operator $B$ is \emph{admissible} if there exist $K, T > 0$ such that solutions $x$ to \cref{eq:cauchy-u} satisfy
\begin{equation}
\label{eq:B-adm-est}
\|x(T)\|_X \leq K \|u\|_{L^2(0, T; U)}, \quad u \in H^1(0, T; U).
\end{equation}
\end{defi}
\begin{rem}
\label{rem:quant-adm}
Equivalently, admissibility of $B$ can be defined as follows: for all inputs $u \in L^2(0, T; U)$, the final state $x(T)$, \emph{a priori} defined in the extrapolation space $X_{-1}$, belongs to the original state space $X$. Indeed, the quantified estimate \cref{eq:B-adm-est} then follows from the closed graph theorem.
\end{rem}
Again, the particular choice of $T$ does not matter, and when admissibility holds, the map $u \mapsto x(T)$ uniquely extends as a bounded linear operator from $L^2(0, T; U)$ into $X$. Furthermore, a simple translation argument shows that solutions $x$ generated by inputs $u \in L^2(0, T; U)$ in fact satisfy $x \in \C([0, T], X)$.

Our next result is the control operator counterpart to \cref{theo:inter-group} and, likewise, sheds light on situations intermediate between admissibility and mere $A$-boundedness.
\begin{theo}[Admissibility of control operators] 
\label{theo:input-group}
Let $0 \leq \eta \leq 1$. Consider solutions $x$ to the inhomogeneous Cauchy problem \cref{eq:cauchy-u} with controls $u$.
The following are equivalent:
\begin{enumerate}[label=(\roman*'), series=inter-group-input]
\item
\label{it:input-frac}
\emph{(Rougher state, fractional power.)} There exist $K, T > 0$ such that
\begin{equation}
\|x(T)\|_{X^\fra_{-\eta}} \leq K \|u\|_{L^2(0, T; U)}, \quad u \in H^1(0, T; U);
\end{equation}
\item
\label{it:input-int} 
\emph{(Rougher state, interpolation.)} There exist $K, T > 0$ such that
\begin{equation}
\|x(T)\|_{X^\inte_{-\eta}} \leq K \|u\|_{L^2(0, T; U)}, \quad u \in H^1(0, T; U);
\end{equation}
\item
\label{it:smooth-input}
\emph{(Smoother inputs.)} There exist $K, T > 0$ such that
\begin{equation}
  \|x(T)\|_X \leq K \|u\|_{H^\eta(0, T; U)}, \quad u \in H^1(0, T; U).
\end{equation}
\end{enumerate}
Furthermore, if \cref{it:input-frac,it:input-int,it:smooth-input} are satisfied then:
\begin{enumerate}[resume=inter-group-input,label=(\roman*')]
  \item
  \label{it:input-FQ} 
  \emph{(Frequency-domain condition.)} For some (equivalently, all) $\sigma > \sigma_0$,
 \begin{equation} 
  \label{eq:input-FQ}
  \|(\sigma + \i \omega - A)^{-1}B\|_{\L(U, X)} =
  O(|\omega|^\eta),
  \quad \omega \in \R, \quad |\omega| \to + \infty.
  \end{equation}
\end{enumerate}
Finally, under the additional assumption that $\sg{S}$ is right-invertible, \cref{it:input-frac,it:input-int,it:smooth-input,it:input-FQ} are all equivalent.
\end{theo}
We make several comments analogous to those following \cref{theo:inter-group}.
The case $\eta = 0$ is already covered in the literature; see, e.g., \cite[Proposition 4.4.6 and Theorem 5.2.2]{TucWei09book}. For $\eta = 1$, the conditions of \cref{theo:input-group} are always satisfied by $A$-boundedness of $B$: \cref{it:input-frac,it:input-int} are immediate, \cref{it:smooth-input} follows from  \cref{eq:conv-ipp} and \cref{it:input-FQ} is obtained by means of resolvent identity.
With a closed graph argument as in \cref{rem:quant-adm}, we may reformulate \cref{it:input-frac,it:input-int} as the range of the input-to-state map $u \mapsto x(T)$ applied to $L^2(0, T; U)$ being contained in the negative order spaces $X_{-\eta}^\fra$ and $X_{-\eta}^\inte$, respectively. As a matter of fact, \cref{theo:input-group} reveals that, if this range is contained in $X^\fra_{-\eta}$ or $X^\inte_{-\eta}$ for some $0 < \eta < 1$, it must in fact be contained in the (in general smaller by \cref{rem:int-fra}) subspace
$X^\fra_{-\eta} \cap X^\inte_{-\eta}$.
On the other hand, \cref{it:smooth-input} means that inputs $\eta$-smoother in time on the (positive) Sobolev scale produce $X$-valued (i.e., finite-energy) solutions.

\begin{defi}[$\eta$-admissible control operators] 
\label{def:eta-input} Let $0 \leq \eta \leq 1$. We say that $B$ is $\eta$-admissible  if it satisfies any of the conditions \cref{it:input-frac,it:input-int,it:smooth-input} of \cref{theo:input-group}.
\end{defi}


\subsection{Transfer function asymptotics}
\label{sec:TF-IO}

We now turn our attention to the properties of the input-output map $u \mapsto Cx$, where $x$ solves the inhomogeneous Cauchy problem with zero initial data \cref{eq:cauchy-u},
in relation to the operator-valued \emph{transfer function}
\begin{equation}
\label{eq:transfer-function}
p \mapsto C(p - A)^{-1}B.
\end{equation}
The previous standing assumptions remain in place. In contrast to our input and output results, we have to resort to an additional 
 hypothesis in order to make sense of \cref{eq:transfer-function}.
Namely, we assume that there exists an intermediate Hilbert space $Z$ satisfying $X_1 \hookrightarrow Z \hookrightarrow X$ with continuous embeddings, and  that
\begin{equation}
\label{eq:hyp-Z}
C \in \L(Z, Y), \qquad \ran((p - A)^{-1}B) \subset Z ~\mbox{for some (hence all)}~p \in \rho(A).
\end{equation}
Since $Z$ is continuously embedded into $X$ and $(p - A)^{-1}B \in \L(U, X)$, a closed graph argument yields $(p - A)^{-1}B \in \L(U, Z)$ for any $p \in \rho(A)$. In particular, this allows us to define the transfer function \cref{eq:transfer-function} as a $\L(U, Y)$-valued function on the open right-half plane. Note that $Z$ need not be invariant under the semigroup.
\begin{rem}
\label{rem:solution-space}
A possible choice of $Z$ is given by the \emph{solution space} $X_1^B \triangleq X_1 + \ran((\mu - A)^{-1}B)$ equipped with the norm
\begin{equation}
 \|x\|^2_{X_1^B} \triangleq \inf \{ \|w\|^2_{X_1} + \|u\|_U^2 : w \in X_1,~ u \in Y,~ x = w + (\mu - A)^{-1}Bu \}.
\end{equation}
It is a Hilbert space \cite[Lemma 4.3.12]{Sta05book} that satisfies our requirements for $Z$. In fact, any suitable space $Z$ must contain $X_1^B$.
\end{rem}
Let $T > 0$; further integrating by parts in \cref{eq:conv-ipp} and using \cref{eq:hyp-Z} show that, when $u \in \D(0, T; U)$, the corresponding solution $x$ to \cref{eq:cauchy-u} satisfies $x \in \C([0, T], Z)$.\footnote{
See the proof of \cref{lem:u-cut-Z} below for details.}
 For such solutions, $Cx$ is defined in a classical, pointwise sense. 

\begin{theo}[Transfer function estimate]
\label{theo:transfer-function}
 Let $0 \leq \eta_1, \eta_2 \leq 1$. Assume that the control and observation operators $B$ and $C$ are $\eta_1$- and $\eta_2$-admissible, respectively.
Consider solutions $x$ to the inhomogeneous Cauchy problem
\cref{eq:cauchy-u} with controls $u$.
The following are equivalent:
\begin{enumerate}[label=(\alph*)]
  \item
  \label{it:L2-H-eta}
  \emph{($L^2$-inputs produce $H^{-\eta}$-outputs.)} There exist $K, T > 0$ such that
  \begin{equation}
  \label{eq:L2-H-eta}
  \|Cx\|_{H^{-\eta_1 - \eta_2}(0, T; Y)} \leq K \|u\|_{L^2(0, T; U)}, \quad u \in \D(0, T; U).
  \end{equation}
  \item 
  \label{it:FQ-IO}
  \emph{(Frequency-domain condition.)} For some (equivalently, all) $\sigma > \sigma_0$,
  \begin{equation}
  \label{eq:FQ-IO}
  \|C(\sigma + \i \omega - A)^{-1}B\|_{\L(U, Y)} = O(|\omega|^{\eta_1 + \eta_2}), \quad \omega \in \R, \quad  |\omega| \to + \infty.
  \end{equation}
\end{enumerate}
 If \cref{eq:FQ-IO} holds for some $\sigma > \sigma_0$, then  \cref{it:L2-H-eta} holds, even without assuming  $\eta_1$- and $\eta_2$-admissibility of $B$ and $C$.
\end{theo}
Note that equivalence between the time-domain and frequency-domain conditions does not require additional invertibility properties of the semigroup $\sg{S}$. This is in contrast to \cref{theo:inter-group,theo:input-group}. Now, solutions $x$ to \cref{eq:cauchy-u} are exactly convolution products with the semigroup $\sg{S}$, and \cref{it:FQ-IO} can be seen as a related Fourier multiplier property; hence, the connection with $H^{-\eta}$-spaces is natural and the result might seem straightforward. The challenging part of the proof actually consists in extending the finite-time property \cref{it:L2-H-eta} into a suitable Bessel potential space estimate; see \cref{prop:h-eta-est-io} in the proof below. In that regard, we wish to point out that the idea of viewing weighted transfer functions as Laplace multipliers also appears in \cite{GuiOpm23,GuiLog24}.

 Similarly as before, one can pass to the limit in \cref{eq:L2-H-eta} and see that, assuming \cref{it:L2-H-eta}, all solutions $x$ to \cref{eq:cauchy-u} produced by inputs $u \in L^2(0, T; U)$ possess an output $Cx$ that is well-defined in a distributional sense in $H^{-\eta_1 - \eta_2}(0, T; Y)$. Also, in the case that $\eta_1 = \eta_2 = 0$, $B$ and $C$ are admissible control and observation operators, and \cref{it:L2-H-eta} means that, in addition, the system defined by the semigroup generator $A$ and the control and observation operators $B$ and $C$ is \emph{well-posed} in the sense of Weiss; see, e.g., 
\cite{TucWei14survey}.
In the other extremal case, namely $\eta_1 = \eta_2 = 1$, \cref{it:L2-H-eta,it:FQ-IO} are again automatically satisfied.


\begin{rem}
Existence of a space $Z$ satisfying \cref{eq:hyp-Z} allows a straightforward definition of the transfer function \cref{eq:transfer-function}, at the cost of some generality. Perhaps more so than \cref{theo:inter-group,theo:input-group}, \cref{theo:transfer-function} is mostly application-driven and such a hypothesis is rarely an issue in the context of partial differential equations. The reader interested in  system-theoretic and algebraic aspects of transfer function theory is referred to \cite{Sta05book,TucWei14survey} and the references therein.
\end{rem}




\section{Proofs of the main results
}

\label{sec:proof}


\subsection{Proof of \cref{theo:inter-group}}

\label{sec:proof-output}

This section is devoted to the proof of \cref{theo:inter-group}. 
\Cref{theo:input-group} will then follow by duality arguments.

\para{Notation}
In this subsection, the variable $x$ will always denote the (unique) solution to the abstract Cauchy problem
\begin{equation}
\label{eq:cauchy-x-A}
\dot{x} = Ax, \quad x(0) = x_0, \quad x_0 \in X,
\end{equation}
i.e., $x(t) = S_t x_0$ for $t \geq 0$. If $x_0 \in X_1$, then $x$ is a {classical} solution to \cref{eq:cauchy-x-A} and satisfies $x \in \C(\R^+, X_1) \cap \C^1(\R^+, X)$. We shall frequently assume that $x_0 \in X_1$, derive estimates that are uniform with respect to the $X$-norm, and then pass to the limit for general data $x_0 \in X$ by density.  Throughout the proofs, we will repeatedly make use of exponential shifting, so it is convenient to introduce the following notation: given any vector-valued function $f$ defined on $\R^+$, we let
\begin{equation}
\tilde{f}(t) \triangleq e^{-t} f(t), \quad t \geq 0.
\end{equation}
Furthermore, unless specified otherwise, vector-valued functions $f$ defined on $\R^+$ are extended to functions on $\R$ by setting $f(t) = 0$ for $t < 0$. In particular, if $\tilde{f} \in L^1\hl{E}$ we have the identity
\begin{equation}
\F [\tilde{f}](\omega) = \L[\tilde{f}](\i \omega) =  \hat{f}(1 + \i \omega), \quad \omega \in \R,
\end{equation}
where we recall that the hat denotes the Laplace transform.

\para{Simplification} We observe that, given $\sigma \in \R$, replacing $x$ by $e^{-\sigma \cdot} x$ (which amounts to replacing $A$ by $A - \sigma$) does not change whether or not conditions \cref{it:adm-eta-fra,it:adm-eta,it:alt-adm} hold.
Furthermore, by the resolvent identity, for any $\sigma, \sigma' > \sigma_0$, we have
\begin{equation}
\label{eq:res-id-sigma}
C(\sigma + \i \omega - A)^{-1} = C(\sigma' + \i \omega - A)^{-1} + (\sigma' - \sigma) C(\sigma + \i \omega - A)^{-1}(\sigma' + \i \omega - A)^{-1}, \quad \omega \in \R.
\end{equation}
Since $\sup_{\omega \in \R} \|(\sigma' + \i \omega - A)^{-1}\|_{\L(X)} < + \infty$, \cref{eq:res-id-sigma} proves the ``for some (equivalently, all)'' part in \cref{it:FQ-cond}. Without loss of generality, we may assume that the semigroup $\sg{S}$ is uniformly bounded and that $0 \in \rho(A)$. In particular, we choose $\mu = 0$ when defining the norms of \cref{sec:adm-semi}. Note also that $\sg{S}$ remains uniformly bounded as a semigroup on any of the spaces $X_s^\inte$ or $X_s^\fra$, $s \in \R$.

\subsubsection{Basic lemmas}

We continue by establishing three auxiliary results. First, the following lemma shows that we can always define $Cx$ in a distributional sense even for initial data in $X$.
\begin{lemma}
\label{lem:output-H-1}
The map $x_0 \mapsto C\tilde{x}$, defined from $X_1$ into $\C(\R^+, Y)$, uniquely extends to a bounded linear operator from $X$ into $H^{-1}\hl{Y}$.
\end{lemma}
\begin{proof}
We recall that $X_1$ is dense in $X$ and let $x_0 \in X_1$. Then $x$ is a classical solution to \cref{eq:cauchy-x-A} and, by $A$-boundedness of $C$, $C\tilde{x}$ is defined in a classical, pointwise sense, with $C\tilde{x} \in \C(\R^+, Y)$. For $t \geq 0$, let 
$
 z(t) \triangleq \int_0^t x(s) \, \d s.
$
 Then, $z(t) = S_t A^{-1}x_0 - A^{-1}x_0$ for $t \geq 0$, which gives
\begin{equation}
\label{eq:first-est-z}
\|C\tilde{z}\|_{L^2(0, +\infty; Y)} \leq K \|A^{-1}x_0\|_{X_1} = K \|x_0\|_X.
\end{equation}
Furthermore,
 $\tilde{z} \in \C^1(\R^+, X_1)$  with $(\d / \d t)\tilde{z} = - \tilde{z} + \tilde{x}$. Integration by parts yields
\begin{equation}
\label{eq:ipp-z-x-varphi}
\langle C\tilde{x}, \varphi \rangle_{L^2\hl{Y}} = \int_0^{+\infty} \langle C\tilde{z}(t), \varphi(t) \rangle_Y \, \d t + \int_0^{+\infty} \langle C\tilde{z}(t), (\d / \d t){\varphi}(t) \rangle_Y \, \d t, \quad \varphi \in \D\hl{Y}.
\end{equation}
We use \cref{eq:ipp-z-x-varphi} and our previous $L^2$-estimate \cref{eq:first-est-z}  to deduce that
\begin{equation}
|\langle C\tilde{x}, \varphi \rangle_{L^2\hl{Y}}| \leq K \|C\tilde{z}\|_{L^2(0, +\infty; Y)} \|\varphi\|_{H^1(0, +\infty; Y)} \leq K' \|x_0\|_X \|\varphi\|_{H^1(0, +\infty; Y)}
\end{equation}
for arbitrary $\varphi \in \D\hl{U}$. Together with \cref{eq:dual-norm}, this gives $\|C\tilde{x}\|_{H^{-1}\hl{Y}} \leq K \|x_0\|_X$, as required.
\end{proof}


Next, 
we demonstrate how we may take advantage of the semigroup properties in order to extend the $H^{-\eta}$-norm estimate \cref{it:alt-adm} from the finite time interval $(0, T)$ to the half-line $(0, +\infty)$.
\begin{lemma}
\label{lem:X-to-Cx}
Let $0 \leq \eta \leq 1$.
Supposing that \cref{it:alt-adm} holds, the map $x_0 \mapsto C\tilde{x}$ is in fact continuous from $X$ into $H^{-\eta}(0, +\infty; Y)$.
\end{lemma}
\begin{proof}
Let us fix $T > 0$ as in \cref{it:alt-adm}. Let $x_0 \in X_1$.
By \cref{lem:localize-shift},
 we have 
\begin{equation}
\label{eq:Cx-sum-sum}
\| C\tilde{x}\|^2_{H^{-\eta}\hl{Y}} \leq K \sum_{n=0}^{+\infty} \| C\tilde{x}(nT + \cdot) \|^2_{H^{-\eta}(0, T; Y)} + \|C\tilde{x}((n+1/2)T + \cdot)\|^2_{H^{-\eta}(0, T; Y)}.
\end{equation}
 Multiplication by $e^{-\cdot}$ is continuous on $H^{-\eta}(0, T; Y)$. Thus,
\begin{equation}
\|C\tilde{x}(nT + \cdot)\|_{H^{-\eta}(0, T; Y)} = e^{-nT} \|Cx(nT + \cdot)\|_{H^{-\eta}(0, T; Y)} \leq K e^{-nT} \|Cx\|_{H^{-\eta}(0, T; Y)}, \quad n \in \N.
\end{equation}
Now, for all $n \in \N$ and $0 \leq t \leq T$, $Cx(nT + t) = CS_t x(nT)$. Therefore, using the hypothesis \cref{it:adm-eta} 
and boundedness of the semigroup $\sg{S}$ yields
\begin{equation}
\label{eq:est-Cx}
\|Cx(nT + \cdot)\|_{H^{-\eta}(0, T; Y)} \leq K \|x(nT)\|_X \leq K' \|x_0\|_X, \quad n \in \N.
\end{equation}
It follows that
\begin{equation}
\sum_{n=0}^{+\infty} \| C\tilde{x}(nT + \cdot) \|^2_{H^{-\eta}(0, T; Y)} \leq K \sum_{n=0}^{+\infty} e^{-2nT} \|x_0\|^2_X \leq K' \|x_0\|^2_X.
\end{equation}
We complete the proof by estimating the second term in \cref{eq:Cx-sum-sum} in a similar way.
\end{proof}



Finally, under condition \cref{it:adm-eta-fra} (resp.\ \cref{it:adm-eta}) we derive some other half-line estimates for data in $X_\eta^\fra$ (resp.\ $X_\eta^\inte$) as well as preliminary resolvent bounds.
\begin{lemma}
\label{lem:adm-bounded-CR}
Let $0 \leq \eta \leq 1$ and let $X_\eta$ denote $X_\eta^\fra$ (resp. $X_\eta^\inte$).
Assume that  \cref{it:adm-eta-fra} (resp. \cref{it:adm-eta}) holds. Then, 
\begin{equation}
\label{eq:inf-adm-eta}
\|C\tilde{x}\|_{L^2\hl{Y}} \leq K \|x_0\|_{X_\eta}, \quad x_0 \in X_1.
\end{equation}
Furthermore,
\begin{equation}
\label{eq:adm-bounded-CR}
\|C(1 + \i \omega - A)^{-1}\|_{\L(X_\eta, Y)} = O(1), \quad \omega \in \R, \quad |\omega| \to + \infty.
\end{equation}
\end{lemma}
\begin{proof}
Let $x_0 \in X_1$ and $\omega \in \R$. As $x$ is a classical solution to \cref{eq:cauchy-x-A} and $\sg{S}$ is uniformly bounded, $x$ remains bounded in $X_1$. In particular, $x$ is Laplace transformable and for all complex $p$ with $\re p > 0$, $\hat{x}(p)$ exists as an absolutely convergent $X_1$-valued integral.
Therefore,
\begin{equation}
C(2 + \i \omega - A)^{-1}x_0 = \int_0^{+\infty} e^{-(2 + \i \omega)t} Cx(t) \, \d t = \int_0^{+\infty} e^{-(1 + \i \omega)t} C\tilde{x}(t) \, \d t,
\end{equation}
with integrals that are absolutely convergent in $Y$. Using the Cauchy--Schwarz inequality, we then obtain
\begin{equation}
\label{eq:CS-shift}
\|C(2 + \i \omega - A)^{-1}x_0\|_Y^2 \leq \left ( \int_0^{+\infty} e^{-2t} \, \d t \right ) \left ( \int_0^{+\infty} \|C\tilde{x}(t)\|^2_Y \, \d t \right ) = \frac{1}{2} \int_0^{+\infty} \|C\tilde{x}(t)\|^2_Y \, \d t.
\end{equation}
On the other hand,
\begin{equation}
\label{eq:sum-n-Cx}
 \int_0^{+\infty} \| C\tilde{x}(t)\|^2_Y \, \d t = \sum_{n =0}^{+\infty} e^{-2nT} \int_0^T e^{-2t}\|Cx(nT +t)\|^2_Y \, \d t \leq \sum_{n =0}^{+\infty} e^{-2nT} \int_0^T \|Cx(nT +t)\|^2_Y \, \d t.
\end{equation}
Using the hypothesis \cref{it:adm-eta-fra} or \cref{it:adm-eta} together with the property that the operators $S_t$
are also uniformly bounded in $\L(X_\eta)$-norm, we get
\begin{equation}
\label{eq:est-CR-n}
 \int_0^T \|Cx(nT + t)\|^2_Y \, \d t \leq K \|x(nT)\|^2_{X_\eta} = K \|S_{nT}x_0\|^2_{X_\eta} \leq K' \|x_0\|_{X_\eta}, \quad n \in \N.
\end{equation}
Substituting \cref{eq:est-CR-n} into the sum \cref{eq:sum-n-Cx}, which is made convergent by the factors $e^{-2nT}$, leads to the first desired estimate \cref{eq:inf-adm-eta}. By plugging \cref{eq:inf-adm-eta} into \cref{eq:CS-shift}, we then obtain
\begin{equation}
\label{eq:est-CR-X1}
\|C(1 + \i \omega - A)^{-1}x_0\|_Y \leq
K \|x_0\|_{X_\eta}.
\end{equation}
This holds for arbitrary $x_0 \in X_1$ and $\omega \in \R$. Because $X_1$ is dense in $X_\eta$, it follows that
\begin{equation}
\label{eq:adm-bounded-CR-2}
\sup_{\omega \in \R} \|C(2 + \i \omega - A)^{-1}\|_{\L(X_\eta, Y)} < + \infty.
\end{equation}
By using that
$
C(1 + \i \omega - A)^{-1} = C(2 + \i \omega - A)^{-1} + C(2 + \i \omega -A)^{-1} (1 + \i \omega - A)^{-1}
$ for $\omega \in \R$,
and that $\sup_{\omega \in \R} \|(1 + \i \omega - A)^{-1}\|_{\L(X_\eta)} < + \infty$, we can deduce \cref{eq:adm-bounded-CR} from \cref{eq:adm-bounded-CR-2} and complete the proof.
\end{proof}



\subsubsection{The case of fractional powers of the generator}

We shall now prove the part of \cref{theo:inter-group} involving fractional powers of the semigroup generator. Our approach is based on the technique of resolvent growth cancellation; see, e.g., \cite[Lemma 3.2]{LatShv01} or the more general \cite[Theorem 5.5]{BatChi16}. Given $0 \leq \eta \leq 1$, we introduce  the following intermediate condition:
\begin{enumerate}[resume*=inter-group]
  \item 
  \label{it:bounded-fra} (Alternative frequency-domain condition, fractional powers.)
   \begin{equation}
  \label{eq:bounded-fra}
    \|C(1 + \i \omega - A)^{-1}\|_{\L(X^\fra_\eta,Y)} = O(1), \quad \omega \in \R, \quad |\omega| \to + \infty.
\end{equation}
\end{enumerate} 
We already know from \cref{lem:adm-bounded-CR} that \cref{it:adm-eta-fra} implies \cref{it:bounded-fra}. The other implications  we need are collected in the following proposition.

\begin{prop}[Fractional powers]
\label{prop:frac}
Let $0 < \eta \leq 1$. Then,
\begin{subequations}
\begin{align}
&\cref{it:adm-eta-fra} \implies \cref{it:bounded-fra}, \quad \cref{it:FQ-cond} \iff \cref{it:bounded-fra}, \quad \cref{it:alt-adm} \implies \cref{it:adm-eta-fra} &&\mbox{in general;}\\
&\cref{it:bounded-fra} \implies \cref{it:adm-eta-fra} &&\mbox{under the left-invertibility assumption.}
\end{align}
\end{subequations}
\end{prop}



\begin{proof}
 We start by  introducing some notation and parameters. Fix some real number $\mu_0 \geq 4$. Let $0 <\phi < \pi/6$ and let $\gamma$ be the piecewise smooth simple curve running from $\infty e^{-\i \phi}$ to $\infty e^{\i \phi}$ and made up of the union of the rays $\{- 1 + te^{\i \phi}, t \geq 0 \}$ and $\{- 1 - te^{\i \phi}, t \geq 0 \}$. Because $A - \mu_0$ is a semigroup generator and $\{ p \in \Co : \re p \geq 2 \} \subset\rho(A - \mu_0)$, one can always find $\phi$ for which there exists $K > 0$ such that, denoting by $\Sigma_\phi$ the sector generated by $\gamma$, $\Sigma_\phi \subset \rho(A - \mu_0)$ and
\begin{equation}
\label{eq:sectorial}
\|(\mu_0 + p - A)^{-1}\|_{\L(X)} \leq \frac{K}{1 + |p|}, \quad p \in \Sigma_\phi.
\end{equation}
The operator $(\mu_0 -A)^{-\eta}$ admits the integral representation 
\begin{equation}
(\mu_0 -A)^{-\eta}x_0 = \frac{1}{2\pi \i} \int_\gamma (-\mu)^{-\eta} (\mu_0 + \mu - A)^{-1}x_0 \, \d \mu, \quad x_0 \in X.
\end{equation}
The integral converges absolutely in $X$ due to \cref{eq:sectorial}.
Just as in the proof of \cite[Lemma 3.2]{LatShv01}, as a consequence of the resolvent identity we see  that
\begin{multline}
\label{eq:formula-res-A-eta}
(p - A)^{-1}(\mu_0-A)^{-\eta}x_0 = \frac{1}{2\pi \i} \int_\gamma (-\mu)^{-\eta}(p - A)^{-1}(\mu_0 + \mu - A)^{-1}x_0 \, \d \mu
 \\ = \frac{1}{2 \pi \i} \left (\int_\gamma \frac{(- \mu)^{-\eta}}{\mu_0 + \mu - p} \, \d \mu \right ) (p - A)^{-1}x_0  - \frac{1}{2 \pi \i} \int_\gamma \frac{(- \mu)^{-\eta}}{\mu_0 + \mu - p} (\mu_0 + \mu - A)^{-1}x_0 \, \d \mu
\end{multline}
for all $p \in \Co$ satisfying $p - \mu_0 \not \in \gamma$ and $\re p > 0$, and all $x_0 \in X$.
Let
\begin{equation}
I_p \triangleq \int_\gamma \frac{(- \mu)^{-\eta}}{\mu_0 + \mu - p} \, \d \mu, \quad p \not \in \gamma.
\end{equation}
We observe that
$\re (\mu_0 + \mu - 1 - \i \omega) \geq 2$ for all $\omega \in \R$ and $\mu \in \gamma$.
As a consequence, $I_{1 + \i \omega}$ is defined for any $\omega \in \R$. The following estimate, proved in \cite{LatShv01}, will be instrumental in the proof: there exist positive constants $K_0$ and $\omega_0$ such that
\begin{equation}
\label{eq:est-I}
\frac{1}{K_0|\omega|^\eta} \leq |I_{1 + \i \omega}| \leq \frac{K_0}{|\omega|^\eta}, \quad |\omega| \geq \omega_0.
\end{equation}
We shall complement \cref{eq:est-I} by proving also that $|I_{1 + \i \omega}|$ is bounded uniformly in $\omega \in \R$. First, we note that
\begin{equation}
|\mu_0 + \mu - 1 - \i \omega | \geq \re (\mu_0 + \mu) - 1 = \re(\mu + 1) + \mu_0 -2, \quad \omega \in \R, \quad \mu \in \gamma.
\end{equation}
But by definition of $\gamma$, $\re(\mu + 1) = \cos(\phi) |\mu + 1|$. Recalling that $\mu_0 \geq 4$, we find that
\begin{equation}
\label{eq:est-big-mu}
\frac{|\mu|^{-\eta}}{|\mu_0 + \mu - 1 - \i \omega |} \leq \frac{|\mu|^{-\eta}}{\cos(\phi)|\mu + 1| + 2}, \quad \omega \in \R, \quad \mu \in \gamma, \quad \re \mu \geq 0.
\end{equation}
On the other hand, we also have
\begin{equation}
\label{eq:est-small-mu}
\frac{|\mu|^{-\eta}}{|\mu_0 + \mu - 1 - \i \omega |}  \leq \frac{1}{2}, \quad \omega \in \R, \quad \mu \in \gamma, \quad -1 \leq \re \mu < 0.
\end{equation}
Combining \cref{eq:est-big-mu,eq:est-small-mu} finally yields a $K_1 > 0$ such that
\begin{equation}
\label{eq:est-I-global}
|I_{1+ \i \omega}| \leq K_1, \quad \omega \in \R.
\end{equation}
After these preparations, we are ready to prove the various implications in \cref{prop:frac}.
 As mentioned above, $\cref{it:adm-eta-fra} \implies \cref{it:bounded-fra}$ already follows from \cref{lem:adm-bounded-CR}.

{$\cref{it:FQ-cond} \iff \cref{it:bounded-fra}$:} First, we can rewrite  \cref{it:bounded-fra} 
as follows: for $\omega \in \R$, $\|C(1 + \i \omega - A)^{-1}(-A)^{-\eta}\|_{\L(X, Y)} = O(1)$ as $|\omega| \to +\infty$.
We then use the resolvent identity to obtain
\begin{subequations}
\begin{align}
&C(1+ \i \omega - A)^{-1} = CA^{-1} + (1 + \i \omega)CA^{-1}(1 + \i \omega - A)^{-1}, \quad \omega \in \R, \\
\label{eq:res-frac-eta}
&C(1+ \i \omega - A)^{-1}(-A)^{-\eta} = CA^{-1}(-A)^{-\eta} + (1 + \i \omega)CA^{-1}(1 + \i \omega - A)^{-1}(-A)^{-\eta}, \quad \omega \in \R.
\end{align}
\end{subequations}
Recalling our assumption of $A$-boundedness of $C$, we see that the desired result amounts to equivalence between, on one hand,
\begin{equation}
\label{eq:frac-shift-eta}
\|CA^{-1}(1 + \i \omega - A)^{-1}\|_{\L(X, Y)} = O(|\omega|^{-1+\eta}), \quad\omega \in \R, \quad |\omega| \to + \infty,
\end{equation}
and, on the other hand,
\begin{equation}
\label{eq:frac-shift-weight}
\|CA^{-1}(1 + \i \omega - A)^{-1}(-A)^{-\eta}\|_{\L(X, Y)} = O(|\omega|^{-1}), \quad\omega \in \R, \quad |\omega| \to + \infty.
\end{equation}
Furthermore, since $(-A)^{\eta}(\mu_0 - A)^{-\eta}$ is an isomorphism on $X$, 
 we may replace $(-A)^{-\eta}$ by $(\mu_0 - A)^{-\eta}$ in \cref{eq:res-frac-eta,eq:frac-shift-weight}.
Let $\omega \in \R$ be such that $|\omega| \geq \omega_0$.
By applying $CA^{-1} \in \L(X, Y)$ to \cref{eq:formula-res-A-eta}, we obtain, for all $x_0 \in X$,
\begin{multline}
\label{eq:func-calc}
CA^{-1}(1 + \i \omega - A)^{-1}(\mu_0 -A)^{-\eta}x_0 \\ = \frac{I_{1 + \i \omega}}{2\pi \i} CA^{-1}(1 + \i \omega - A)^{-1}x_0  - \frac{1}{2\pi\i} CA^{-1} \int_\gamma \frac{(- \mu)^{-\eta}}{\mu_0 + \mu - 1 - \i \omega} (\mu_0 + \mu - A)^{-1}x_0 \, \d \mu.
\end{multline}
The second term on the right-hand side of \cref{eq:func-calc} is dealt with by making use of the sectorial growth condition \cref{eq:sectorial} together with the uniform boundedness \cref{eq:est-small-mu} of $I_{1+\i\omega}$:
\begin{multline}
\label{eq:est-LOT}
\left \| CA^{-1} \int_\gamma \frac{(- \mu)^{-\eta}}{\mu_0 + \mu - 1 - \i \omega} (\mu_0 + \mu - A)^{-1}x_0 \, \d \mu \right \|_Y \\ \leq \|CA^{-1}\|_{\L(X, Y)} \int_\gamma \frac{|\mu|^{-\eta}}{|\mu_0 + \mu - 1 - \i \omega|} \|(\mu_0 + \mu - A)^{-1}x_0\|_X \, \d \mu
\leq K \|x_0\|_X, \quad x_0 \in X.
\end{multline}
Upon combining \cref{eq:func-calc,eq:est-LOT}, we end up with the following estimates, valid for $|\omega| \geq \omega_0$:
\begin{subequations}
\begin{align}
&|I_{1 + \i \omega}|^{-1}\|CA^{-1}(1 + \i \omega - A)^{-1}(\mu_0-A)^{-\eta}\|_{\L(X, Y)} \leq  K  \|CA^{-1}(1 + \i \omega)^{-1}\|_{\L(X, Y)} + K', \\
& |I_{1 + \i \omega}| \|CA^{-1}(1 + \i \omega)^{-1}\|_{\L(X, Y)} \leq K  + K'\|CA^{-1}(1 + \i \omega - A)^{-1}(\mu_0-A)^{-\eta}\|_{\L(X, Y)}.
\end{align}
\end{subequations}
The desired equivalence between \cref{eq:frac-shift-eta,eq:frac-shift-weight} then follows from the lower and upper bounds on $|I_{1 + \i \omega}|$ provided by the crucial estimate \cref{eq:est-I}.

$\cref{it:alt-adm} \implies \cref{it:adm-eta-fra}$:
Given $x_0 \in X_1$,
 we will write 
\begin{equation}
x(t) = S_t x_0, \quad
v_0 \triangleq (\mu_0 - A)^{-\eta} x_0, \quad v(t) \triangleq S_t v_0 = S_t (\mu_0 - A)^{-\eta}x_0, \quad t \geq 0.
\end{equation}
 \emph{A priori}, the $v$-variable is smoother than the $x$-variable by $\eta$ unit 
on the scale $\{X_s^\fra\}_{s\in \R}$. We also recall that $(\mu_0 - A)^{-\eta}$ is an isomorphism from $X$ onto $X^\fra_\eta$ and from $X_1$ onto $X^\fra_{1 + \eta}$. By density of $X^\fra_{1+\eta}$ in $X_1$, to get \cref{it:adm-eta-fra},
it
suffices to derive an estimate of the form
$\|C\tilde{w}\|_{L^2\hl{Y}} \leq K \|x_0\|_X$ for arbitrary $x_0 \in X_1$,
which we now fix.
By \cref{it:alt-adm} and \cref{lem:X-to-Cx},
$
\|C\tilde{x}\|_{H^{-\eta}\hl{Y}} \leq K \|x_0\|_X.
$
This has the following consequences. First, if $0 < \eta \leq 1/2$, then by \cref{coro:ext-zero} extension by zero is a continuous map from $H^{-\eta}\hl{Y}$ into $H^{-\eta}(\R, Y)$ and we directly obtain that
  \begin{equation}
  \label{eq:C-tilde-x-bis}
  \|C\tilde{x}\|_{H^{-\eta}(\R, Y)} \leq K \|x_0\|_X.
  \end{equation}
If $1/2 < \eta \leq 1$, again by \cref{coro:ext-zero} extension by zero is continuous $H^{1-\eta}_0\hl{Y} \to H^{1-\eta}(\R, Y)$. This suggests using the variable $z(t) \triangleq \int_0^{t}x(r) \, \d r = S_t A^{-1}x_0 - A^{-1}x_0$, $t \geq 0$, similarly as in \cref{lem:output-H-1}, so that $C\tilde{z} \in H^1_0\hl{Y}$. Then,
 $(\d / \d t)\tilde{z} = \tilde{x} - \tilde{z}$ and by \cref{lem:diff-shift} we obtain
  \begin{equation}
  \label{eq:C-tilde-z-bis}
  \|C\tilde{z}\|_{H^{1-\eta}(\R, Y)} \leq K \|C\tilde{z}\|_{H^{1-\eta}\hl{Y}} \leq K' \|C \tilde{z}\|_{L^2\hl{Y}} + K'\|C\tilde{x}\|_{H^{-\eta}\hl{Y}} \leq K'' \|x_0\|_X,
  \end{equation}
  where we also used $A$-boundedness of $C$ to control the $L^2$-term. Now, regardless of the value of $\eta$, we observe that
 \cref{eq:C-tilde-x-bis,eq:C-tilde-z-bis} lead to the same frequency-domain estimate, namely
\begin{equation}
\label{eq:freq-est}
\int_\R \frac{1}{(1 + \omega^2)^{\eta}} \|C \hat{x}(1 + \i \omega)\|^2_Y \, \d \omega \leq K \|x_0\|_X^2.
\end{equation}
With \cref{eq:freq-est} in hand, we return to the smoother $v$-variable.
Thanks to the formula \cref{eq:formula-res-A-eta}, we are able to derive the following frequency-domain expression of $C\tilde{v}$:
\begin{equation}
\label{eq:dec-w-FQ}
C \hat{v}(1 + \i \omega)  = \frac{I_{1 + \i \omega}}{2\pi \i} C(1 + \i \omega - A)^{-1}x_0  - \frac{1}{2\pi\i} C \int_\gamma \frac{(- \mu)^{-\eta}}{\mu_0 + \mu - 1 - \i \omega} (\mu_0 + \mu - A)^{-1}x_0 \, \d \mu, \quad \omega \in \R.
\end{equation}
Recall that for the purpose of the proof $x_0$ is chosen to be in $X_1$; thus
 the integral in \cref{eq:dec-w-FQ} is absolutely convergent in $X_1$ and we may apply $C$ to it.
For $\omega \in \R$, let $G_1(\omega)$ and $G_2(\omega)$ be the first and second terms in \cref{eq:dec-w-FQ}, respectively. Let $g_i \triangleq \F^{-1} [G_i]$, so that $C\tilde{v} = g_1 + g_2$ a.e.\  in $(0, +\infty)$. A closer examination of $G_2$ reveals that,
for a.e.\  $t \in \R$,
\begin{equation}
g_2(t) = - \mathds{1}_{(- \infty, 0)}(t) \frac{1}{2\pi\i } C \int_\gamma e^{(\mu_0 + \mu - 1)t} (-\mu)^{-\eta}(\mu_0 + \mu - A)^{-1} x_0 \, \d \mu.
\end{equation}
In particular,
 $g_2 = 0$ a.e.\ in $(0, +\infty)$, meaning that the $L^2\hl{Y}$-norms of $C\tilde{v}$ and $g_1$ coincide.
 Using Plancherel's theorem,
\begin{equation}
\label{eq:norm-g1}
\|g_1\|_{L^2\hl{Y}}\leq \|g_1\|_{L^2(\R, Y)}^2 = \frac{1}{2\pi} \int_\R \|G_1(\omega)\|^2_Y \, \d \omega \leq K \int_\R |I_{1 + \i \omega}|^2 \|C(1 + \i \omega - A)^{-1}x_0\|_Y^2 \, \d \omega.
\end{equation}
We split the last integral in \cref{eq:norm-g1} into low frequencies $|\omega| < \omega_0$ and high frequencies $|\omega| > \omega_0$.
 On the one hand,
\begin{equation}
\label{eq:LF-v}
\int_{-\omega_0}^{\omega_0} |I_{1 + \i \omega}|^2 \|C(1 + \i \omega - A)^{-1}x_0\|_Y^2 \, \d \omega  \leq K \int_{-\omega_0}^{\omega_0}  \|C(1 + \i \omega - A)^{-1}x_0\|_Y^2 \, \d \omega \leq K' \|x_0\|^2_X,
\end{equation}
where we first used \cref{eq:est-I-global} and then $A$-boundedness of $C$ together with boundedness of the resolvent in $\L(X, X_1)$-norm on \emph{bounded} subsets of the open right-half plane. On the other hand,
\begin{equation}
\label{eq:HF-v}
\begin{aligned}
\int_{
|\omega| > \omega_0}
 |I_{1 + \i \omega}|^2 \|C(1 + \i \omega - A)^{-1}x_0\|_Y^2 \, \d \omega &\leq K \int_{
|\omega| > \omega_0
} |\omega|^{-2\eta} \|C(1 + \i \omega - A)^{-1}x_0\|_Y^2 \, \d \omega \\ &\leq K' \int_\R \frac{1}{(1 + |\omega|^\eta)^2} \|C(1 + \i \omega -A)^{-1}x_0\|_Y^2 \, \d \omega
\end{aligned}
\end{equation}
by to the upper bound for $|I_{1+\i\omega}|$ in \cref{eq:est-I}. Recalling \cref{eq:freq-est}, we see that the last integral in \cref{eq:HF-v} is bounded by $K\|x_0\|^2$, and we finally obtain the desired estimate $\|C\tilde{v}\|_{L^2\hl{Y}} \leq K \|x_0\|$ by substituting \cref{eq:LF-v,eq:HF-v} back into \cref{eq:norm-g1}, concluding the proof of \cref{it:adm-eta-fra}.

$\cref{it:bounded-fra} \implies \cref{it:adm-eta-fra}$ (With $\sg{S}$ left-invertible): Let $C_\eta \triangleq C(-A)^{-\eta}$.  Then $C_\eta$ is $A$-bounded and since $(-A)^{-\eta}$ commutes with the resolvent of $A$, \cref{it:bounded-fra} is equivalent to
\begin{equation}
\label{eq:bound-fra-bis}
\|C_\eta(1 + \i \omega - A)^{-1}\|_{\L(X, Y)} = O(1), \quad \omega \in \R, \quad |\omega| \to + \infty.
\end{equation}
Thus, taking advantage of the left-invertibility of $\sg{S}$ (as a semigroup on $X$), by \cite[Corollary 5.2.4]{TucWei09book},\footnote{The reader will also find the argument in the proof of \cref{prop:interpolation} below.} we may deduce from \cref{eq:bound-fra-bis} that $C_\eta$ is an admissible observation operator in the usual sense of \cref{def:admissible-observation}. Recall that  $(-A)^{-\eta}$ commutes with the semigroup, and is also an isomorphism from $X$ onto $X_\eta^\inte$ and from $X_1$ onto $X_{\eta+1}^\inte$.
 We may therefore reformulate the admissibility property of $C_\eta$ as the existence of positive constants $K, T$ such that
\begin{equation}
\int_0^T \|C S_t x_0\|^2_Y \, \d t \leq K \|x_0\|^2_{X_\eta^\fra}, \quad x_0 \in X_{1 + \eta}^\fra.
\end{equation}
By density of $X^\fra_{1 + \eta}$ in $X_1$, \cref{it:adm-eta-fra} follows.
\end{proof}

\begin{rem}
For the implication $\cref{it:bounded-fra} \implies \cref{it:adm-eta-fra}$, it is to tempting to apply \cite[Corollary 5.2.4]{TucWei09book} directly to $\sg{S}$ as a semigroup on $X^\fra_\eta$;  however the left inverses provided by left-invertibility assumption might not be in $\L(X_\eta^\fra)$ in general.
\end{rem}


\subsubsection{The case of interpolation spaces}

We now turn to the part of the proof of \cref{theo:inter-group} involving the interpolation spaces $X_\eta^\inte$. 
The key result here is the following proposition.
\begin{prop} 
\label{prop:est-z}
Let $0 \leq \eta \leq 1$ and assume that \cref{it:adm-eta} holds. Given $x_0 \in X$,
 consider the new variable $z$ defined by
\begin{equation}
\label{eq:formula-z}
z(t) \triangleq \int_0^t x(r) \, \d r = S_t A^{-1}x_0 - A^{-1}x_0, \quad t \geq 0.
\end{equation}
There exists $K > 0$ such that
\begin{equation}
\label{eq:est-z}
\|C\tilde{z}\|_{H^{1-\eta}\hl{Y}} \leq K \|x_0\|_X, \quad x_0 \in X.
\end{equation}
\end{prop}

To prove \cref{prop:est-z} we will make use of a general interpolation result, which we state here for the convenience of the reader; see, e.g., \cite[Proposition 2.4, Chapter 4]{Tay96book} or \cite[Theorem 10.1, Chapter 1]{LioMag68book} for proofs and additional statements.
\begin{theo}
\label{theo:interpolation-sg}
Let $E_0$ and $E_1$ be Hilbert spaces with continuous and dense embedding $E_1 \hookrightarrow E_0$. Let $\sg{T}$ be {any} bounded strongly continuous semigroup on $E_0$ such that the domain of its generator (equipped with the graph norm) is $E_1$. Fix  $ 0 < \theta <1$. Then, given a vector $u\in E_0$,
\begin{equation}
\label{eq:carac-int-sg}
u \in [E_1, E_0]_\theta ~\mbox{if and only if}~ \int_0^{+\infty} \frac{1}{\tau^{3 - 2\theta}}\|(T_t - 1)u\|^2_{E_0} \, \d \tau < + \infty.
\end{equation}
Furthermore, there exists $K > 0$ such that 
\begin{equation}
\label{eq:carac-int-sg-norm}
K^{-1}\|u\|_{[E_1, E_0]_\theta}^2 \leq \|u\|^2_{E_0} + \int_0^{+\infty} \frac{1}{\tau^{3 - 2\theta}}\|(T_t - 1)u\|^2_{E_0} \, \d \tau \leq K \|u\|_{[E_1, E_0]_\theta}^2, \quad u \in [E_1, E_0]_\theta.
\end{equation}
\end{theo}

\begin{proof}[Proof of \cref{prop:est-z}] The result is straightforward if $\eta = 0$ or $\eta = 1$, so we assume that $0 < \eta < 1$.
Applying \cref{theo:interpolation-sg} to the left translation semigroup on $L^2\hl{Y}$ leads to a Besov-type characterisation of the space $H^{1-\eta}\hl{Y}$:
given $f \in L^2\hl{Y}$, we have $f \in H^{1-\eta}\hl{Y}$ if and only if
\begin{equation}
\label{eq:trans-norm}
\int_0^{+\infty} \frac{1}{\tau^{3 -2\eta}} \int_0^{+\infty} \|f(t + \tau) - f(t)\|^2_Y \, \d t \, \d \tau < + \infty.
\end{equation} 
The integral in \cref{eq:trans-norm} defines an equivalent norm on $H^{1-\eta}\hl{Y}$
if supplemented with a lower-order $L^2$-term; see also \cite[Section 10, Chapter 1]{LioMag68book} or \cite{Sim90}.
In what follows,
we write $\tilde{S}_t \triangleq e^{-t} S_t$ for $t \geq 0$. Clearly, $\sg{\tilde{S}}$ is also a bounded strongly continuous semigroup and its generator is $A - 1$.
Fix
$x_0 \in X$. We already know that $\|C\tilde{z}\|_{L^2\hl{Y}} \leq K \|x_0\|_X$ by $A$-boundedness of $C$.
We also observe that
\begin{equation}
\label{eq:dec-z-fra}
\tilde{z}(t + \tau) - \tilde{z}(t) 
= \tilde{S}_t (\tilde{S}_\tau - 1)A^{-1}x_0 -(e^{-t - \tau} - e^{-t})A^{-1}x_0, \quad t \geq 0, \quad \tau \geq 0.
\end{equation}
Therefore, with \cref{eq:trans-norm} in mind, we obtain
\begin{multline}
\label{eq:dec-z-fra-est}
\int_0^{+\infty} \frac{1}{\tau^{3 -2\eta}} \int_0^{+\infty} \|C\tilde{z}(t + \tau) - C\tilde{z}(t)\|^2_Y \, \d t \, \d \tau 
\leq 2 \int_0^{+\infty} \frac{1}{\tau^{3 -2\eta}} \int_0^{+\infty} \|C\tilde{S}_t (\tilde{S}_\tau - 1)A^{-1}x_0\|_Y^2 \, \d t \, \d \tau \\ + 2 \|CA^{-1}x_0\|^2_Y \int_0^{+\infty} \frac{1}{\tau^{3 -2\eta}} \int_0^{+\infty} |e^{-t - \tau} -e^{-t}|^2 \, \d t \, \d \tau.
\end{multline}
Of course
$e^{-\cdot}$ is in $H^{1-\eta}(0, +\infty)$, which means that the second double integral on the right-hand side of \cref{eq:dec-z-fra-est} is finite. Moreover,
$\|CA^{-1}x_0\|_Y \leq K \|x_0\|_X$ by $A$-boundedness of $C$.
Now, under condition \cref{it:adm-eta}
we can use
 \cref{lem:adm-bounded-CR}, whose first statement can be rewritten as follows: there exists $K > 0$ such that
\begin{equation}
\label{eq:eta-adm-v}
\int_0^{+\infty} \|C\tilde{S}_t v_0\|^2_Y \, \d t \leq K \|v_0\|^2_{X_\eta^\inte},
\quad v_0 \in X_\eta^\inte.
\end{equation}
\Cref{eq:eta-adm-v} applied to $v_0 = (\tilde{S}_\tau - 1)x_0$, $ \tau \geq 0$, yields
\begin{equation}
\label{eq:est-shift-adm}
\int_0^{+\infty} \|C \tilde{S}_t (\tilde{S}_\tau - 1) A^{-1}x_0\|^2_Y \, \d t \leq K \|(\tilde{S}_\tau - 1)A^{-1}x_0\|_{X_\eta^\inte}^2, \quad \tau \geq 0.
\end{equation}
Plugging \cref{eq:est-shift-adm} into \cref{eq:dec-z-fra-est} and using our preliminary $L^2$-estimate, we obtain
\begin{multline}
\label{eq:Cztilde-L2-trans}
\|C\tilde{z}\|^2_{L^2\hl{Y}} + 
\int_0^{+\infty} \frac{1}{\tau^{3 -2\eta}} \int_0^{+\infty} \|C\tilde{z}(t + \tau) - C\tilde{z}(t)\|^2_Y \, \d t \, \d \tau  \\ \leq K \|x_0\|^2_X + K'
\int_0^{+\infty} \frac{1}{\tau^{3 -2\eta}} \|(\tilde{S}_\tau - 1)A^{-1}x_0\|_{X_\eta^\inte}^2 \, \d \tau.
\end{multline}
Recall from \cref{sec:pre-sta} that the domain of the generator of $\sg{\tilde{S}}$ \emph{restricted} to $X_\eta^{\inte}$ is precisely $X_{1 + \eta}^\inte$. Notice also that
$
X_1 = [X^\inte_{1 + \eta}, X_\eta^{\inte}]_\eta
$.
Therefore, \cref{theo:interpolation-sg} yields
\begin{equation}
\label{eq:final-frac-int}
\int_0^{+\infty} \frac{1}{\tau^{3 -2\eta}} \|(\tilde{S}_\tau - 1)A^{-1}x_0\|^2_{X_\eta^\inte} \, \d \tau \leq K \|A^{-1} x_0\|^2_{X_1} = K \|x_0\|^2_X,
\end{equation}
which finally leads to $\|C\tilde{z}\|^2_{H^{1-\eta}} \leq K \|x_0\|^2$, as required.
\end{proof}


\begin{coro}[\cref{it:adm-eta} implies \cref{it:alt-adm}]
\label{coro:ii-to-iii}
Let $0 \leq \eta \leq 1$ and assume that \cref{it:adm-eta} holds. Then there exists $K > 0$ such that
\begin{equation}
\label{eq:alt-ex-bis}
\|C\tilde{x}\|_{H^{-\eta}(0, +\infty; Y)} \leq K \|x_0\|_X, \quad x_0 \in X_1.
\end{equation}
In particular, for any $T > 0$, there exists $K > 0$ such that 
\begin{equation}
\label{eq:alt-adm-bis}
\|Cx\|_{H^{-\eta}(0, T; Y)} \leq K \|x_0\|_X, \quad x_0 \in X_1.
\end{equation}
\end{coro}
\begin{proof}
Let $x_0 \in X_1$ and define $z$ as in \cref{prop:est-z}. Since $(\d / \d t)\tilde{z} = - \tilde{z} + \tilde{x}$, \cref{eq:alt-ex-bis} follows from
the $H^{1-\eta}$-estimate \cref{eq:est-z} and the elementary $L^2$-bound 
$\|C\tilde{z}\|_{L^2\hl{Y}} \leq K \|x_0\|_X$ using \cref{lem:diff-shift}.
\Cref{eq:alt-adm-bis} then readily follows from the restriction property and 
the fact that multiplication by $e^{\cdot}$ is continuous $H^{-\eta}(0, T; Y) \to H^{-\eta}(0, T; Y)$.
\end{proof}

With the help of \cref{prop:est-z,coro:ii-to-iii}, we are able prove another series of implications.
\begin{prop}[Interpolation]
\label{prop:interpolation}
Let $0 \leq \eta \leq 1$. Then,
\begin{subequations}
\begin{align}
&\cref{it:adm-eta} \implies \cref{it:alt-adm} \implies \cref{it:FQ-cond}, \quad \cref{it:adm-eta} \iff \cref{it:alt-adm} &&\mbox{in general;}\\
&\cref{it:bounded-fra} \implies \cref{it:alt-adm} &&\mbox{under the left-invertibility assumption.}
\end{align}
\end{subequations}
\end{prop}


\begin{proof}
We already know from \cref{coro:ii-to-iii} that \cref{it:adm-eta} implies \cref{it:alt-adm}.

$\cref{it:adm-eta} \implies \cref{it:FQ-cond}$:
We shall prove that
\begin{equation}
\label{eq:freq-2-omega}
\sup_{\omega \in \R} \frac{1}{|2 + \i \omega|}\|C(2 + \i \omega - A)^{-1}\|_{\L(X, Y)}< + \infty.
\end{equation}
The corresponding growth estimate on $1 + \i \R$ can then be deduced using the resolvent identify, much as in the proof of \cref{lem:adm-bounded-CR}.


Let $x_0 \in X_1$. The corresponding solution $x$ to \cref{eq:cauchy-x-A} satisfies
$x \in \C(\R^+, X_1) \cap \C^1(\R^+, X)$ and $x$ is bounded in $X_1$. Since $C$ is $A$-bounded,
the $Y$-valued function $F$ defined by
\begin{equation}
\label{eq:def-F-eta}
F(p) \triangleq \frac{1}{p^\eta} C\hat{x}(p) = \frac{1}{p^\eta} C(p - A)^{-1}x_0, \quad \re p > 0,
\end{equation}
is holomorphic and bounded on every open half-plane of the form $\re p \geq \sigma$ with $\sigma > 0$.
By the Paley--Wiener theorem \cite[Theorem 1.8.3]{AreBat01book}, $F$ is the Laplace transform of a measurable $Y$-valued function $f$ that satisfies $e^{-\sigma \cdot}f \in L^2(0, +\infty; Y)$ for any $\sigma > 0$. We can then write
\begin{equation}
\label{eq:time-int-F}
F(2 + \i \omega) = \int_0^{+\infty} e^{-(2 + \i \omega)t} f(t) \, \d t, \quad \omega \in \R,
\end{equation}
where the integral is absolutely convergent in $Y$. The Cauchy--Schwarz inequality yields
\begin{equation}
\label{eq:simp-cs}
\|F(2 + \i \omega)\|_Y^2 \leq \left (\int_0^{+\infty} e^{-2t} \, \d t \right) \left (\int_0^{+\infty} \|e^{-t}f(t)\|^2_Y \, \d t \right ) = \frac{1}{2}\int_0^{+\infty} \|e^{-t}f(t)\|^2_Y \, \d t, \quad \omega \in \R.
\end{equation}
By Plancherel's theorem,
\begin{equation}
\label{eq:F-CS-plan}
\int_0^{+\infty} \|e^{-t}f(t)\|^2_Y \, \d t =  \frac{1}{2 \pi}  \int_\R \|\hat{f}(1 + \i \xi)\|^2_Y \, \d \xi = \frac{1}{2\pi} \int_\R \frac{1}{|1 + \i \xi|^{2\eta}} \|C\hat{x}(1 + \i \xi)\|^2_Y \, \d \xi.
\end{equation}
Substituting \cref{eq:F-CS-plan} into \cref{eq:simp-cs} gives
\begin{equation}
\label{eq:frac-C-Cx}
\frac{1}{|2 + \i \omega|^2} \|C(2 + \i \omega - A)^{-1}x_0\|_Y^2 \leq \frac{1}{4\pi} \int_\R \frac{1}{|1 + \i \xi|^{2\eta}} \|C\hat{x}(1 + \i \xi)\|^2_Y \, \d \xi, \quad \omega \in \R.
\end{equation}
Under condition \cref{it:adm-eta}, \cref{prop:est-z,coro:ii-to-iii} yield the estimates
\begin{equation}
\label{eq:est-Cx-Cz}
\|C\tilde{x}\|_{H^{-\eta} \hl{Y}} \leq K \|x_0\|_X, \quad \|C\tilde{z}\|_{H^{1-\eta}\hl{Y}} \leq K' \|x_0\|_X.
\end{equation}
Just as in the proof of the implication from \cref{it:alt-adm} to \cref{it:adm-eta-fra} in \cref{prop:frac}, we may deduce from \cref{eq:est-Cx-Cz} that 
\begin{equation}
\int_\R \frac{1}{(1 + \xi^2)^\eta} \|C\hat{x}(1 + \i \xi)\|^2_Y \, \d \xi \leq K \|x_0\|^2_X.
\end{equation}
In combination with \cref{eq:frac-C-Cx}, this shows that
\begin{equation}
\label{eq:unif-estimate-TF}
\frac{1}{|2 + \i \omega|} \|C(2 + \i \omega - A)^{-1}x_0\|_Y  \leq K \|x_0\|_X, \quad \omega \in \R.
\end{equation}
\Cref{eq:unif-estimate-TF} is valid for arbitrary data $x_0$ in $X_1$ and thus for $x_0$ in $X$ by density.

$\cref{it:alt-adm} \implies \cref{it:adm-eta}$:
We already know from \cref{lem:X-to-Cx} that, under condition \cref{it:alt-adm}, the map $x_0 \mapsto C\tilde{x}$ is continuous  $X \to H^{-\eta}(0, +\infty; Y)$. 
Let us show that it is also continuous $X_1 \to H^{1-\eta}(0, +\infty; Y)$. Pick some (smoother) data $v_0 = A^{-1}x_0 \in X_1$, and write
 $v(t) \triangleq S_t v_0$ and $x(t) \triangleq S_t x_0$ for $t\geq 0$.
We have $(\d / \d t)\tilde{v} = - \tilde{v} + \tilde{x}$ so that, using \cref{lem:diff-shift} and also $A$-boundedness of $C$,
\begin{equation}
\label{eq:diff-Cw}
\|C\tilde{v}\|_{H^{1-\eta}\hl{Y}} \leq K \|C\tilde{v}\|_{L^2\hl{Y}} + K \|C \tilde{x}\|_{H^{-\eta}\hl{Y}}  \leq K' \|v_0\|_{X_1},
\end{equation}
as required.
Now,
by \cref{lem:L2-rec}, 
$
L^2(0, +\infty; Y) = [H^{1-\eta}(0, +\infty; Y), H^{-\eta}(0, +\infty; Y)]_{1 - \eta}
$
.
With that in hand,
we obtain continuity $X_\eta^\inte \to L^2\hl{Y}$ of the map $x_0 \mapsto C\tilde{x}$
by interpolating between
 $X \to H^{-\eta}\hl{Y}$ and $X_1 \to H^{1-\eta}\hl{Y}$.
 It immediately follows that the map $x_0 \mapsto Cx$ is continuous $X_\eta^\inte \to L^2(0, T; Y)$ for any $T > 0$, which proves \cref{it:adm-eta}.

$\cref{it:FQ-cond} \implies \cref{it:alt-adm}$ (With $\sg{S}$ left-invertible): 
Here we will take advantage of the left-invertibility assumption to deduce information on semigroup orbits $x$ from convolution products $S\ast x$.\footnote 
{
In the context of standard admissibility theory, a similar idea is used in the proof of \cite[Corollary 5.2.4]{TucWei09book}.
}
To this end, we start by observing that, as a consequence of \cref{it:FQ-cond},
\begin{equation}
\sup_{\omega \in \R} \frac{1}{(1 + \omega^2)^{\eta/2}} \|C(1 + \i \omega - A)^{-1}\|_{\L(X, Y)} < + \infty.
\end{equation}
Hence $\omega \mapsto C(1 + \i \omega - A)^{-1}$ is a Fourier multiplier from $L^2(\R, X)$ into $H^{-\eta}(\R, Y)$, i.e.,
\begin{equation}
\int_\R \frac{1}{(1 + \omega^2)^\eta} \|C(1 + \i \omega - A)^{-1} \hat{f}(\i \omega)\|^2_Y \, \d \omega \leq K \int_\R \|f(t)\|^2_X \, \d t, \quad f \in L^2(\R, X).
\end{equation}
In particular,
\begin{equation}
\label{eq:eta-mult-g}
\|C(\widetilde{S \ast g})\|_{H^{-\eta}(\R, Y)} \leq K \|\tilde{g}\|_{L^2(0, +\infty; X)}, \quad g \in L^2(0, +\infty, X).
\end{equation}
Now, let $x_0 \in X_1$ and observe that
\begin{equation}
x(t) = \frac{1}{t} \int_0^{t} x(t) \, \d r = \frac{1}{t} \int_0^t S_{t - r} S_r x_0 \, \d r = \frac{1}{t} (S \ast x)(t), \quad t > 0.
\end{equation}
Therefore (notice that we take the norms on the interval $(1, +\infty)$),
\begin{equation}
\label{eq:coincide-norms-1}
\|C\tilde{x}\|_{H^{-\eta}(1, +\infty; Y)} = \|(1/\cdot) C(\widetilde{S \ast x})\|_{H^{-\eta}(1, +\infty; Y)}.
\end{equation}
Both the function $t \mapsto 1/t$ and and its derivative $t \mapsto t-1/t^2$ are bounded on $(1, +\infty)$, and with interpolation and duality it is straightforward to show that
that multiplication by $t \mapsto 1/t$ is continuous with respect to the $H^{-\eta}(1, +\infty)$-norm. Using the restriction property, we then have

\begin{equation}
\label{eq:conv-est}
 \|(1/\cdot) C(\widetilde{S \ast x})\|_{H^{-\eta}(1, +\infty; Y)} \leq K \|C(\widetilde{S \ast x})\|_{H^{-\eta}(1, +\infty; Y)} \leq K' \|C(\widetilde{S \ast x})\|_{H^{-\eta}(\R, Y)}.
\end{equation}
On the other hand, the Fourier multiplier estimate \cref{eq:eta-mult-g} applied to $x$ yields
\begin{equation}
 \|C(\widetilde{S \ast x})\|_{H^{-\eta}(\R, Y)} \leq K \|\tilde{x}\|_{L^2\hl{X}} \leq K' \|x_0\|_X.
\end{equation}
With \cref{eq:coincide-norms-1,eq:conv-est}, it follows that
\begin{equation}
\label{eq:est-eta-hl}
\|C\tilde{x}\|_{H^{-\eta}(1, +\infty; Y)}  \leq K \|x_0\|_X.
\end{equation}
\Cref{eq:est-eta-hl}
 valid for arbitrary  data $x_0 \in X$
and we may rephrase it as continuity $X \to H^{1-\eta}(0, +\infty; Y)$ of the map
$x_0 \mapsto [C\tilde{x}](\cdot + 1)$. 
Next, we recall from \cref{lem:output-H-1}
 that $x \mapsto C\tilde{x}$ is well-defined as a bounded linear operator $X \to H^{-1}\hl{Y}$. Note here that, as right translations continuously map $H^{1}_0\hl{Y}$ into itself, by duality left translations extend to bounded linear operators on  $H^{-1}(0, +\infty; Y)$.
 Since we assume left-invertibility of the semigroup,
 there exists $S_{-1}^{\mathrm{left}} \in \L(X)$ such that $S_{-1}^{\mathrm{left}} S_1 = \id$.
Let $x_0 \in X$, $x_{-1} \triangleq S_{-1}^\mathrm{left} x_0$ and $v(t) \triangleq S_{t}x_{-1}$ for $t \geq 0$. It is easy to check that $C\tilde{x}$ and $e^{-1}[C\tilde{v}](\cdot + 1)$ coincide as elements of $H^{-1}\hl{E}$. But we also know from above that 
 $[C\tilde{v}](\cdot + 1)$ must in fact belong to $H^{-\eta}\hl{Y}$ and we have an estimate of its $H^{-\eta}$-norm in terms of $\|v(0)\|_X = \|x_{-1}\|_X$. 
Thus
\begin{equation}
\label{eq:est-Cx-inf}
\|C\tilde{x}\|_{H^{-\eta}\hl{Y}} \leq K \|x_{-1}\|_X = K \|S_{-1}^\mathrm{left}x_0\|_X \leq K' \|x_0\|_X.
\end{equation}
 Just as in the proof of \cref{coro:ii-to-iii}, we infer from
\cref{eq:est-Cx-inf}, which is valid for arbitrary $x \in X$, that the map $x \mapsto Cx$ is continuous $X \to H^{-\eta}(0, T; Y)$ for any $T > 0$. In particular, \cref{it:alt-adm} is true.
\end{proof}

\subsection{Proof of \cref{theo:input-group}}
\label{sec:proof-input}

\Cref{theo:input-group} is essentially the dual version of \cref{theo:inter-group}. Considering the inhomogeneous Cauchy problem with zero initial data
\begin{equation}
\label{eq:in-cauchy-zero}
\dot{x} = Ax + Bu, \quad x(0) = 0,
\end{equation}
we shall show that each item of \cref{theo:input-group} is equivalent to its counterpart in \cref{theo:inter-group} stated for the adjoint semigroup generator $A^\ast$ and the dual output operator $B^\ast$. Again, without loss of generality we may assume that $\sg{S}$ is uniformly bounded and that $0 \in \rho(A)$.

%
\subsubsection{Duality in abstract Sobolev scales}

Let $A^* : \dom(A^*) \to X$ be the adjoint of $A$ and let $X_1^\d$ denote $\dom(A^*)$ equipped with the graph norm $\|A^* \cdot \|_X$. 
As in \cref{sec:pre-sta}, we may define for any $-1 \leq \theta \leq 1$ the spaces $X_\theta^{\d, \fra}$ and $X_\theta^{\d, \inte}$ associated with $A^\ast$. 
Let $0 \leq \theta \leq 1$ and let $X_\theta^\d$ denote either
 $X_\theta^{\d, \fra}$ or $X_\theta^{\d, \inte}$. 
Recall that $X_1^\d \hookrightarrow X_\theta^\d \hookrightarrow X$ with continuous and dense embeddings.
By considering the
 natural restriction maps (which are the adjoints of the previous embeddings), we see that
\begin{equation}
\label{eq:dual-chain}
X^\ast \hookrightarrow (X_\theta^\d)^\ast \hookrightarrow (X_1^\d)^\ast.
\end{equation}
The embeddings are continuous and, since our spaces are reflexive, also dense. Now let $j : X \to X^\ast$ be the map defined by
\begin{equation}
\langle jx, \varphi\rangle_{X^\ast, X} \triangleq \langle x, \varphi \rangle_X, \quad x \in X, \quad \varphi \in X.
\end{equation}
In view of \cref{eq:dual-chain}, we have
\begin{equation}
\langle x, \varphi \rangle_X = \langle jx, \varphi \rangle_{X^\ast, X} = \langle jx, \varphi \rangle_{(X_\theta^\d)^\ast, X_\theta^\d} = \langle j x, \varphi\rangle_{(X_1^\d)^\ast, X_1^\d}, \quad x \in X, \quad \varphi \in X_1^\d.
\end{equation}
Furthermore, by Riesz' theorem, $j$ is an isometric isomorphism from $X$ onto $X^\ast$.
In the sequel,  $X_{-\theta}$ denotes either $X_{-\theta}^\fra$ (if $X_\theta^\d = X_\theta^{\d, \fra}$) or $X_{-\theta}^\inte$ (if $X_\theta^\d = X_\theta^{\d, \inte}$).

\begin{prop}
\label{lem:embedding}
Let $0 \leq \theta \leq 1$. Then $j$ extends to an isometric isomorphism between $X_{-\theta}$ and $(X_\theta^\d)^\ast$.
\end{prop}
\begin{proof}
We start by writing
\begin{equation}
\langle x, \varphi \rangle_X  = \langle x, (- A^\ast)^{-\theta} (-A^\ast)^\theta \varphi \rangle_X = \langle (-A)^{-\theta} x, (-A^\ast)^\theta \varphi \rangle_X, \quad x \in X, \quad \varphi \in X_\theta^{\d,\fra}.
\end{equation}
Recalling that $(-A^\ast)^\theta$ is an isometric isomorphism from $X_\theta^{\d, \fra}$ onto $X$, we have, for all $x \in X$ 
\begin{equation}
\label{eq:j-X-theta}
\begin{aligned}
\|j x \|_{(X^{\d, \fra}_\theta)^\ast} &
= \sup_{\varphi \in X_\theta^{\d,\fra} \setminus \{0\}} \frac{|\langle (-A)^{-\theta} x, (-A^\ast)^\theta \varphi \rangle_X|}{\|(-A^\ast)^\theta \varphi\|_X}
\\ &= \sup_{\psi \in X \setminus \{0\}} \frac{|\langle (-A)^{-\theta} x, \psi \rangle_X|}{\|\psi\|_X}
 = \|(-A)^{-\theta} x \|_X = \|x\|_{X_{-\theta}^\fra}.
\end{aligned}
\end{equation}
Since $X$ is dense in $X_{-\theta}^\fra$, \cref{eq:j-X-theta} shows that 
$j$ extends to an isometry
$
X_{-\theta}^\fra \to (X_\theta^{\d, \fra})^\ast
$.
In particular, the extension, which we still denote by $j$, is injective. 
To check that it is surjective, we first note that
 the image of $X^\fra_{-\theta}$ under $j$ must contain $X^\ast$ and is therefore dense in $(X_\theta^{\d, \fra})^\ast$; owing to the isometric nature of $j$ and completeness of $X^\fra_{-\theta}$,
this image must also be closed and the claim follows.

The result is now proved for the scale of domains of fractional powers. In particular, the case $\theta = 1$ means that
$j : X_{-1} \to (X_{1}^\d)^\ast$ is an isometric isomorphism.
We can now deduce the required property for the interpolation scale by interpolating between $X_{-1} \to (X_1^{\d})^\ast$ and $X \to X^\ast$. More precisely,
we
recall that $X_{-\theta}^\inte = [X, X_{-1}]_{1 - \theta}$ and
apply \cite[Theorem 3.7]{ChaHew15} to see that
\begin{equation}
(X_\theta^{\d, \inte})^\ast = ([X_1^\d, X]_{1 - \theta})^\ast = [(X_1^\d)^\ast, X]_{\theta},
\end{equation}
with \emph{equality} of norms. As a result, considering $j$ and $j^{-1}$, we finally obtain that
$j : X_{-\theta}^\inte \to (X_\theta^{\d, \inte})^\ast$ is an isometric\footnote{
Recall that our choice of norm for interpolation spaces preserves constants in interpolation inequalities; see \cref{eq:geo-inter}.}
 isomorphism as well, which completes the proof.
\end{proof}

\Cref{lem:embedding} makes possible the identification $X_{-\theta} \simeq (X_\theta^\d)^\ast$ through the map $j$. In particular, $X$ is identified with its antidual $X^\ast$ and we can work
with the chain of continuous and dense embeddings
\begin{equation}
\label{eq:pivot-duality}
X_1^\d \hookrightarrow X_\theta^\d \hookrightarrow X \hookrightarrow X_{-\theta} \hookrightarrow X_{-1}, \quad 0 < \theta < 1,
\end{equation}
where the antiduality pairing $\langle \cdot, \cdot \rangle_{X_{-1}, X_1^\d}$ naturally extends $\langle \cdot, \cdot \rangle_{X_{-\theta}, X_\theta^\d}$, which in turn extends the scalar product $\langle \cdot, \cdot \rangle_X$.


\subsubsection{Duality between control and observation}
\label{sec:input-output-dual}

Having also identified $U$ with its antidual, we define the adjoint $B^* \in \L(X_1^\d, U)$ of $B$ in the pivot duality $X_1^\d \hookrightarrow X \hookrightarrow X_{-1}$ by
\begin{equation}
\langle u, B^*x \rangle_U \triangleq \langle Bu, x \rangle_{X_{-1}, X_1^\d}, \quad u \in U, \quad x \in X_1^\d.
\end{equation} 
We will apply \cref{theo:inter-group} 
 to the adjoint semigroup $\sg{S^\ast}$ with  generator $A^\ast : X_1^\d \to X$ and the $A^\ast$-bounded observation operator $B^\ast : X_1^\d \to U$. We also note that $\sg{S}$ is right-invertible if and only if $\sg{S^*}$ is left-invertible. Furthermore,
for all $\omega \in \R$, $u \in U$ and $x \in X$,
\begin{equation}
\langle (1 + \i \omega - A)^{-1}Bu, x \rangle_X = \langle Bu, (1 - \i \omega - A^\ast)^{-1}x \rangle_{X_{-1}, X_1^\d} = \langle u, B^\ast (1 - \i \omega -A^\ast)^{-1}x \rangle_U.
\end{equation}
As a consequence,
\begin{equation}
\label{eq:id-dual-op-norm}
\|(1 + \i \omega - A)^{-1}B\|_{\L(U, X)} = \|B^\ast(1 - \i \omega - A^\ast)^{-1}\|_{\L(X, U)}, \quad \omega \in \R.
\end{equation}
Now, given $T > 0$,
let $\Phi_T$ be the map that
 associates to an input $u$ defined on $(0, T)$ the state $x(T)$, where $x$ solves the inhomogeneous Cauchy problem \cref{eq:in-cauchy-zero}. It follows from \cref{eq:conv,eq:conv-ipp} that
\begin{equation}
\Phi_T \in \L(L^2(0, T; U), X_{-1}), \quad \Phi_T \in \L(H^1_0(0, T; U), X).
\end{equation}
Similarly as in \cref{lem:output-H-1}, we also see that
the map $\Psi_T : x_0 \mapsto B^*x$, where $x$ solves $\dot{x} = A^*x$, $x(0) = x_0$, satisfies
\begin{equation}
\Psi_T \in \L(X_1^\d, L^2(0, T; U)), \quad \Psi_T \in \L(X, H^{-1}(0, T; U)),
\end{equation}
where we extended $\Psi_T$ by density in the second statement.
 The following lemma is an implementation of a well-known duality relation between input and output maps; see, e.g.,
\cite[Section 4.10]{TucWei09book}. 
\begin{lemma}
\label{lem:dual-input-output}
The following identity holds:
\begin{equation}
\langle \Phi_T [u(T - \cdot)], x_0\rangle_{X} = \langle u, \Psi_T x_0\rangle_{L^2(0, T; U)}, \quad u \in H^1_0(0, T; U), \quad x_0 \in X_1^\d.
\end{equation}
\end{lemma}
\begin{proof}
Let $u \in H^1_0(0, T; U)$ and $x_0 \in X_1^\d$. Certainly, $u \in H^1_0(0, T; U)$ if and only if $u(T - \cdot) \in H^1_0(0, T; U)$. We recall that $[\Psi_T x_0](t) = B^\ast S_t^\ast x_0$ for $0 \leq t \leq 1$ and
\begin{equation}
\Phi_T u =  \int_0^T S_{T - t} A^{-1}B \dot{u}(t) \, \d t, \quad \Phi_T [u(T - \cdot)] =  - \int_0^T S_{T - t} A^{-1}B \dot{u}(T - t) \, \d t.
\end{equation}
 Therefore,
\begin{equation}
\label{eq:integral-Phi}
\langle \Phi_T [u(T - \cdot)], x_0 \rangle_X = - \int_0^T \langle S_{T - t}A^{-1}B \dot{u}(T -t), x_0 \rangle_X \, \d t = - \int_0^T \langle A^{-1}B \dot{u}(T -t), S^\ast_{T - t} x_0 \rangle_X \, \d t.
\end{equation}
Now, for a.e.\  $0 < s < T$,
\begin{equation}
\label{eq:bracket-A-star-B}
\langle A^{-1} B \dot{u}(T - s), S_{T - s}^\ast x_0 \rangle_X
= \langle B\dot{u}(T - s), (A^\ast)^{-1}S_{T - s}^\ast x_0 \rangle_{X_{-1}, X_1^\d}.
\end{equation}
Furthermore, as $x_0 \in X_1^\d$,
the function
$t \mapsto (A^\ast)^{-1}S_{T - t}^\ast x_0$ belongs to $
\C^1([0, T], X_1^\d)$ and
\begin{equation}
\label{eq:diff-A-star}
\dt{} (A^\ast)^{-1}S_{T - t}^\ast x_0 = - S^\ast_{T - t} x_0, \quad 0 < t < T,
\end{equation}
in the sense of strong differentiation in $X_1^\d$. Therefore, using differentiation rules for continuous  sesquilinear  forms, we may 
integrate by parts in \cref{eq:integral-Phi} and obtain
\begin{equation}
\langle \Phi_T [u(T - \cdot)], x_0 \rangle_X =  \int_0^T \langle B u(T - t), S^\ast_{T - t} x_0\rangle_{X_{-1}, X_1^\d} \, \d t = \int_0^T \langle u(t), B^\ast S^\ast_t x_0\rangle_U,
\end{equation}
which is the desired equality.
\end{proof}

We are finally in a position to complete the proof of \cref{theo:input-group}.




\begin{proof}[Proof of \cref{theo:input-group}] 

As mentioned above, we shall show that each input condition in \cref{it:input-frac,it:input-int,it:smooth-input,it:input-FQ} stated for $A$ and $B$ in \cref{theo:input-group} is equivalent to its output counterpart in \cref{it:adm-eta-fra,it:adm-eta,it:alt-adm,it:FQ-cond} for $A^\ast$ and $B^\ast$ in \cref{theo:inter-group}.  

$\cref{it:adm-eta-fra}  \iff \cref{it:input-frac}$, $\cref{it:adm-eta} \iff \cref{it:input-int}$: We treat the cases of fractional powers and interpolation spaces simultaneously. Recall that \cref{it:adm-eta-fra} (if $X_\eta^\d = X_\eta^{\d, \fra}$) or \cref{it:adm-eta} (if $X_\eta^\d = X_\eta^{\d, \inte}$) hold if and only if there exist $K, T > 0$ such that
\begin{equation}
\|\Psi_T x_0\|_{L^2(0, T; U)} \leq K \|x_0\|_{X_\eta^{\d}}, \quad x_0 \in X_1^\d;
\end{equation}
whereas \cref{it:input-frac} or \cref{it:input-int} hold if and only if there exist $K, T > 0$ such that
\begin{equation}
\|\Phi_T u\|_{X_{-\eta}} \leq K \|u\|_{L^2(0, T; U)}, \quad u \in H^1_0(0, T; U).
\end{equation}
Now, let $T > 0$.
If $u \in H^1_0(0, T; U)$, then $\Phi_T u \in X$ and, recalling from \cref{lem:embedding} that $X_{-\eta} \simeq (X_\eta^{\d})^\ast$ through the Riesz map of $X$,
\begin{equation}
\label{eq:norm-Phi}
\begin{aligned}
\|\Phi_T u\|_{X_{-\eta}} & = \sup_{x_0 \in X_1^\d \setminus \{0\}} \frac{|\langle \Phi_T u, x_0 \rangle_{X_{-\eta}, X_\eta^\d}|}{\|x_0\|_{X_\eta^\d}} =  \sup_{x_0 \in X_1^\d \setminus \{0\}} \frac{|\langle \Phi_T u, x_0 \rangle_{X}|}{\|x_0\|_{X_\eta^\d}} \\ &=  \sup_{x_0 \in X_1^\d \setminus \{0\}} \frac{|\langle u(T - \cdot), \Psi_T x_0 \rangle_{L^2(0, T; U)}|}{\|x_0\|_{X_\eta^\d}},
\end{aligned}
\end{equation}
where we used the density of $X_1^\d$ in $X_\eta^\d$ and also \cref{lem:dual-input-output}. Similarly, if $x_0 \in X_1^\d$, then $\Psi_T x_0 \in L^2(0, T; U)$ and
\begin{equation}
\label{eq:norm-Psi}
\|\Psi_T x_0\|_{L^2(0, T; U)} = 
 \sup_{u \in L^2(0, T; U) \setminus \{0\}} \frac{|\langle \Phi_T [u(T - \cdot)], x_0 \rangle_{X_{-\eta}, X_\eta^\d}|}{\|u\|_{L^2(0, T; U)}}.
\end{equation}
We note that the map $u \mapsto u(T - \cdot)$ preserves the $L^2$-norm and we deduce from \cref{eq:norm-Phi,eq:norm-Psi} that if \cref{it:adm-eta-fra} (resp.\ \cref{it:adm-eta}) holds for some $K, T > 0$, then \cref{it:input-frac} (resp.\ \cref{it:input-int}) holds with the same constants, and vice versa.


$\cref{it:alt-adm} \iff \cref{it:smooth-input}$:
Let $T > 0$. 
We use \cref{lem:dual-input-output} again to obtain the following equalities:
\begin{subequations}
\label{eq:norm-Phi-Psi-bis}
\begin{align}
&\|\Phi_T u \|_X = \sup_{x_0 \in X_1^\d \setminus \{0\}} \frac{|\langle \Psi_T x_0, u(T - \cdot) \rangle_{H^{-\eta}(0, T; U), H^\eta_0(0, T; U)}|}{\|x_0\|_X}, && u \in H^1_0(0, T; U); \\
& \|\Psi_T x_0\|_{H^{-\eta}(0, T; U)} = 
\sup_{u \in H^1_0(0, T; U) \setminus \{0\}} 
\frac{|\langle x_0, \Phi_T [u(T - \cdot)] \rangle_X|}{\|u\|_{H^\eta(0, T; U)}},
&& x_0 \in X_1^\d.
\end{align}
\end{subequations}
Using interpolation,
we  see that the map $u \mapsto u(T - \cdot)$ is bounded with respect to the $H^\eta$-norm,  
and \cref{eq:norm-Phi-Psi-bis} shows that
\cref{it:alt-adm} in \cref{theo:inter-group} is equivalent to the condition
\begin{equation}
\label{eq:X-H-eta-zero}
\|x(T)\|^2_X \leq K \|u\|_{H^{\eta}(0, T; U)}, \quad u \in H^1_0(0, T; U),
\end{equation}
where $x$ solves the inhomogeneous Cauchy problem \cref{eq:in-cauchy-zero}. On the other hand,
\cref{it:smooth-input} in \cref{theo:input-group}  is almost the same as \cref{eq:X-H-eta-zero}, but for $u \in H^1(0, T; U)$ instead of $H^1_0(0, T; U)$. It therefore remains to show that \cref{eq:X-H-eta-zero} actually implies \cref{it:smooth-input}.
\begin{itemize}[wide]
\item
\emph{First case: $0 \leq \eta \leq 1/2$.} By density of $\D(0, T; U)$ in $H^\eta(0, T; U)$ (see \cref{coro:charac-Hs-0}),
 we have
 $H^\eta(0, T; U) = H^\eta_0(0, T; U)$ and there is nothing left to prove.
\item
\emph{Second case: $1/2 < \eta \leq 1$.} We recall from trace theory that the map $u \mapsto u(\tau)$ is well-defined and continuous from $H^\eta(0, T; U)$ into $U$ for any $0 \leq \tau \leq T$; see, e.g., \cite[Corollary 27]{Sim90}. Furthermore, according to \cref{coro:charac-Hs-0} and \cref{rem:finite-interval}, $H_0^\eta(0, T; U)$ is precisely the space of those $u \in H^\eta(0, T; U)$ such that $u(0) = u(T) = 0$. 
Let $u \in H^\eta(0, T; U)$. We may write $u = u_1 + u_2$, where $u_2$ is the unique affine function such that $u_2(0) = u(0)$ and $u_2(T) = u(T)$. Then, by construction $u_1 \in H^\eta_0(0, T; U)$ and, by \cref{eq:X-H-eta-zero},
$
\|\Psi_Tu_1\|_{X} \leq K \|u_1\|_{H^\eta(0, T; U)} \leq K \|u\|_{H^\eta(0, T; U)} +  K\|u_2\|_{H^\eta(0, T; U)}
$.
On the other hand, because $u_2$ is affine,
\begin{equation}
\label{eq:u2-affine}
\|u_2\|_{H^\eta(0, T; U)} \leq K \|u_2\|_{H^1(0, T; U)}  \leq K' \|u(0)\|_U + K' \|u(T)\|_U \leq K'' \|u\|_{H^\eta(0, T; U)}.
\end{equation}
Furthermore, recall from \cref{eq:conv-ipp} that
\begin{equation}
\label{eq:sol-aff}
\Phi_T u_2 = S_TA^{-1}Bu(0) -A^{-1}Bu(T) + \int_0^T S_{T-t}A^{-1}B\dot{u}_2(t) \, \d t.
\end{equation}
Using $A$-boundedness of $B$ and \cref{eq:u2-affine}, we deduce from \cref{eq:sol-aff} that
$
\|\Phi_T u_2\|_X \leq K \|u_2\|_{H^1(0, T; U)} \leq K' \|u\|_{H^\eta(0, T; U)}
$.
After summation and use of the triangle inequality we are left with the desired estimate  $\|\Phi_T u\|_X \leq K \|u\|_{H^\eta(0, T; U)}$.
\end{itemize}

$\cref{it:FQ-cond} \iff \cref{it:input-FQ}$: The equivalence readily follows from \cref{eq:id-dual-op-norm}.
\end{proof}


\subsection{Proof of \cref{theo:transfer-function}}

This section contains the proof of \cref{theo:transfer-function}. At this point, \cref{theo:inter-group,theo:input-group} are proved and will be used in the arguments. In particular, recalling \cref{def:eta-input} we will say that the input operator $B$ is $\eta_1$-admissible if any of the (equivalent) conditions \cref{it:input-frac,it:input-int,it:smooth-input} of \cref{theo:input-group} is satisfied with parameter $0 \leq \eta_1 \leq 1$. Similarly, in accordance with \cref{def:eta-output}, the output operator $C$ is said to be $\eta_2$-admissible if \cref{it:adm-eta-fra,it:adm-eta,it:alt-adm} are verified with parameter $0 \leq \eta_2 \leq 1$.

First we need to take care of the dependence on the real numbers $\sigma$ in part \cref{it:FQ-IO} of \cref{theo:transfer-function}.
\begin{lemma}
\label{lem:sigma-IO}
Assume that $B$ and $C$ are $\eta_1$- and $\eta_2$-admissible, respectively.
Let $\sigma, \sigma' > \sigma_0$. Then,
\begin{equation}
\|C(\sigma + \i \omega - A)^{-1}B\|_{\L(U, Y)} = \|C(\sigma' + \i \omega - A)^{-1}B\|_{\L(U, Y)} + O(|\omega|^{\eta_1 + \eta_2}), \quad \omega \in \R, \quad |\omega| \to + \infty.
\end{equation}
\end{lemma}
\begin{proof}
By the resolvent identity,
\begin{equation}
C(\sigma + \i \omega - A)^{-1}B = C(\sigma' + \i \omega - A)^{-1}B + (\sigma' - \sigma) C(\sigma + \i \omega - A)^{-1}(\sigma' + \i \omega - A)^{-1}B, \quad \omega \in \R.
\end{equation}
On the other hand, for all $\omega \in \R$,
\begin{equation}
  \label{eq:CB-C-B}
 \|C(\sigma + \i \omega - A)^{-1}(\sigma' + \i \omega - A)^{-1}B\|_{\L(U, Y)} \leq  \|C(\sigma + \i \omega - A)^{-1}\|_{\L(X, Y)} \|(\sigma' + \i \omega - A)^{-1}B\|_{\L(U, X)}.
\end{equation}
By \cref{theo:inter-group,theo:input-group}, the factors on the right-hand side of \cref{eq:CB-C-B} grow like $O(|\omega|^{\eta_2})$ and $O(|\omega|^{\eta_1})$, respectively, as $|\omega| \to + \infty$, and the result  follows immediately.
\end{proof}

\Cref{lem:sigma-IO} proves the ``for some (equivalently, all)'' part of \cref{it:FQ-IO}. Furthermore, 
 it is clear that replacing $A$ by $A - \sigma$ for $\sigma \in \R$ does not change \cref{it:L2-H-eta}. Therefore, under $\eta_1$- and $\eta_2$-admissibility of $B$ and $C$, for the proof of equivalence between \cref{it:L2-H-eta,it:FQ-IO}, we may assume that $\sg{S}$ is uniformly bounded and that $0 \in \rho(A)$.


As indicated previously, the main idea behind \cref{theo:transfer-function} is rather simple but the step of extending the finite-time condition \cref{it:L2-H-eta} into a suitable  estimate (\cref{prop:h-eta-est-io} below) on the real line is lengthy and contains some tedious technicalities. The reader may wish to skip these by assuming  \cref{prop:h-eta-est-io} in the first instance and going straight to the end of the present section.

Let us recall and introduce some notation.
 Given $T > 0$,
 the operator $\Phi_T \in \L(L^2(0, T; U), X_{-1})$ maps input signals $u$ defined on $(0, T)$ to the final state $x(T)$, where $x$ is the corresponding solution to $\dot{x} = Ax + Bu$, $x(0) = 0$.  For the sake of convenience, if $0 < t < T$, we shall write $\Phi_t u$ instead of $\Phi_t [u_{|(0, t)}]$, and we also set $\Phi_0 \triangleq 0$. 
 On the other hand, we define $\Psi_T \in \L(X_1, L^2(0, T; Y))$ to be the operator mapping an initial state $x_0$ into $Cx$, where $x$ solves $\dot{x} = Ax$, $x(0) = x_0$. Finally, as mentioned above, when $u \in \D(0, T; U)$, then the function $t \mapsto \Phi_t u$ belongs to $\C([0, T], Z)$. This allows us to define the input-output map $\Xi_T$ as follows:
\begin{equation}
\label{eq:IO-map}
\Xi_T u \triangleq (t \in [0, T] \mapsto C\Phi_t u), \quad u \in \D(0, T; U).
\end{equation}
Finally, we recall that the tilde denotes multiplication by $e^{-\cdot}$.
Before delving into the input-output problem, we prove a lemma that also complements \cref{theo:inter-group}. 
\begin{lemma}[Regularity shift]
\label{lem:reg-shift}
Let $0 \leq \eta_1, \eta_2 \leq 1$ and $T > 0$. If $C$ is $\eta_2$-admissible then $\Psi_T$ extends to a continuous mapping from $X_{-\eta_1}^\inte$ into $ H^{-\eta_1 - \eta_2}(0, T; Y)$.
\end{lemma}
\begin{proof}
Because $C$ is assumed to be $\eta_2$-admissible, we already know
that the map $x_0 \mapsto C\tilde{x}$ is continuous $X_1 \to H^{1-\eta_2}\hl{Y}$.
In order to avoid the trouble of interpolating between spaces of positive and negative order, we will start from smoother solutions and then differentiate to reach the $X^\inte_{-\eta_1}$-level. First, let $x_0 \in X_1$, $v_0 \triangleq A^{-1}x_0$, $x(t) \triangleq S_t x_0$ and $v(t) \triangleq S_t v_0$, $t \geq 0$, so that $\dot{v} = x$. By $A$-boundedness of $C$ and the aforementioned continuity property,
we see that
\begin{equation}
\label{eq:w-x-x0-w0}
\|Cv\|_{H^1(0, T; Y)} + \|Cx\|_{H^{1-\eta_2}(0, T; Y)} \leq K \|x_0\|_{X_1} = K \|v_0\|_{X_2}.
\end{equation}
Using
 \cref{lem:diff-shift}, we infer from \cref{eq:w-x-x0-w0} that $\|C\tilde{v}\|_{H^{2-\eta_2}\hl{Y}} \leq K \|v_0\|_{X_2}$. Renaming the variables, we may reformulate this estimate as follows: 
\begin{equation}
x_0 \mapsto C{x} ~\mbox{is continuous}~X_2 \to H^{2-\eta_2}(0, T; Y).
\end{equation}
 We can then interpolate between $X_2 \to H^{2-\eta_2}\hl{Y}$ and $X_1 \to H^{1-\eta_1}\hl{Y}$ to obtain that
\begin{equation}
\label{eq:new-cont}
x_0 \mapsto C \tilde{x} ~\mbox{is continuous}~ X_{2-\eta_1}^\inte \to H^{2-\eta_1 - \eta_2}\hl{Y}.
\end{equation}
To complete the proof, given $x_0 \in X_1$, this time we let $v_0 = A^{-2}x_0$,  $x(t) \triangleq S_t x_0$ and $v(t) \triangleq S_t v_0$ for $t \geq 0$. Now $\ddot{v} = x$ and applying \cref{lem:diff-shift} twice (with \cref{rem:finite-interval}) leads to
\begin{equation}
\|Cx\|_{H^{-\eta_1 - \eta_2}(0, T; Y)} \leq K \|Cv\|_{H^{2-\eta_1 - \eta_2}(0, T; Y)} \leq K' \|v_0\|_{X^\inte_{2-\eta_1}} = K'\|x_0\|_{X^\inte_{-\eta_1}},
\end{equation}
where we applied \cref{eq:new-cont} to the $v$-variable and also used that $A^{-2}$ is an isometric isomorphism
from $X^\inte_{-\eta_1}$ onto  $X^\inte_{2 - \eta_1}$.
\end{proof}

With \cref{lem:reg-shift} in hand, let us summarise the different continuity properties of the input, output and input-ouput maps that we will use in the proof:
\begin{itemize}
  \item Given $0 \leq \eta_1 \leq 1$,  if the control operator $B$ is $\eta_1$-admissible then $\Phi_T$ is continuous from $L^2(0, T; U)$ into $X_{-\eta_1}^{\inte}$ for any $T > 0$;
  \item Given $0 \leq \eta_2 \leq 1$, if the observation operator $C$ is $\eta_2$-admissible, then $\Psi_T$ is continuous from $X_{-\eta_1}^\inte$ into $H^{-\eta_1 - \eta_2}\hl{Y}$ for any $T > 0$;
  \item If $0 \leq \eta_1, \eta_2 \leq 1$ and $T > 0$ are such that the input-ouput hypothesis \cref{it:L2-H-eta} is satisfied, then $\Xi_T$ defined by \cref{eq:IO-map} extends to a bounded linear operator from $L^2(0, T; U)$ into $H^{-\eta_1 -\eta_2}(0, T; Y)$.
\end{itemize}
From now on, the variable $x$ will be used to denote the unique (in general $X_{-1}$-valued) solution to the inhomogeneous Cauchy problem with zero initial condition
\begin{equation}
\label{eq:CPZ}
\dot{x} = Ax + Bu, \quad x(0) = 0,
\end{equation}
where the input $u$ is assumed to be at least locally integrable.
As pointed out before, when $u \in \D\hl{U}$, then $x \in \C(\R^+, Z)$; a stronger statement is given in \cref{lem:u-cut-Z} below.
 Our first goal here is to establish the following proposition.
\begin{prop}[Windowed estimate]
\label{prop:L2-H-eta-cut}
Let $0 \leq \eta_1,\eta_2 \leq 1$.
 If condition \cref{it:L2-H-eta} holds for some $T > 0$,  under $\eta_1$- and $\eta_2$-admissibility of $B$ and $C$, there exists $K > 0$ such that
\begin{equation}
\label{eq:L2-H-eta-cut}
\sum_{n=0}^{+\infty} e^{-2nT} \| C\tilde{x}(nT + \cdot) \|^2_{H^{-\eta_1 - \eta_2}(0, T; Y)}  \leq K \|\tilde{u}\|^2_{L^2\hl{U}}, \quad u \in \D\hl{E}.
\end{equation}
\end{prop}
\Cref{prop:L2-H-eta-cut} is the first step towards an $H^{-\eta_1 -\eta_2}$-estimate of 
 $C\tilde{x}$ in terms of the $L^2$-norm of $\tilde{u}$.
Its proof makes use of two simple lemmas that are given below. Fix $T > 0$.
In order to simplify notation, given an input $u$ and the corresponding solution $x$, we will write
\begin{equation}
\label{eq:notation-n}
x_n \triangleq x(\cdot + nT), \quad u_n \triangleq u(\cdot + nT), \quad n \in \N.
\end{equation}
\begin{lemma}
\label{lem:formula-sol-n}
Let $u \in L^2\hl{U}$. Then,
\begin{equation}
\label{eq:formula-sol-n}
x_n(t) = \Phi_{t}u_n + \sum_{k = 0}^{n - 1} S_{kT + t} \Phi_T u_{n-k-1}, \quad 0 \leq t \leq T, \quad n \in \N.
\end{equation}
\end{lemma}
\begin{proof}
This is a straightforward consequence of the variation of constants formula \cref{eq:conv}.
\end{proof}
Next, 
let $\U$ be the subspace 
 of $L^2\hl{U}$ comprised of all those $u \in \D\hl{U}$ such that $u_n \in \D(0, T; U)$ for all $n \in \N$. It is easy to see that $\U$ is dense in $L^2\hl{U}$.

\begin{lemma}
\label{lem:u-cut-Z}
Let $u \in \D\hl{U}$.
Then, $x \in \C(\R^+, Z) \cap L^\infty\hl{Z}$. Furthermore, if $u \in \U$, then for all $n \in \N$, the function $t \mapsto \Phi_t u_n$ belongs to $\C([0, T], Z)$, and also $\Phi_T u_n \in X_1$.
\end{lemma}
\begin{proof}
Here $u$ is smooth and satisfies $u(0) = \dot{u}(0) = 0$. Using the formula \cref{eq:conv-ipp} with an additional integration by parts leads to
\begin{equation}
x(t) =  - A^{-1}Bu(t) - A^{-2}B\dot{u}(t)  +  \int_0^t S_{t-s}A^{-2}B\ddot{u}(s)  \, \d s, \quad t \geq 0.
\end{equation}
Thus, boundedness of $x$ in $Z$ follows from  $A^{-1}B \in \L(U, Z)$, $A^{-2}B \in \L(U, X_1)$ and the continuous embedding $X_1 \hookrightarrow Z$, together with boundedness of the semigroup and the fact that $u$ is smooth and compactly supported. Furthermore,
$u_n \in \D(0, T; U)$ by definition of $\U$, which implies continuity of the function $t \mapsto \Phi_t u_n$
from $[0, T]$ into $Z$.
Since every $u_n$ vanishes at the end points $t = 0$ and $t =T$, using the variation of constants formula again we see that
\begin{equation}
\Phi_T u_n = \int_0^T S_{T-s} A^{-2}B \ddot{u}(s) \, \d s
\end{equation}
which in turn proves  that $\Phi_T u_n \in X_1$ for all $n \in \N$.
\end{proof}
\begin{rem}
Note that $\Phi_t u_n$ need not be in $X_1$ for $0 < t < T$.
\end{rem}

As a consequence of \cref{lem:u-cut-Z},  when $u \in \U$, we can apply the operator $C$ to each term of the formula \cref{eq:formula-sol-n} in a classical, pointwise sense. We can now give the proof of \cref{prop:L2-H-eta-cut}, which combines our technical machinery with some ideas from the system-theoretic literature, namely a convolution argument seen in, e.g., \cite[Theorem 2.5.4]{Sta05book}.

\begin{proof}[Proof of \cref{prop:L2-H-eta-cut}] We will write $\eta \triangleq \eta_1 + \eta_2$ and assume that $T > 0$ is as in \cref{it:L2-H-eta}. By density, it is enough the prove the result for $u \in \U$.
First, as noted in the proof of \cref{coro:ii-to-iii},
 $\|C \tilde{x}_n \|_{H^{-\eta}(0, T; Y)} \leq K \|C x_n\|_{H^{-\eta}(0, T; Y)}$  for all $n \in \N$. Next, recalling the formula \cref{eq:formula-sol-n} from \cref{lem:formula-sol-n}, we have
\begin{equation}
\label{eq:formula-Cx-est}
e^{-nT} \|C {x}_n \|_{H^{-\eta}(0, T; Y)} \leq e^{-nT} \|\Xi_T u_n \|_{H^{-\eta}(0, T; Y)} + \sum_{k=0}^{n-1} \|\Psi_T S_{kT}\Phi_T u_{n - k - 1}\|_{H^{-\eta}(0, T; Y)}, \quad n \in \N.
\end{equation}
By $\eta_2$-admissibility of $C$ and \cref{lem:reg-shift},  for all $0 \leq k \leq n- 1$, 
\begin{equation}
\label{eq:eta2-C-k}
\|\Psi_T S_{kT} \Phi_T u_{n-k-1}\|_{H^{-\eta}(0, T; Y)} \leq K \|S_{kT}\Phi_T u_{n-k-1}\|_{X_{-\eta_1}^\inte} \leq K' \|\Phi_Tu_{n-k-1}\|_{X_{-\eta_1}^\inte},
\end{equation}
where we also used that $\sg{S}$ extends to a uniformly bounded semigroup on $X_{-\eta_1}^\inte$.
 Furthermore, by $\eta_1$-admissibility of $B$, for all $0 \leq k \leq n- 1$, we have
\begin{equation}
\label{eq:eta1-B-n}
\begin{aligned}
e^{-nT}\|\Phi_Tu_{n-k-1}\|_{X_{-\eta_1}^\inte} &\leq K e^{-nT} \|u_{n-k-1}\|_{L^2(0, T; U)} \\
&= Ke^{-T} e^{-kT} e^{-(n - k -1)T} \|u_{n-k-1}\|_{L^2(0, T; U)}.
\end{aligned}
\end{equation}
Also, since condition \cref{it:L2-H-eta} is assumed, the bound
$\|\Xi_T u_{n}\|_{H^{-\eta}(0, T; Y)} \leq K \|u_n\|_{L^2(0, T; U)}$, which relates outputs and inputs, holds uniformly in  $n \in \N$.
Plugging this estimate, together with \cref{eq:eta2-C-k,eq:eta1-B-n}, into \cref{eq:formula-Cx-est} leads to
\begin{equation}
\label{eq:formula-Cx-final}
e^{-nT} \|C {x}_n \|_{H^{-\eta}(0, T; Y)} \leq K e^{-nT} \|u_n\|_{L^2(0, T; U)} + K \sum_{k = 0}^{n-1} e^{-kT} e^{-(n-k-1)T} \|u_{n-k-1}\|_{L^2(0, T; U)}
\end{equation}
for all $n \in \N$.
By means of a discrete Young's convolution inequality, we deduce from \cref{eq:formula-Cx-final} that
\begin{equation}
\sum_{n=0}^{+\infty} e^{-2nT} \|C {x}_n \|_{H^{-\eta}(0, T; Y)}^2 \leq K^2 \left( \sum_{n=0}^{+\infty} e^{-nT} \right)^{2} \sum_{n=0}^{+\infty} e^{-2nT} \|u_n\|^2_{L^2(0, T; U)} \leq K' \|\tilde{u}\|^2_{L^2\hl{U}},
\end{equation}
as required.
\end{proof}

We are now able to turn the finite-time property of \cref{it:L2-H-eta} into a half-line estimate.
\begin{coro}[Half-line estimate]
\label{coro:L2-H-eta}
In the setting of \cref{prop:L2-H-eta-cut},
there exists $K > 0$ such that
\begin{equation}
\|C\tilde{x}\|_{H^{-\eta_1 - \eta_2}\hl{Y}} \leq K \|\tilde{u}\|_{L^2\hl{U}}, \quad u \in \D(0, +\infty; U).
\end{equation}
\end{coro}

\begin{proof}
Let $\eta \triangleq \eta_1 + \eta_2$. 
We write
\begin{subequations}
\begin{align}
&x_{n + 1/2}(t) \triangleq x((n+1/2)T + t),   &&0 \leq t \leq T, \quad n \in \N, \\ &x_{1/2}(t) \triangleq x(T/2 + t),\quad  u_{1/2}(t) \triangleq u(T/2 + t), && t \geq 0.
\end{align}
\end{subequations}
 By \cref{lem:localize-shift}, it suffices to show that, for all $u \in \D\hl{U}$,
\begin{equation}
\label{eq:Cx-u-shift}
\sum_{n=0}^{+\infty} e^{-2nT} \|C \tilde{x}_n \|^2_{H^{-\eta}(0, T; Y)} + \sum_{n=0}^{+\infty} e^{-2(n+1/2)T} \|C\tilde{x}_{n+1/2}\|^2_{H^{-\eta}(0, T; Y)} \leq K \|\tilde{u}\|^2_{L^2\hl{U}}.
\end{equation}
The estimate for the first sum in \cref{eq:Cx-u-shift} follows directly from \cref{prop:L2-H-eta-cut}. To handle the second sum, we observe that $x_{1/2}$ solves the inhomogeneous Cauchy problem
\begin{equation}
\dot{x}_{1/2} = Ax_{1/2} + Bu_{1/2}, \quad x_{1/2}(0) =x(T/2),
\end{equation}
so that, formulating the variation of constants formula in terms of the operators $S_t$ and $\Phi_t$,
\begin{equation}
\label{eq:dec-v-w}
x_{1/2}(t) = S_t \Phi_{T/2}u + \Phi_t u_{1/2} 
\quad t \geq 0.
\end{equation}
Let $v(t) \triangleq S_t\Phi_{T/2}u$ and $w(t) \triangleq \Phi_tu_{1/2}$ for $t \geq 0$.
In accordance with \cref{eq:notation-n}, 
the subscript $n$ will indicate translation by $nT$ and restriction to $[0, T]$. It is immediate that
\begin{equation}
\sum_{n=0}^{+\infty} e^{-2(n+1/2)T} \|C\tilde{x}_{n+1/2}\|^2_{H^{-\eta}(0, T; Y)} \leq K \sum_{n=0}^{+\infty} e^{-2nT} \left(\|C\tilde{v}_n\|^2_{H^{-\eta}(0, T; Y)} + \|C\tilde{w}_n\|^2_{H^{-\eta}(0, T; Y)}\right).
\end{equation}
Since $v$ is a semigroup orbit,  $\eta_2$-admissibility of $C$, \cref{lem:reg-shift} and boundedness of the semigroup on $X_{-\eta_1}^\fra$ imply that
\begin{equation}
\label{eq:est-Cvn}
\|C\tilde{v}_n\|^2_{H^{-\eta}(0, T; Y)} \leq K \|C{v}_n\|^2_{H^{-\eta}(0, T; Y)} \leq K' \| S_{nT} \Phi_{T/2}u\|^2_{X^\inte_{-\eta_1}} \leq K'' \|\Phi_{T/2}u\|^2_{X^\inte_{-\eta_1}}, \quad n \in \N.
\end{equation}
Summing \cref{eq:est-Cvn} over $\N$ and using $\eta_1$-admissibility of $B$ yield
\begin{equation}
\label{eq:est-Cv}
\sum_{n=0}^{+\infty} e^{-2nT} \|C\tilde{v}_n\|^2_{H^{-\eta}(0, T; Y)} \leq K \| \Phi_{T/2}u\|^2_{X^\inte_{-\eta_1}}
\leq K' \|u\|^2_{L^2(0, T/2; U)} \leq K' \|\tilde{u}\|^2_{L^2\hl{U}}.
\end{equation}
Let us now deal with $w$.  Note that $u_{1/2}$ need not be in $\D\hl{U}$; nevertheless we may temporarily replace it with approximations $f_\eps \in \D\hl{U}$ that satisfy $f_\eps \to u_{1/2}$ in $L^2\hl{U}$ as $\eps \to 0$. After applying \cref{prop:L2-H-eta-cut} and passing to the limit, we obtain
\begin{equation}
\label{eq:est-Cw}
\sum_{n=0}^{+\infty} e^{-2nT} \|C\tilde{w}_n\|^2_{H^{-\eta}(0, T; Y)} \leq K \|\tilde{u}_{1/2}\|^2_{L^2\hl{U}} \leq K \|\tilde{u}\|^2_{L^2\hl{U}},
\end{equation}
and we complete the proof of \cref{eq:Cx-u-shift} by summing \cref{eq:est-Cv,eq:est-Cw}.
\end{proof}
Finally, we are in a position to deduce the required Bessel potential space estimate from the half-line property of \cref{coro:L2-H-eta}.
\begin{prop}[$H^{-\eta}(\R, Y)$-estimate]
\label{prop:h-eta-est-io}
Let $0 \leq \eta_1,\eta_2 \leq 1$.
 Assuming \cref{it:L2-H-eta}, $\eta_1$-admissibility of $B$ and $\eta_2$-admissibility of $C$, there exists $K > 0$ such that
\begin{equation}
\label{eq:freq-est-y-u}
\int_\R \frac{1}{(1 + \omega^2)^{\eta_1 + \eta_2}} \|C\hat{x}(1 + \i \omega)\|^2_Y \, \d \omega \leq K \|\tilde{u}\|^2_{L^2\hl{U}}, \quad u \in \D\hl{U}.
\end{equation}
\end{prop}
\begin{proof}
Let $\eta \triangleq \eta_1 + \eta_2$. The argument depends on the value of $\eta$.
\begin{itemize}[wide]
\item \emph{First case: $0 \leq \eta \leq 1/2$.} By \cref{coro:ext-zero}, the extension by zero is continuous $H^{-\eta}\hl{Y} \to H^{-\eta}(\R, Y)$ and we may  deduce \cref{eq:freq-est-y-u} directly from the estimate $\|C\tilde{x}\|_{H^{-\eta}\hl{Y}} \leq K \|\tilde{u}\|_{L^2\hl{U}}$ provided by \cref{coro:L2-H-eta}.
\item \emph{Second case: $1/2 < \eta < 3/2$.}
Given $u \in \D\hl{U}$,  consider the twice integrated variable $z$ defined by
$z(t) \triangleq \int_0^t \int_0^r x(s) \, \d s \, \d r$ for  $t \geq 0$.
Then
\begin{equation}
\label{eq:input-z-lap}
C\hat{z}(p) = \frac{1}{p^2} C\hat{x}(p) = -\frac{1}{p^2} CA^{-1}B\hat{u}(p) - \frac{1}{p} CA^{-2}B\hat{u}(p) + CA^{-1}(p - A)^{-1}A^{-1}B\hat{u}(p), \quad \re p > 0.
\end{equation}
Here we recall that $CA^{-1} \in \L(X, Y)$, $A^{-1}B \in \L(U, X)$,
and also $CA^{-1}B \in \L(U, Y)$ by \cref{eq:hyp-Z}. Thus, with Plancherel's theorem we readily infer from \cref{eq:input-z-lap} that 
\begin{equation}
\label{eq:preliminary-z-L2}
\|C\tilde{z}\|^2_{L^2\hl{Y}} \leq K \|\tilde{u}\|_{L^2\hl{U}}.
\end{equation}
Next, we wish to estimate the $H^{2-\eta}$-norm of $\tilde{z}$ in terms of the $H^{-\eta}$-norm of $C\tilde{x}$. To do so  we implement an ``elliptic regularity'' argument. Let $L$ be the operator on $L^2\hl{Y}$ given by $L \triangleq - (\d^2 / \d t^2) - 2(\d / \d t) +4$ with $\dom (L) \triangleq H^2\hl{Y} \cap H^1_0\hl{Y}$. Note that $L^\ast = - (\d^2 / \d t^2) + 2(\d / \d t) +4$ with $\dom(L^\ast) = \dom(L)$. In what follows, $y$ denotes an arbitrary element of $\dom(L)$. First, pairing $Ly$ and $y$ leads to
\begin{equation}
\|\dot{y}\|^2_{L^2\hl{Y}} - 2\langle\dot{y}, y\rangle_{L^2\hl{Y}} + 4 \|y\|^2_{L^2\hl{Y}} = \langle Ly, y\rangle_{L^2\hl{Y}}.
\end{equation}
so that, by Young's inequality,
\begin{subequations}
\label{eq:H2-H1-apriori}
\begin{align}
\label{eq:apriori-H1}
&\frac{1}{4}\|\dot{y}\|^2_{L^2\hl{Y}} + \frac{7}{4}\|y\|^2_{L^2\hl{Y}} \leq \|Ly\|^2_{H^{-1}\hl{Y}}, \\
\label{eq:apriori-H2}
&\frac{1}{2}\|\dot{y}\|^2_{L^2\hl{Y}} + \frac{3}{2}\|y\|^2_{L^2\hl{Y}} \leq \frac{1}{2} \|Ly\|^2_{L^2\hl{Y}}.
\end{align}
\end{subequations}
Note that \cref{eq:H2-H1-apriori} also holds with $L$ replaced by $L^\ast$. Thus, both $L$ and $L^\ast$ are coercive and in particular $L^{-1}$ is well-defined and continuous from $L^2\hl{Y}$ onto $\dom(L)$. In fact, \cref{eq:apriori-H1} shows that $L^{-1}$ and $(L^\ast)^{-1}$ extend to bounded linear operators from $H^{-1}\hl{Y}$ onto $H^1_0\hl{Y}$. On the other hand, \cref{eq:apriori-H2} allows us to prove that the graph norms of $L$ and $L^\ast$ are equivalent to the $H^2$-norm. We now aim to establish the following property:
\begin{equation}
\label{eq:req-cont-L-1}
L^{-1} ~\mbox{is continuous}~H^{-\eta}\hl{Y} \to H_0^{2-\eta}\hl{Y}.
\end{equation}
If $\eta = 1$, \cref{eq:req-cont-L-1} is already proved.
Assume then that $1/2 < \eta < 1$. By interpolation,
\begin{equation}
\label{eq:int-L-eta}
L^{-1}~\mbox{is continuous}~[L^2\hl{Y}, H^{-1}\hl{Y}]_{ \eta} \to [\dom(L), H^{1}_0\hl{Y}]_{\eta}.
\end{equation}
First,
\begin{equation}
\begin{aligned}
[L^2\hl{Y}, H^{-1}\hl{Y}]_{ \eta} &= ([H^1_0\hl{Y}, L^2\hl{Y}]_{1-\eta})^\ast \\&= (H^\eta_0\hl{Y})^\ast = H^{-\eta}\hl{Y},
\end{aligned}
\end{equation}
where we used \cite[Theorem 3.7]{ChaHew15} to pass to antiduals and then \cref{lem:int-H10-L2} to characterise the interpolation space on the right-hand side. Second,
we recall that $\dom(L)$ equipped with the graph norm coincides algebraically and topologically with $H^2\hl{Y} \cap H^1_0\hl{Y}$ equipped with the $H^2$-norm. Therefore by \cref{lem:yet-another-interpolation-lemma} the interpolation space at the right-hand side of \cref{eq:int-L-eta} is precisely $H^{2-\eta}_0\hl{Y}$, and \cref{eq:req-cont-L-1} is proved. For the case $1 < \eta < 3/2$, we notice that by duality \cref{eq:req-cont-L-1} is equivalent to:
\begin{equation}
(L^\ast)^{-1} ~\mbox{is continuous}~ H^{\eta-2}\hl{Y} \to H^\eta_0\hl{Y},
\end{equation}
which we can prove exactly as in the previous case since $L^\ast$ enjoys the same properties as $L$ and $\eta$ is such that $1/2 < 2- \eta < 1$. Having established \cref{eq:req-cont-L-1}, we may return to our $z$-variable. We see that $- L C\tilde{z} = C\tilde{x} - 3C\tilde{z}$, and by \cref{eq:preliminary-z-L2,eq:req-cont-L-1} we have
\begin{equation}
\|C\tilde{z}\|_{H^{2-\eta}\hl{Y}} \leq K \|C\tilde{x}\|_{H^{-\eta}\hl{Y}} + K \|C\tilde{z}\|_{L^2\hl{Y}} \leq K' \|\tilde{u}\|_{L^2\hl{U}}.
\end{equation}
 Now recall from \cref{coro:ext-zero} that  extension by zero is continuous $H^{2-\eta}_0\hl{Y} \to H^{2-\eta}(\R, E)$. This allows us to deduce that
$
\|C\tilde{z}\|_{H^{2-\eta}(\R, Y)} \leq K \|\tilde{u}\|_{L^2\hl{U}}
$,
which in turn leads to the required estimate \cref{eq:freq-est-y-u} via the Laplace representation \cref{eq:input-z-lap} of $z$.
\item \emph{Third  case: $3/2 < \eta \leq 2$.} We use the $z$-variable again but in a slightly different way. 
Because $\eta \geq 1$, it follows from \cref{eq:preliminary-z-L2} that
\begin{equation}
\|(\d / \d t)C\tilde{z}\|_{H^{-\eta}\hl{Y}} \leq K \|(\d / \d t)C\tilde{z} \|_{H^{-1}\hl{Y}} \leq K' \|\tilde{u}\|_{L^2\hl{U}}.
\end{equation}
Recall that $(\d^2 / \d t^2)C\tilde{z} = -2(\d/\d t)C\tilde{z} - C\tilde{z} + C\tilde{x}$.
By \cref{lem:diff-shift} we obtain
\begin{equation}
\label{eq:dtCz-1-eta}
\begin{aligned}
\|(\d / \d t)C\tilde{z}\|_{H^{1-\eta}\hl{Y}}
&\leq K \|(\d^2/\d t^2)C\tilde{z}\|_{H^{-\eta}\hl{Y}} + K \|(\d/\d t)C\tilde{z}\|_{H^{-\eta}\hl{Y}}
\\
&\leq K' \|\tilde{u}\|_{L^2\hl{Y}}.
\end{aligned}
\end{equation}
Applying \cref{lem:diff-shift} again yields
\begin{equation}
\|\tilde{z}\|_{H^{2-\eta}\hl{Y}} \leq K \|\tilde{z}\|_{H^{1-\eta}\hl{Y}} + K \|(\d / \d t)\tilde{z}\|_{H^{1-\eta}\hl{Y}} \leq K' \|\tilde{u}\|_{L^2\hl{Y}},
\end{equation}
where the $H^{1-\eta}$-norm can be controlled by the $L^2$-norm since  $\eta \geq 1$. By \cref{coro:ext-zero},
extension by zero is continuous $H^{2-\eta}\hl{Y} \to H^{2-\eta}(\R, Y)$ and we complete the proof just as in the previous case.
\item \emph{Fourth case: $\eta = 3/2$.} 
In this case \cref{eq:dtCz-1-eta} remains valid. We recall from \cref{coro:ext-zero} that the extension by zero map (denoted by $\pi$) is continuous $H^{-1/2}\hl{Y} \to H^{-1/2}(\R, Y)$. Therefore
\begin{equation}
\label{eq:ext-zero-dt-z}
\|\pi(\d / \d t)\tilde{z}\|_{H^{-1/2}(\R, Y)} \leq K \|\tilde{u}\|_{L^2\hl{U}}.
\end{equation}
Here we use the explicit notation $\pi$ to avoid confusion as the derivative of the extension by zero need not be the extension by zero of the derivative, although it is true here for $\tilde{z} \in H^{1}_0\hl{Y}$.
Then
\begin{equation}
\label{eq:lap-dt-z}
\L\left[ \dt{\tilde{z}}  \right](p) = \frac{p}{(1 + p)^2} C(1 + p -A)^{-1}B\hat{u}(1 + p), \quad \re p \geq 0,
\end{equation}
so that we can readily turn our estimate \cref{eq:ext-zero-dt-z} for $(\d / \d t)\tilde{z}$ into
the desired inequality \cref{eq:freq-est-y-u}. \qedhere
\end{itemize}
\end{proof}

\begin{rem}
When $0 \leq \eta < 1$, it it not clear whether one can derive an intermediate estimate of $(\d / \d t)\tilde{z}$ in $H^{-\eta}$-norm using merely $A$-boundedness of $B$ and $C$ or the technical hypothesis \cref{eq:hyp-Z}. This is why we used the \emph{ad hoc} elliptic argument in that case, as opposed to the more straightforward application of \cref{lem:diff-shift} in the other cases.
\end{rem}

\begin{proof}[Proof of \cref{theo:transfer-function}] We first prove $\cref{it:L2-H-eta} \implies \cref{it:FQ-IO}$ and then, simultaneously, $\cref{it:FQ-IO} \implies \cref{it:L2-H-eta}$ and the last statement in the theorem.

$\cref{it:L2-H-eta} \implies \cref{it:FQ-IO}$: Here we assume that $B$ and $C$ are  $\eta_1$- and $\eta_2$-admissible, respectively. Let $u \in \D\hl{U}$. By \cref{lem:u-cut-Z}, $x \in L^{\infty}\hl{Z}$ so $x$ is Laplace transformable as a $Z$-valued function and  $C\hat{x}(p) = C(p -A)^{-1}B\hat{u}(p)$ for all $p \in \Co$ with $\re p > 0$. In particular, the $Y$-valued function $F$ defined by
$
F(p) \triangleq p^{-\eta} C(p-A)^{-1}B\hat{u}(p)
$
is holomorphic and bounded on each set of the form $\{ p \in \Co, \re p \geq \sigma \}$, $\sigma > 0$. Therefore, by the Paley--Wiener theorem, $F$ is the Laplace transform of some measurable $Y$-valued function $f$ satisfying $e^{-\sigma \cdot}f \in L^2\hl{Y}$ for all $\sigma > 0$. Similarly as in the proof of \cref{theo:inter-group}, by Plancherel's theorem the $H^{-\eta}$-estimate \cref{eq:freq-est-y-u} provided by \cref{prop:h-eta-est-io} implies that $f$ must satisfy
$
\|\tilde{f}\|_{L^2\hl{Y}} \leq K \|\tilde{u}\|_{L^2\hl{U}}
$.

Since $u \mapsto \tilde{u}$ is a bijection on $\D\hl{U}$, it then follows  that the map $\tilde{u} \mapsto \tilde{f}$ is well-defined and extends as a bounded linear operator $\Theta : L^2\hl{U} \to L^2\hl{Y}$. Define defined the $\L(U, Y)$-valued function $G$ 
\begin{equation}
G(p) \triangleq \frac{1}{(1 + p)^\eta} C(1 + p - A)^{-1}B, \quad \re p \geq 0.
\end{equation}
We may check that $\Theta$ is exactly multiplication by $G$ in the Laplace domain: $\L[\Theta v](p) = G(p) \hat{v}(p)$ for all $p \in \Co$ with $\re p > 0$ and $v \in L^2\hl{U}$. By definition of $\Theta$, this is immediate when $v \in \D\hl{U}$, and it follows from a density argument using the Paley--Wiener theorem  when $v$ is merely in $L^2\hl{U}$. As a consequence of its Laplace multiplier representation, the operator $\Theta$ is also  \emph{shift invariant}, i.e., $\Theta[v](\cdot - \tau) = \Theta[v(\cdot - \tau)]$ for all $\tau \geq 0$ and $v \in L^2\hl{U}$. 

Now, recall that
shift invariant bounded operators on Hilbert-valued $L^2(0, +\infty)$-spaces are {characterised} by bounded holomorphic transfer functions on the open right-half plane; see \cite{FouSeg55,Wei94}.
More precisely, \cite[Theorem 3.1]{Wei94} provides a bounded holomorphic $\L(U, Y)$-valued function $H$ defined on the open right-half plane and  such that $\L[\Theta v](p)= H(p)v(p)$ for all $p \in \Co$ with $\re p > 0$ and $v \in L^2\hl{U}$. Using elements $v$ of the form $v(t) = e^{-t}v_0$, $v_0 \in U$, it is easy to see that $G$ and $H$ coincide on the open right-half plane. Therefore, we have proved that
\begin{equation}
\label{eq:G-est}
\sup_{\re p > 0} \|G(p)\|_{\L(U, Y)} < + \infty, \quad \mbox{and thus}\quad \sup_{\omega \in \R} \|G(\i \omega)\|_{\L(U, Y)} < +\infty,
\end{equation}
where we also used that $G$ is continuous on the closed right-half plane. The growth condition \cref{it:FQ-IO} for $C(1 + \i \omega - A)^{-1}B$ then follows from \cref{eq:G-est}.

%


$\cref{it:FQ-IO} \implies \cref{it:L2-H-eta}$: Here we no longer assume $\eta_1$- and $\eta_2$-admissibility of $B$ and $C$. We also drop the working hypothesis that $\sg{S}$ is uniformly bounded and $0 \in \rho(A)$. We shall prove that if \cref{eq:FQ-IO} holds for some $\sigma > \sigma_0$ then \cref{it:L2-H-eta} holds,
 which 
also gives $\cref{it:FQ-IO} \implies \cref{it:L2-H-eta}$.
Assuming \cref{eq:FQ-IO},
 $\omega \mapsto C(\sigma + \i \omega - A)^{-1}B$ is a Fourier multiplier between $L^2(\R, U)$ and $H^{-\eta}(\R, Y)$. Let $u \in \D(0, 1; U)$, which we extend with zero outside of $(0, 1)$, so that $u \in \D\hl{U}$. The Fourier transform of $e^{-\sigma \cdot} C x$ is precisely $\omega \mapsto C(\sigma + \i \omega - A)^{-1}B\hat{u}(\sigma + \i \omega)$
and the above multiplier property leads to
\begin{equation}
\|Ce^{-\sigma\cdot}{x}\|_{H^{-\eta}(\R)} \leq K \|e^{-\sigma \cdot}{u}\|_{L^2\hl{U}} \leq e^\sigma K \|u\|_{L^2(0, 1; U)}.
\end{equation}
Now \cref{it:L2-H-eta} follows from  the restriction property and continuity on $H^{-\eta}(0, 1; Y)$ of multiplication by $e^{\sigma \cdot}$.
\end{proof}
\section{A class of second-order systems}
\label{sec:second-order}

\subsection{Admissibility for second-order systems}
\label{sec:second-order-adm}

In this section, we present specialised versions of \cref{theo:inter-group,theo:input-group} for a class of abstract (wave-type) second-order systems studied in, e.g., \cite{Mil05,Mil12,ChiPau23,KleWan23arxiv}.


Let $L : \dom(L) \to H$ be a  nonnegative self-adjoint operator on the Hilbert space $H$. 
Fix $\eps > 0$ or take $\eps = 0$ if $0 \in \rho(L)$.
The spectral theory of self-adjoint operators  provides fractional powers of  $(L+\eps)^s$  for $s \in \R$. Define the Hilbert spaces\footnote{
We use the same scaling as in \cite{Mil05,Mil12} so that, in the prototype case where $L$ is a positive Laplacian with prescribed boundary conditions, the abstract spaces $H_s$ and the usual Sobolev spaces $H^s(\Omega)$ are somewhat comparable. Our spaces $H_s$ are the spaces $H_{s/2}$ in the notation of \cite{TucWei09book,ChiPau23,KleWan23arxiv}.
}
\begin{equation}
H_s \triangleq \dom((L + \eps)^{s/2}), \quad \|\cdot\|_{H_s} \triangleq \|(L+\eps)^{s/2}\cdot\|_H, \quad s \geq 0.
\end{equation}
For $s < 0$ we let $H_s$ be the completion of $H$ with respect to the norm $\|(L + \eps)^{s/2} \cdot\|_H$.
Different choices of $\eps$ lead to equivalent norms. For all real numbers $s_1 \geq s_2$, $H_{s_1}$ is continuously and densely embedded into  $H_{s_2}$. Also, after identifying $H$ with its antidual, we have $(H_s)^\ast = H_{-s}$ for all $s \geq 0$, with equality of norms. Furthermore,  $\{H_s\}_{s \in \R}$ constitutes a continuous scale of interpolation spaces.
Consider the second-order Cauchy problem
\begin{equation}
\label{eq:second-order-cauchy}
\ddot{w} + Lw = 0, \quad w(0) = w_0, \quad \dot{w}(0) = w_1.
\end{equation}
Initial data $(w_0, w_1)$ in $H_2 \times H_1$ (resp. $H_1 \times H$) generate classical solutions $w \in \C(\R^+, H_{2}) \cap \C^1(\R^+, H_1)$ (resp. mild or finite-energy solutions $w \in \C(\R^+, H_{1}) \cap \C^1(\R^+, H)$) to \cref{eq:second-order-cauchy}. The $(w, \dot{w})$-variables give rise to a strongly continuous group $\{S_t\}_{t\in \R}$ on the phase space $X \triangleq H_1 \times H$. 


Let $Y$ be another Hilbert space.  Given $Q_0 \in \L(H_2, U)$ and
 $Q_1 \in \L(H_1, U)$,
we are interested in outputs of the form $Q_0w$ and $Q_1 \dot{w}$. Typically $Q_0$ is allowed to be ``more unbounded'' than $Q_1$. We deal with both cases simultaneously for the sake of compactness; the reader interested in one or the other may of course set either $Q_0$ or $Q_1$ to be $0$.

\begin{theo}[Observation operators for second-order systems]
\label{theo:observation-SO}
 Let $0 \leq \eta \leq 1$.
Consider solutions $w$ to the second-order Cauchy problem \cref{eq:second-order-cauchy} with initial data $(w_0, w_1)$.
The following are equivalent:
\begin{enumerate}[label=(\roman*)]
  \item \label{it:SO-obs-sm} \emph{(Smoother data.)} There exist $K, T > 0$ such that
\begin{equation}
\|Q_0w\|_{L^2(0, T; Y)} + \|Q_1 \dot{w}\|_{L^2(0, T; Y)} \leq K \|w_0\|_{H_{1+\eta}} + K\|w_1\|_{H_\eta}, \quad (w_0, w_1) \in H_2 \times H_1;
\end{equation}
  \item \emph{(Distributional outputs.)}\label{it:SO-obs-do} There exist $K, T > 0$ such that
\begin{equation}
\|Q_0w\|_{H^{-\eta}(0, T; Y)} + \|Q_1 \dot{w}\|_{H^{-\eta}(0, T; Y)} \leq K \|w_0\|_{H_1} + K\|w_1\|_{H}, \quad (w_0, w_1) \in H_2 \times H_1;
\end{equation}
\item \emph{(Frequency-domain condition.)}\label{it:SO-obs-FQ}
\begin{equation}
\label{eq:SO-obs-FQ}
\left \{
\begin{aligned}
&\|Q_0((1 + \i \lambda)^2 + L)^{-1}\|_{\L(H, Y)} = O(|\lambda|^\eta),&& \\
&\|Q_1((1 + \i \lambda)^2 + L)^{-1}\|_{\L(H, Y)} = O(|\lambda|^{\eta - 1}), &&
\end{aligned}
 \right. \quad \lambda \in \R,\quad |\lambda| \to + \infty.
\end{equation}
\end{enumerate}
\end{theo}
\begin{rem}
The high-frequency bounds in \cref{eq:SO-obs-FQ} can be reformulated as ``resolvent conditions with variable coefficients'' in the terminology of \cite{Mil12}. For instance, the estimate related to $Q_1$ reads as follows:
there exist $K, \lambda_0 > 0$ such that
\begin{equation}
\|Q_1w\|^2_U \leq \frac{K (1 + \lambda^{2\eta})}{\lambda^2} \|(L  - \lambda^2)w\|^2_H + K(1 + \lambda^{2\eta}) \|w\|^2_H, \quad w \in H_2, \quad \lambda \in \R, \quad |\lambda| \geq \lambda_0.
\end{equation}
\end{rem}

Let $U$ be yet another Hilbert space and let $P \in \L(U, H_{-1})$ represent a control operator. We now consider the inhomogeneous second-order Cauchy problem
\begin{equation}
\label{eq:w-u-cauchy}
\ddot{w} + Lw = Pu, \quad w(0) = 0, \quad \dot{w}(0) = 0.
\end{equation}
Given $T > 0$, inputs $u \in L^2(0, T; U)$ produce solutions $w$ in $\C([0, T], H) \cap \C^1([0, T], H_{-1})$; furthermore, when $u \in H^1(0, T; U)$,
$w \in \C([0, T], H_1) \cap \C^1([0, T], H)$. This can be seen using the semigroup arguments of \cref{sec:adm-input}. 
\begin{theo}[Control operators for second-order systems]
\label{theo:control-SO}
Let $0 \leq \eta \leq 1$. Consider solutions $w$ to the second-order inhomogeneous 
Cauchy problem \cref{eq:w-u-cauchy} with controls $u$. The following are equivalent:
\begin{enumerate}[label=(\roman*')]
  \item \emph{($L^2$-inputs.)} There exist $K, T > 0$ such that
  \begin{equation}
  \|w(T)\|_{H_{1 - \eta}} + \|\dot{w}(T)\|_{H_{-\eta}} \leq K \|u\|_{L^2(0, T; U)}, \quad u \in H^1(0, T; U);
  \end{equation}
  \item \emph{(Smoother inputs.)} There exist $K, T > 0$ such that
  \begin{equation}
  \|w(T)\|_{H_{1}} + \|\dot{w}(T)\|_{H} \leq K \|u\|_{H^\eta(0, T; U)}, \quad u \in H^1(0, T; U);
  \end{equation}
  \item \emph{(Frequency-domain condition.)}
  \begin{equation}
  \label{eq:control-SO-FQ}
  \|((1 + \i \lambda)^2 + L)^{-1}P\|_{\L(U, H)} = O(|\lambda|^{\eta - 1}), \quad \lambda \in \R, \quad |\lambda| \to + \infty.
  \end{equation}
\end{enumerate}
\end{theo}

\Cref{theo:observation-SO,theo:control-SO} tell us that, even though the semigroup formulation of problems \cref{eq:second-order-cauchy,eq:w-u-cauchy} give rise to two state variables (or ``coordinates''),
 the admissibility properties of \cref{theo:inter-group,theo:input-group} are completely determined by the behaviour of a single quadratic operator pencil in $\L(U,H)$ or $\L(H, Y)$.

\begin{proof}[Proofs of \cref{theo:observation-SO,theo:control-SO}] The first part of the proof consists in recasting the second-order problems \cref{eq:second-order-cauchy,eq:w-u-cauchy} in the semigroup framework of \cref{sec:adm-semi,sec:main-res}.
The generator $A$ of the group $\{S_t\}_{t \in\R}$ associated with the second-order Cauchy problem \cref{eq:second-order-cauchy} posed on $X = H_1 \times H$ is given by
\begin{equation}
A = \begin{pmatrix}
0 & 1 \\ -L & 0
\end{pmatrix}, \quad \dom(A) = H_2 \times H_1.
\end{equation}
Up to norm equivalence, the associated Sobolev scales $\{X_s^\inte\}_{s \in \R}$ and $\{X_s^\fra\}_{s\in \R}$ coincide and are given by
$
X_s \triangleq H_{1+s} \times H_s
$.
In particular, $X_{-1} = H \times H_{-1}$. The resolvent of $A$ has the form
\begin{equation}
\label{eq:res-A-SO}
(p - A)^{-1} = \begin{pmatrix}
p(p^2 + L)^{-1} & (p^2 + L)^{-1} \\ -L(p^2 + L)^{-1} & p(p^2 + L)^{-1}
\end{pmatrix}
, \quad \re p > 0.
\end{equation}
In first-order form, the associated observation and control operators $C_0, C_1 \in \L(X_1, Y)$, $B \in \L(U, X_{-1})$ are given by
\begin{equation}
C_0 \triangleq \begin{pmatrix}
Q_0 & 0 
\end{pmatrix}, \quad 
C_1 \triangleq \begin{pmatrix}
0 & Q_1 
\end{pmatrix}, \quad B \triangleq \begin{pmatrix}
0 \\ P
\end{pmatrix}.
\end{equation}
We start by giving the proof of \cref{theo:observation-SO}.
It follows from \cref{eq:res-A-SO} that
\begin{equation}
\left\{
\begin{aligned}
&C_0(p - A)^{-1} = \begin{pmatrix} 
pQ_0(p^2 + L)^{-1} & Q_0(p^2 + L)^{-1}
\end{pmatrix}, \\
&C_1(p- A)^{-1} =
\begin{pmatrix} 
-Q_1 L(p^2 + L)^{-1} & p Q_1(p^2 + L)^{-1}
\end{pmatrix},
\end{aligned}
\right. \quad \re p > 0.
\end{equation}
Applying \cref{theo:input-group} to $C_0$ and $C_1$ as observation operators for the group $\{S_t\}_{t \in \R}$ shows that
 \cref{it:SO-obs-sm,it:SO-obs-do} in \cref{theo:observation-SO} are equivalent and
 amount to the following four estimates when $\lambda \in \R$, $|\lambda|\to +\infty$:
\begin{subequations}
\label{eq:4-lambda}
\begin{align}
\label{eq:Q0-gr}
&\|Q_0((1 + \i \lambda)^2 + L)^{-1}\|_{\L(H_1, Y)} = O(|\lambda|^{\eta-1}), 
&& \|Q_0((1+\i \lambda)^2 + L)^{-1}\|_{\L(H, Y)} = O(|\lambda|^\eta), \\
\label{eq:Q1-gr}
&\|Q_1L((1 + \i \lambda)^2 + L)^{-1}\|_{\L(H_1, Y)} = O(|\lambda|^\eta), 
&& \|Q_1((1 + \i \lambda)^2 + L)^{-1}\|_{\L(H, Y)} = O(|\lambda|^{\eta-1}).
\end{align}
\end{subequations}
In particular,  both \cref{it:SO-obs-sm,it:SO-obs-do} imply \cref{it:SO-obs-FQ}.
Assume now that \cref{it:SO-obs-FQ} holds, i.e., suppose that the frequency-domain estimates for $Q_0$ and $Q_1$ in \cref{eq:SO-obs-FQ}
are valid. We need to prove that the second estimate in \cref{eq:Q0-gr} (resp. \cref{eq:Q1-gr}) implies the first one.

Let us start with $Q_0$. 
The resolvent identity yields
\begin{equation}
\label{eq:res-id-Q0}
Q_0((1+\i \lambda)^2 + L)^{-1} = Q_0(1 + L)^{-1} + (2\i \lambda - \lambda^2)Q_0((1 + \i \lambda)^2 + L)^{-1}(1+L)^{-1}, \quad \lambda \in \R.
\end{equation}
Recall that $(1 + L)^{-1}$ is an isomorphism from $H$ onto $H_2$ and $Q_0(1 + L)^{-1} \in \L(H, Y)$. Plugging the $Q_0$-part of \cref{eq:SO-obs-FQ} into \cref{eq:res-id-Q0}, we see that $\|Q_0((1 + \i \lambda)^2 + L)^{-1}\|_{\L(H_2, Y)} = O(|\lambda|^{\eta-2})$ when $\lambda \in \R$, $|\lambda| \to + \infty$. Using interpolation between $H \to Y$ and $H_2 \to Y$, we deduce that $\|Q_0((1 + \i \lambda)^2 + L)^{-1}\|_{\L(H_1, Y)} = O(|\lambda|^{\eta-1})$,
as required.

Next, we consider $Q_1$. Recalling that $H_1 = \dom(L^{1/2})$, it suffices to prove that
\begin{equation}
\label{eq:goal-Q1}
\|Q_1((1 + \i \lambda)^{2} + L)^{-1}L^{1/2} \|_{\L(H, Y)} = O(|\lambda|^\eta), \quad \lambda \in \R, \quad |\lambda| \to +\infty.
\end{equation}
To do so, we start from the factorisation
$((1 + \i \lambda)^2 + L)^{-1} = (1 + \i \lambda + \i L^{1/2})^{-1}(1 + \i \lambda - \i L^{1/2})^{-1}$
valid for all $\lambda \in \R$.
After using the resolvent identity in the second factor, we obtain that, for all $\lambda \in \R$,
\begin{multline}
\label{eq:fact-res}
Q_1((1 + \i \lambda)^{2} + L)^{-1}L^{1/2} = Q_1 (1 - \i L^{1/2})^{-1}(1 + \i \lambda + \i L^{1/2})^{-1} L^{1/2}  \\ - \i \lambda Q_1((1 + \i \lambda)^2 + L)^{-1}(1 - \i L^{1/2})^{-1}L^{1/2}.
\end{multline}
Now, $Q_1(1 - \i L^{1/2})^{-1} \in \L(H, Y)$ and $(1- \i L^{1/2})^{-1}L^{1/2}$ is an isomorphism on $H$. Furthermore, using the spectral theorem and the functional calculus of the (self-adjoint) operator $L^{1/2}$, we can write
\begin{equation}
\label{eq:spe-res}
(1 + \i \lambda + \i L^{1/2})^{-1}L^{1/2}  =  \int_0^{+\infty} \frac{\mu}{1 + \i \lambda + \i \mu} \, \d E(\mu), \quad \lambda \in \R,
\end{equation}
where $E$ is a spectral measure, and deduce from \cref{eq:spe-res} that $\|(1 + \i \lambda + \i L^{1/2})^{-1}L^{1/2}\|_{\L(H)} \leq 1$ whenever $\lambda \geq 0$. This shows that the first term on the right-hand side of \cref{eq:fact-res} is uniformly bounded in $\L(H, Y)$-norm for $\lambda \geq 0$. By the hypothesis \cref{eq:SO-obs-FQ},  the $\L(H, Y)$-norm of the second term must grow like $O(\lambda^\eta)$ as $\lambda \to +\infty$. In turn, we see that the growth rate \cref{eq:goal-Q1} holds as $\lambda \to +\infty$. To show that \cref{eq:goal-Q1} is also valid as $\lambda \to - \infty$ and thereby complete the proof, we carry out an analogous argument, but this time we use the resolvent identity to expand the factor $(1 + \i \lambda + \i L^{1/2})^{-1}$ instead and employ the spectral theorem to  bound $(1 + \i \lambda - \i L^{1/2})^{-1}L^{1/2}$ uniformly for $\lambda \leq 0$. 

It remains to prove \cref{theo:control-SO}. We have
\begin{equation}
(p - A)^{-1}B = \begin{pmatrix}
(p^2 + L)^{-1}P \\
p(p^2 + L)^{-1}P
\end{pmatrix}, \quad \re p > 0,
\end{equation}
and by \cref{theo:input-group} it suffices to show that the $\L(U, H)$-estimate \cref{eq:control-SO-FQ} implies
\begin{equation}
\label{eq:control-H1-growth}
\|((1 + \i \lambda)^2 + L)^{-1}P\|_{\L(U, H_1)} = O(|\lambda|^\eta), \quad \lambda \in \R, \quad |\lambda| \to + \infty.
\end{equation}
Here we use a simple pairing argument. Let $\lambda \in \R$ with $|\lambda| \geq 1$ (say), $u \in U$ and $w \triangleq ((1 + \i \lambda)^2 + L)^{-1}Pu$. Then $
\langle Lw +  2 \i \lambda w - \lambda^2 w, w\rangle_H = \langle Pu, w \rangle_{H_{-1}, H}$ and taking the real part gives
$
\|L^{1/2}w\|^2_{H} -  \lambda^2 \|w\|_H^2 = \re \langle Pu, w\rangle_{H_{-1}, H}
$.
Applying the Cauchy--Schwarz and Young inequalities, we arrive at
\begin{equation}
\label{eq:H1-w}
\|w\|^2_{H_1} \leq K(\|u\|^2_U + \lambda^2\|w\|^2_{H}),
\end{equation}
where we used that $(\|L^{1/2}\cdot\|^2_H + \|\cdot\|^2_H)^{1/2}$ defines an equivalent norm on $H_1$. Rewriting \cref{eq:H1-w} in resolvent terms and plugging it into \cref{eq:control-SO-FQ} directly yields \cref{eq:control-H1-growth}.
\end{proof}


\subsection{Transfer functions for second-order systems}

We now give a second-order version of \cref{theo:transfer-function} which applies to transfer functions of the form
\begin{equation}
p \mapsto Q_0(p^2 + L)^{-1}P, \quad p \mapsto pQ_1(p^2 + L)^{-1}P.
\end{equation}
 Note that for velocity observation modelled by $Q_1 \in \L(H_1, Y)$, we automatically have $Q_1(p^2 + L)^{-1}P \in \L(U, Y)$ for $\re p > 0$.
However, 
for $Q_0$ we will make an additional assumption, similar to that of \cref{sec:TF-IO}: namely, there exists a Hilbert space $W$ satisfying $H_2 \hookrightarrow W \hookrightarrow H_1$ with continuous embeddings and
\begin{equation}
\label{eq:hyp-W-space}
Q_0 \in \L(W, Y), \quad \ran((\mu + L)^{-1}P) \subset W ~\mbox{for some (hence all)}~\mu \in \rho(L). 
\end{equation}
This implies  $Q_0(p^2 + L)^{-1}P \in \L(U, Y)$ for $\re p > 0$. We will also see that, given $T > 0$, when $u \in \D(0, T; U)$, the corresponding solution $w$ to \cref{eq:w-u-cauchy} belongs to $\C^1([0, T], W)$.


\begin{theo}[Transfer functions for second-order systems]\label{theo:TF-SO}
Let $0 \leq \eta_1, \eta_2 \leq 1$. Suppose that $P$ satisfies the conditions of \cref{theo:control-SO} with parameter $\eta_1$.  Suppose also that $Q_0$ and $Q_1$ satisfy the (equivalent) conditions of \cref{theo:observation-SO} with parameter $\eta_2$. Consider solutions $w$ to the inhomogeneous second-order Cauchy problem \cref{eq:w-u-cauchy} with controls $u$. The following are equivalent:
\begin{enumerate}[label=(\alph*)]
  \item \emph{($L^2$-inputs give $H^{-\eta}$-outputs.)} There exist $K, T > 0$ such that  \begin{equation}
  \|Q_0w\|_{H^{-\eta_1 - \eta_2}(0, T; Y)} + \|Q_1 \dot{w}\|_{H^{-\eta_1 - \eta_2}(0, T; Y)} \leq K \|u\|_{L^2(0, T; U)}, \quad u \in \D(0, T; U);
  \end{equation}
  \item \emph{(Frequency-domain condition.)}\label{it:TF-SO-FQ}
\begin{equation}
\left \{
\begin{aligned}
&\|Q_0((1 + \i \lambda)^2 + L)^{-1}P\|_{\L(U, Y)} = O(|\lambda|^{\eta_1 + \eta_2}),\\
&\|Q_1((1 + \i \lambda)^2 + L)^{-1}P\|_{\L(U, Y)} = O(|\lambda|^{\eta_1 + \eta_2 - 1}), 
\end{aligned}
 \right. \quad \lambda \in \R,\quad |\lambda| \to + \infty.
\end{equation}
\end{enumerate}
\end{theo}
\begin{proof}
Let
$
Z \triangleq W \times W
$.
We shall first prove that $Z$ satisfies the conditions of \cref{sec:TF-IO} for the group $\{S_t\}_{t\in\R}$. Clearly, $H_2 \times H_1$ is continuously embedded in $Z$; also, $C_0$ and $C_1$ both belong to $\L(U, Z)$. Now, pick $p \in \Co$ with $\re p > 0$.
By \cref{eq:res-A-SO},
\begin{equation}
(p - A)^{-1}B = \begin{pmatrix}
(p^2 + L)^{-1}P \\ p(p^2 + L)^{-1} P
\end{pmatrix}.
\end{equation}
By definition, the operators $p(p^2 + L)^{-1}P$ and $(p^2 + L)^{-1}P$ map $U$ into $W$, which is also continuously embeddeded into $H_1$. 
It follows from \cref{eq:res-A-SO} that the range of $(p - A)^{-1}B$ is indeed contained in $Z$, as required.
Now, the transfer functions associated to $C_0$ and $C_1$ are given by
\begin{equation}
C_0(p - A)^{-1}B = Q_0(p^2 + L)^{-1}P, \quad C_1(p - A)^{-1}B = pQ_1(p^2 + L)^{-1}P, \quad \re p > 0,
\end{equation}
respectively,
and we readily complete the proof by applying \cref{theo:transfer-function}.
\end{proof}
Finally, we consider a special case of interest, sometimes referred to as that of ``collocated input and output''; see, e.g., \cite{CurWei06,AmmNic14book}. We let again $P$ be the control operator, identify the input space $U$ with its antidual and introduce the adjoint $P^\ast \in \L(H_1, U)$ of $P$ in the pivot duality $H_1 \hookrightarrow H \hookrightarrow H_{-1}$:
\begin{equation}
\langle u, P^\ast w\rangle_U \triangleq \langle Pu, w \rangle_{H_{-1}, H_1}, \quad u \in U, \quad w \in H_1.
\end{equation}
When $P^\ast$ plays the role of the observation operator $Q_1$, the situation becomes much simpler. In particular,
one can deduce input-to-state and state-to-output properties from the corresponding input-output properties.
\begin{theo}[Systems with collocated input and output]\label{theo:collocated-SO} Let $0 \leq \eta \leq 1$. The following are equivalent:
\begin{enumerate}[label=(\arabic*),series=collocated]
  \item \label{it:col-output}The equivalent conditions of \cref{theo:observation-SO} with $Q_0 = 0$, $Q_1 = P^\ast$ and parameter $\eta$;
  \item \label{it:con-input}The equivalent conditions of \cref{theo:control-SO} with parameter $\eta$.
\end{enumerate}
Furthermore,
\begin{enumerate}[resume*=collocated]
  \item\label{it:col-FQ} \emph{(Frequency-domain condition.)}
  \begin{equation}
  \|P^\ast((1+\i \lambda)^2 + L)^{-1}P\|_{\L(U)} = O(|\lambda|^{2\eta -1}), \quad \lambda \in \R, \quad |\lambda| \to + \infty,
  \end{equation}
\end{enumerate}
implies \cref{it:col-output,it:con-input}.
\end{theo}

\begin{proof}
Equivalence between \cref{it:col-output,it:con-input} follows from the identity
\begin{equation}
\label{eq:norm-P-Past}
\|((1 + \i \lambda)^2 + L)^{-1}P\|_{\L(U, H)} = \|P^\ast((1 - \i \lambda)^2 + L)^{-1}\|_{\L(H, U)}, \quad \lambda \in \R.
\end{equation}
As for the implication from \cref{it:col-FQ} to \cref{it:con-input,it:col-output}, we use a pairing argument.
Let $\lambda \in \R$, $u \in U$ and $w \triangleq ((1 + \i \lambda)^2 +L)^{-1}Pu$. Then,
$
\langle Lw +  2 \i \lambda w - \lambda^2 w, w\rangle_H = \langle Pu, w \rangle_{H_{-1}, H}$ and taking the imaginary part yields
\begin{equation}
2 \lambda \|w\|^2_H = \im \langle Pu, w \rangle_{H_{-1}, H_1} = \im \langle u, P^\ast((1+\i \lambda)^2 +L)^{-1}Pu\rangle_U.
\end{equation}
If $\lambda \not = 0$, with the Cauchy--Schwarz inequality we get
\begin{equation}
\label{eq:est-P-PP}
\|((1+\i \lambda)^2 +L)^{-1}Pu\|_H^2 \leq \frac{1}{2|\lambda|} \|u\|_U \|P^\ast((1+\i \lambda)^2 +L)^{-1}Pu\|_U.
\end{equation}
\Cref{eq:est-P-PP} holds for arbitrary $u \in U$ and $\lambda \not = 0$. Thus,
\begin{equation}
\label{eq:passive-input-TF}
\|((1+\i \lambda)^2 +L)^{-1}P\|_{\L(U, H)}^2 \leq \frac{1}{2|\lambda|} \|P^\ast((1 + \i \lambda)^2 + L)^{-1}P\|_{\L(U)}, \quad \lambda \in \R, \quad \lambda \not = 0.
\end{equation}
The result follows from the frequency-domain characterisations of properties \cref{it:con-input,it:col-output} in \cref{theo:observation-SO,theo:control-SO}.
\end{proof}
\begin{rem}
Taking $\eta = 0$ in \cref{theo:collocated-SO}, we recover the fact that the \emph{impedance passive} system defined by the  equations \cref{eq:second-order-cauchy,eq:w-u-cauchy} with output $P^\ast \dot{w}$ is well-posed in the sense of Weiss if and only if its transfer function $p \mapsto p P^\ast(p^2 + L)^{-1}P$ is bounded on some  vertical line in the open right-half plane; see, e.g., \cite[Theorem 5.1]{Sta02}.
\end{rem}

\section{The wave and Schr\"odinger equations with Neumann boundary data}

\label{sec:wave}
\subsection{Frequency-domain asymptotics for the wave equation}

Let $\Omega$ be a  domain in $\R^d$, $d \geq 2$, with Lipschitz boundary\footnote{We follows \cite[Definition 1.2.1.1]{Gri85book}, which in particular does not exclude unbounded boundary.} $\partial \Omega$. 
We will assume that $\Omega$ satisfies at least one of the following conditions: 
\begin{itemize*}
 \item $\Omega$ has bounded boundary; \item $\Omega$ is the half-space. 
\end{itemize*}
In this section we consider the initial boundary value problem
\begin{subequations}
\label{eq:neumann-wave-IBVP}
\begin{align}
&(\partial^2_t - \Delta)w = 0&&\mbox{in}~\Omega \times (0, T), \\
&\partial_{\vec{n}} w = u&&\mbox{on}~\partial \Omega \times (0, T), \\
&(w, \partial_t w)|_{t =0} = 0 &&\mbox{in}~\Omega,
\end{align}
\end{subequations}
where $T > 0$ is fixed and $\partial_{\vec{n}}$ denotes the outward normal derivative. 
Define the Neumann Laplacian $\Delta_N$ on $L^2(\Omega)$ as follows:
\begin{equation}
\dom(\Delta_N) \triangleq  \{ w \in H^1(\Omega) : \Delta w \in L^2(\Omega),~ \partial_{\vec{n}} w = 0\}; \qquad \Delta_N w \triangleq \Delta w, \quad w \in \dom(\Delta_N).
\end{equation}
The above definition makes sense since any function $w$ in $H^1(\Omega)$ such that $\Delta w$ (defined \emph{a priori} as a distribution) belongs  to $L^2(\Omega)$ has a Neumann trace $\partial_\vn w$  well-defined in $H^{-1/2}(\partial \Omega)$; to see this, one may use Green's formula.\footnote{
Here and in the sequel, the use of Green's formula is justified by appropriate density arguments.}
When $\Omega$ is bounded and smooth,
the domain of $\Delta_N$ is exactly the space of all $w \in H^2(\Omega)$ such that $\partial_\vn w = 0$, and its graph norm is equivalent to the $H^2(\Omega)$-norm; see, e.g., \cite[Theorem 9.26]{Bre11book}.
In any case,  $-\Delta_N$ is a nonnegative self-adjoint operator, 
and $\dom((-\Delta_N)^{1/2}) = H^1(\Omega)$ with equivalence of norms. Next, denote by $\gamma = \cdot |_{\partial \Omega} \in \L(H^1(\Omega), L^2(\partial \Omega))$ the trace operator. Having identified $L^2(\Omega)$ and $L^2(\partial \Omega)$ with their respective antiduals, the adjoint $\gamma^\ast \in \L(L^2(\partial\Omega), (H^1(\Omega))^\ast)$ of $\gamma$ 
is given by
\begin{equation}
\langle \gamma^\ast u, w \rangle_{(H^1(\Omega))^\ast, H^1(\Omega)} = \langle u, \gamma w\rangle_{L^2(\partial \Omega)} = \int_{\partial \Omega} u \overline{w} \, \d \sigma, \quad u \in L^2(\partial \Omega), \quad w \in H^1(\Omega).
\end{equation}
In accordance  with its standard variational formulation, we can rewrite \cref{eq:neumann-wave-IBVP} as a second-order abstract Cauchy problem fitting into the framework of \cref{sec:second-order}, namely
\begin{equation}
\label{eq:neumann-wave-abstract}
\ddot{w} - \Delta_N w = \gamma^\ast u \quad \mbox{in}~(H^1(\Omega))^\ast, \quad w(0) = 0, \quad \dot{w}(0) = 0.
\end{equation}
It follows that boundary data $u \in L^2(\partial \Omega \times (0, T)) \simeq L^2(0, T; L^2(\partial\Omega))$\footnote{The identification stems from Fubini's theorem.} produce solutions $w$ that satisfy at least $w \in \C([0, T], L^2(\Omega)) \cap \C^1([0, T], (H^1(\Omega))^\ast)$. When $u \in \D(\partial \Omega \times (0, T))$ (or even $\D(0, T; L^2(\partial \Omega)))$, we see from the arguments of \cref{sec:second-order} that $w \in \C^1([0, T], H^1(\Omega))$.
\begin{rem}[Neumann harmonic extension]
 An alternative operator model for \cref{eq:neumann-wave-IBVP}, used for instance in \cite{LasTri81,LasTri91}, is given by
\begin{equation}
\label{eq:neumann-extension}
\ddot{w} - \Delta_N (w - Nu) = 0, \quad w(0) = 0, \quad \dot{w}(0) = 0,
\end{equation}
where 
$N$ is the Neumann harmonic extension map. Recall that $N$ maps $H^s(\partial \Omega)$ continuously into $H^{s +3/2}(\Omega)$ for all $s \in \R$ and in particular $\Delta_N N$ is continuous from $L^2(\partial \Omega)$ into $(H^1(\Omega))^\ast$. Using Green's formula, it is easy to show that $\gamma^\ast u$ and $\Delta_N N u$ coincide for all $u \in L^2(\partial \Omega)$, so that \cref{eq:neumann-wave-abstract,eq:neumann-extension} give rise to the same solutions.

\end{rem}

In the sequel, we are interested in the  transfer function
\begin{equation}
\label{eq:wave-neumann-TF}
p \mapsto  \gamma(p^2 - \Delta_N)^{-1}\gamma^{\ast}
\end{equation}
associated with the input-output behaviour of system \cref{eq:neumann-wave-IBVP}, the output  being the Dirichlet trace of the solutions. It is clear from the above discussion that for each $\lambda \in \R$ the operator in \cref{eq:wave-neumann} belongs to $\L(L^2(\partial \Omega))$. The next lemma, based on Ne\v cas' results on elliptic regularity, shows that it is also well-defined and continuous from $L^2(\partial \Omega)$ into $H^1(\partial \Omega)$.\footnote{When $\Omega$ is smooth, $\partial \Omega$ has a canonical Riemannian submanifold structure and $H^1(\partial \Omega)$ is the subspace of $L^2(\partial \Omega)$ with square-integrable Riemannian gradient; see, e.g., \cite[Section 3, Chapter 4]{Tay96book}. When $\Omega$ is merely Lipschitz but has bounded boundary,
 we can define $H^1(\partial \Omega)$ as in, e.g., \cite[Appendix A]{ChaGra12survey}.}

\begin{lemma}
\label{lem:elliptic-tang}
Let $w \in H^1(\Omega)$ be such that $\Delta w \in L^2(\Omega)$. If $\partial_{\vec{n}}w \in L^2(\partial \Omega)$, then $w|_{\partial \Omega} \in H^1(\partial \Omega)$. In fact, there exists $K > 0$ such that all such $w$ satisfy
\begin{equation}
\label{eq:elliptic-tang}
\|w|_{\partial \Omega}\|_{H^1(\partial \Omega)} \leq K( \|w\|_{H^1(\Omega)}+  \|\partial_{\vec{n}}w\|_{L^2(\partial \Omega)} + \|\Delta w\|_{L^2(\Omega)}).
\end{equation}
\end{lemma}


In the case that $\Omega$ is bounded, \cref{lem:elliptic-tang} is a straightforward application of  \cite[Theorem 2.1, Chapter 5]{Nec12book}. Otherwise some additional (simple) arguments are required and we refer the reader to \cref{sec:cutoff} for the complete proof.
Next, we can recast \cref{lem:elliptic-tang}  
in our transfer function framework and, more specifically, verify that the technical hypothesis \cref{eq:hyp-W-space} holds when considering $H^1(\partial \Omega)$ as an output space for the observation $w|_{\partial \Omega}$.

\begin{coro}
\label{coro:elliptic-W}
Let $W \triangleq \{w \in H^1(\Omega) : \Delta w \in L^2(\Omega),~ \partial_\vn w \in L^2(\partial \Omega)\}$, equipped with the norm
\begin{equation}
\|w\|_W^2 \triangleq \|w\|^2_{H^1(\Omega)} + \|\Delta w \|^2_{L^2(\Omega)} + \|\partial_\vn w\|^2_{L^2(\partial \Omega)}, \quad w \in W.
\end{equation}
Then $W$ is a Hilbert space that satisfies $\dom(\Delta_N) \hookrightarrow W \hookrightarrow H^1(\Omega)$ with continuous embeddings. Furthermore, the trace map $\gamma$ is continuous from $W$ into $H^1(\partial \Omega)$ and, for each $p \in \Co$ with $\re p > 0$, $(p^2 - \Delta_N)^{-1}\gamma^\ast$ maps $L^2(\partial \Omega)$ into $W$. In particular,
\begin{equation}
\label{eq:lambda-L2-H1}
\gamma((1 + \i \lambda)^2 - \Delta_N)^{-1}\gamma^\ast \in \L(L^2(\partial \Omega), H^1(\partial\Omega)), \quad \lambda \in \R.
\end{equation}
\end{coro}
\begin{proof}
It is clear that $W$ is a subspace of $H^1(\Omega)$ with continuous embedding. To check that it is a Hilbert space, observe first that any Cauchy sequence $\{w_m\}_{m\in \N}$ in $W$ must converge in $H^1(\Omega)$ to some limit $w$, which we will prove belongs to $W$. The elements $\Delta w_m$ converge to $\Delta w$ strongly in  $H^{-1}(\Omega)$ but they must also have a weak sublimit in $L^2(\Omega)$, so that $\Delta w \in L^2(\Omega)$ as well. Similarly, using Green's formula, we see that $\partial_\vn w_m$ converges to $\partial_\vn w$ strongly in $H^{-1/2}(\partial \Omega)$ and also have a weak sublimit in $L^2(\partial \Omega)$, so that $\partial_\vn w \in L^2(\partial \Omega)$, as required. 
The continuous embedding $\dom(\Delta_N) \hookrightarrow W$ is also clear. By \cref{lem:elliptic-tang}, $\gamma$ is indeed continuous from $W$ into $H^1(\partial \Omega)$. Finally, let $u \in L^2(\partial \Omega)$, $p \in \Co$ with $\re p > 0$, and $w \triangleq (p^2 - \Delta_N)^{-1}\gamma^\ast u \in H^1(\Omega)$. Pairing $(p^2 - \Delta_N)w = \gamma^\ast u$ with test functions $\varphi \in \D(\Omega)$ shows that $p^2w - \Delta w = 0$, and in particular $\Delta w \in L^2(\Omega)$. Then, pairing the same equality with arbitrary elements $v \in H^1(\Omega)$ and using Green's formula give $\partial_\vn w = u$ in $H^{-1/2}(\partial\Omega)$, which implies $\partial_\vn w \in L^2(\partial \Omega)$. Here, \cref{eq:lambda-L2-H1} directly follows  from the estimate \cref{eq:elliptic-tang} in \cref{lem:elliptic-tang}.
\end{proof}

Recall from \cref{sec:intro-edp} that our ``loss of derivative'' parameter $\eta$ is defined as follows:
\begin{equation}
\label{eq:reminder-loss}
\mbox{If $\Omega$ is smooth}, \quad
\eta \triangleq 
\left \{
\begin{aligned}
&1/6&&\mbox{if $\partial \Omega$ is concave}, \\
&1/4&&\mbox{if $\partial \Omega$ is flat,} \\
&1/3&&\mbox{otherwise;}
\end{aligned}
\right. \qquad \eta \triangleq 1/4 + \eps \quad \mbox{if $\Omega$ is a rectangle},
\end{equation}
where $\eps$ is any (small) positive real. Note that the cases of concave and flat boundary are mostly relevant when $\Omega$ is unbounded, e.g., in exterior problems. The following theorem summarises regularity results from Lasiecka, Triggiani \cite[Theorem A]{LasTri91} and Tataru \cite[Theorem 10]{Tat98}.
\begin{theo}[Sharp interior and boundary regularity\cite{LasTri91,Tat98}]
\label{theo:sharp-wave}
Let $\Omega$ and $\eta$ be as in \cref{eq:reminder-loss}.
For all $u \in L^2(\partial \Omega \times (0, T))$, the corresponding solution $w$ to \cref{eq:neumann-wave-IBVP} 
satisfies
\begin{equation}
\label{eq:sharp-wave}
(w, \partial_t w)|_{t = T} \in H^{1-\eta}(\Omega) \times H^{-\eta}(\Omega), \quad w|_{\partial \Omega} \in H^{1-2\eta}(\partial \Omega \times (0, T)),
\end{equation}
with continuous dependence on $u$.
\end{theo}
First, we use our results in combination with \cref{theo:sharp-wave} to establish an estimate for \cref{eq:wave-neumann-TF} in $\L(L^2(\partial \Omega))$-norm.

\begin{theo}[Transfer function asymptotics, $L^2(\partial\Omega)$-level]
\label{coro:trace-TF} Let $\Omega$ and $\eta$ be as in \cref{eq:reminder-loss}. We have
\begin{equation}
\label{eq:trace-TF}
\|\gamma ((1 + \i \lambda)^2 - \Delta_N)^{-1}\gamma^\ast\|_{\L(L^2(\partial \Omega))} = O(|\lambda|^{2\eta-1}), \quad \lambda \in \R, \quad |\lambda| \to + \infty.
\end{equation}
\end{theo}
\begin{proof}
We will apply \cref{theo:control-SO,theo:TF-SO}. Denote by $\{H_s\}_{s\in \R}$ the Sobolev scale associated with $-\Delta_N$, built following the procedure of \cref{sec:second-order}. 
 Recall that the spaces $H_s$ form an interpolation scale so that, up to equivalence of norms,
\begin{equation}
H_{1-\eta} = [H_{1}, H]_{\eta} = [H^1(\Omega), L^2(\Omega)]_{\eta} = H^{1-\eta}(\Omega).
\end{equation}
Notice also that, regardless of $\Omega$, we always have $0 \leq \eta < 1/2$, which means that $H^\eta(\Omega) = H^\eta_0(\Omega)$. Thus, regarding $H^{-\eta}(\Omega)$ as a space of antilinear forms,
\begin{equation}
H_{-\eta} = (H_\eta)^\ast = (H^\eta_0(\Omega))^\ast = H^{-\eta}(\Omega).
\end{equation}
Therefore,  continuity $L^2(\partial \Omega \times (0, T)) \to H^{1-\eta}(\Omega) \times H^{-\eta}(\Omega)$ of the map $u \mapsto (w, \partial_t w)|_{t = T}$, where $w$ solves \cref{eq:neumann-wave-IBVP}, as given by \cref{theo:sharp-wave}, is exactly $\eta$-admissiblity of the Neumann control operator $\gamma^\ast$ for our wave group on $H_1 \times H = H^1(\Omega) \times L^2(\Omega)$.  Together with \cref{theo:collocated-SO}, this moreover shows that the Dirichlet trace $\gamma$ of the velocity defines an $\eta$-admissible observation operator. We are therefore in a position to apply \cref{theo:TF-SO}. First, recall from, e.g., \cite[Section 13.3, Chapter 1]{LioMag68book} that $H^{1-2\eta}(\partial \Omega \times (0, T)) \simeq L^2(0, T; H^{1-2\eta}(\partial \Omega)) \cap H^{1-2\eta}(0, T; L^2(\partial \Omega))$, so that by \cref{lem:diff-shift} the map $u \mapsto \partial_t w|_{\partial \Omega}$ is continuous $L^2(0, T; L^2(\partial \Omega)) \to H^{-2\eta}(0, T; L^2(\partial \Omega))$. The transfer function estimate \cref{eq:trace-TF} then follows from \cref{theo:TF-SO}.
\end{proof}
\begin{rem}
\label{rem:reg-wave}
In \cite{Tat98}, the regularity result \cref{eq:sharp-wave} appears in a seemingly weaker form, namely:
\begin{equation}
w \in H^{1-\eta}(\Omega \times (0, T)), \quad w|_{\partial \Omega} \in H^{1-2\eta}(\partial \Omega \times (0, T)).
\end{equation}
Fortunately, owing to the special structure of \cref{eq:neumann-wave-abstract}, the property $w \in H^{1-\eta}(\Omega \times (0, T))$ can be transformed into $w \in \C([0, T], H^{1-\eta}(\Omega))  \cap \C^1([0, T], H^{-\eta}(\Omega))$ by means of a short functional-analytic argument, which is presented in full at the end of \cref{sec:cutoff}.
\end{rem}

Our next result provides high-frequency bounds for the transfer function \cref{eq:wave-neumann-TF} in the (stronger) $\L(L^2(\partial\Omega), H^1(\partial \Omega))$-norm. While the $L^2$-regularity property for $u \mapsto \partial_t w |_{\partial \Omega}$ in \cref{theo:sharp-wave} fitted nicely into our abstract framework, here we will make us of additional partial differential equation arguments. By using a wavenumber-dependent \emph{a priori} estimate for the Helmhotz equation, derived by Spence in \cite{Spe14} and based on a quantified version  of Ne\v cas' elliptic regularity results \cite{Nec12book}, we are able to control the $\L(L^2(\partial\Omega), H^1(\partial \Omega))$-norm of \cref{eq:wave-neumann-TF} in terms of its $\L(L^2(\partial \Omega))$-norm up to additional growth and lower-order terms.
\begin{lemma}
\label{lem:comp-H1-L2}
There exists $K > 0$ such that, for all $\lambda \in \R$,
\begin{multline}
\label{eq:comp-H1-L2}
\|\gamma((1+\i \lambda)^2 - \Delta_N)^{-1}\gamma^\ast\|_{\L(L^2(\partial \Omega), H^1(\partial \Omega))} \leq K \big( 1+ (1 + |\lambda|) \|\gamma((1+\i \lambda)^2 - \Delta_N)^{-1}\gamma^\ast\|_{\L(L^2(\partial \Omega))} \\ + \|((1 + \i \lambda)^2 - \Delta_N)^{-1}\gamma^\ast\|_{\L(L^2(\partial \Omega), H^1(\Omega))} + (1 + |\lambda|)\|((1 + \i \lambda)^2 - \Delta_N)^{-1}\gamma^\ast\|_{\L(L^2(\partial \Omega), L^2(\Omega))}\big).
\end{multline}
\end{lemma}
\begin{proof}
\Cref{lem:helmhotz}, which is a slightly modified version of \cite[Lemma 3.5]{Spe14},
states that there exists $K > 0$ such that for all $k \geq 1$ and $f \in L^2(\Omega)$, if $w \in H^1(\Omega)$ solves $\Delta w + k^2 w = - f$ and $\partial_\vn w \in L^2(\partial \Omega)$, then
\begin{equation}
\label{eq:apriori-helm}
\|w|_{\partial \Omega}\|_{H^1(\partial \Omega)} \leq K(\|\partial_\vn w\|_{L^2(\partial \Omega)} + k \|w|_{\partial \Omega}\|_{L^2(\partial \Omega)} + \|\vnabla w\|_{L^2(\Omega)^d} + k \|w\|_{L^2(\Omega)} + \|f\|_{L^2(\Omega)}).
\end{equation}
Here, explicit control over the parameter $k$ is crucial. Now let $u \in L^2(\partial \Omega)$, $\lambda \in \R$ with $|\lambda| \geq 1$  and $w \triangleq ((1 + \i \lambda)^2 - \Delta_N)^{-1} u$. Then $w$ solves $\Delta w + \lambda^2 w = w - 2 \i \lambda w$, $\partial_\vn w = u$, and applying \cref{eq:apriori-helm} yields
\begin{equation}
\label{eq:apriori-helm-comp}
\|w|_{\partial \Omega}\|_{H^1(\partial \Omega)} \leq K'\big( \|u\|_{L^2(\partial\Omega)} + (1 + |\lambda|)\|w|_{\partial \Omega}\|_{L^2(\partial \Omega)} + \|w\|_{H^1(\Omega)} + (1 +|\lambda|) \|w\|_{L^2(\Omega)}\big).
\end{equation}
We reformulate \cref{eq:apriori-helm-comp} in resolvent terms: for all $u \in L^2(\Omega)$ and $\lambda \in \R$ with $|\lambda| \geq 1$,
\begin{multline}
\|\gamma((1+\i \lambda)^2 - \Delta_N)^{-1}\gamma^\ast u\|_{H^1(\Omega)} \leq K'\big(\|u\|_{L^2(\partial\Omega)}+ (1 +|\lambda|) \|\gamma((1+\i \lambda)^2 - \Delta_N)^{-1}\gamma^\ast u\|_{L^2(\partial \Omega)}  \\ + \|((1 + \i \lambda)^2 - \Delta_N)^{-1}\gamma^\ast u \|_{H^1(\Omega)} + (1 +|\lambda|) \|((1 + \i \lambda)^2 - \Delta_N)^{-1}\gamma^\ast u \|_{L^2(\Omega)}\big).
\end{multline}
The vector $u$ being arbitrarily chosen in $L^2(\partial \Omega)$, \cref{eq:comp-H1-L2} follows at once.
\end{proof}
\begin{theo}[Transfer function asymptotics, $L^2(\partial \Omega) \to H^1(\partial \Omega)$-level]\label{theo:trace-TF-tang} Let $\Omega$ and $\eta$ be as in \cref{eq:reminder-loss}. We have
\begin{equation}
\label{eq:trace-TF-tang}
\|\gamma ((1 + \i \lambda)^2 - \Delta_N)^{-1}\gamma^\ast\|_{\L(L^2(\partial \Omega), H^1(\partial \Omega))} = O(|\lambda|^{2\eta}), \quad \lambda \in \R, \quad |\lambda| \to + \infty.
\end{equation}
\end{theo}
\begin{proof}
In the proof of \cref{coro:trace-TF}, we saw that the Neumann control operator $\gamma^\ast$ is $\eta$-admissible for the wave group on $H^1(\Omega) \times L^2(\Omega)$. Thus by \cref{theo:control-SO},
\begin{equation}
\label{eq:interior-TF}
\|((1 + \i \lambda)^2 - \Delta_N)^{-1}\gamma^\ast\|_{\L(L^2(\partial \Omega),L^2(\Omega))} = O(|\lambda|^{\eta - 1}), \quad \lambda \in \R, \quad |\lambda| \to + \infty.
\end{equation}
Just as in the proof of \cref{theo:observation-SO} (see the step around \cref{eq:H1-w}), we can
 deduce from \cref{eq:interior-TF} that
\begin{equation}
\label{eq:interior-TF-H1}
\|((1 + \i \lambda)^2 - \Delta_N)^{-1}\gamma^\ast\|_{\L(L^2(\partial \Omega),H^1(\Omega))} = O(|\lambda|^{\eta}), \quad \lambda \in \R, \quad |\lambda| \to + \infty.
\end{equation}
We complete the proof by substituting \cref{eq:trace-TF,eq:interior-TF,eq:interior-TF-H1} into the inequality \cref{eq:comp-H1-L2} from \cref{lem:comp-H1-L2}.
\end{proof}

As a direct consequence of \cref{theo:control-SO,theo:trace-TF-tang}, we obtain the following improvement over \cref{theo:sharp-wave}.
\begin{coro}
\label{coro:add-tang-reg}
Let $\Omega$ and $\eta$ be as in \cref{eq:reminder-loss}. In addition to the properties listed in \cref{theo:sharp-wave}, solutions $w$ to \cref{eq:neumann-wave-IBVP} with Neumann boundary data $u \in L^2(\partial \Omega \times (0, T))$ satisfy
\begin{equation} 
\label{eq:add-tang-reg}
w|_{\partial \Omega} \in
 H^{-2\eta}(0, T; H^1(\partial \Omega)),
\end{equation}
with continuous dependence on $u$.
\end{coro}
\begin{rem}
\Cref{eq:add-tang-reg} does not directly follow from the property $w|_{\partial \Omega} \in H^{1-2\eta}(\partial \Omega \times (0, T))$. The latter merely implies that
 tangential derivatives lie in $H^{-2\eta}(\partial \Omega \times (0, T))$, which is a larger space than $H^{-2\eta}(0, T; L^2(\partial \Omega))$.
\end{rem}



\subsection{The Schr\"odinger equation on the half-space}

In this section, 
$\Omega$ is the half-space $\{ (x, y) \in \R \times \R^{d-1} : x \geq 0\}$, so that $\partial \Omega \simeq \R^{d-1}$.
We consider the Schr\"odinger equation posed in $\Omega$ with Neumann boundary data $u$:
\begin{subequations}
\label{eq:schro-neumann}
\begin{align}
&(\i \partial_t - \Delta)\Psi = 0 &&\mbox{in}~\Omega \times (0, T), \\
\label{eq:schro-neumann-BC}
&\partial_\vn \Psi = u &&\mbox{on}~\partial \Omega \times (0, T).
\end{align}
\end{subequations}
Perhaps surprisingly, regularity results pertaining to the linear Schr\"odinger equation with Neumann boundary condition appear to be rather scarce in the literature. A negative result, which we shall return to shortly,  is proved in \cite[Section 8.2]{LasTri03}. On the other hand, \cite{LasTri04} contains \emph{a priori} estimates for the $L^2(\Omega)$-energy of solutions to the Schr\"odinger equation (that is, without prescribed boundary conditions). Also, in the special case that $\Omega$ is a rectangle in $\R^2$, \cite[Proposition 3.1]{RamTak05} states that the Dirichlet trace is an admissible observation operator for the Schr\"odinger group with homogeneous Neumann boundary condition on $L^2(\Omega)$, which by duality means that $L^2(\partial \Omega \times (0, T))$-boundary data produce $L^2(\Omega)$-solutions. The goal of our analysis here is twofold:
\begin{itemize}
  \item We illustrate how our results may be used to obtain regularity properties;
  \item We provide an example of an observation operator that is admissible for the Schr\"odinger group on $L^2(\Omega)$ but \emph{not} admissible  for the wave group
 on $H^1(\Omega) \times L^2(\Omega)$. This proves that the technique of transference from waves to Schr\"odinger, as seen in, e.g., \cite[Section 6.8]{TenTuc09} and \cite{Mil12}, only goes one way.
\end{itemize}

The choice of the half-space allows us to keep the partial differential equation techniques at an elementary level, but we believe that this case study is instructive nonetheless.

\begin{theo}[Sharp regularity under $L^2$-Neumann control]
\label{theo:reg-schro} For all $u \in L^2(\partial \Omega \times (0, T))$, the corresponding solution $\Psi$ to \cref{eq:schro-neumann} satisfies
\begin{equation}
\label{eq:reg-schro}
\Psi|_{t = T} \in L^2(\Omega), \quad \Psi|_{\partial \Omega} \in L^2(0, T; L^2(\partial \Omega)) \cap H^{-1/2}(0, T; H^1(\partial \Omega)),
\end{equation}
with continuous dependence on $u$.
\end{theo}

\begin{proof} 
The operator formulation for \cref{eq:schro-neumann} is given by $\dot{\Psi} = - \i \Delta_N \Psi - \i \gamma^\ast u$.
The abstract framework and validity of our technical hypotheses have already been discussed in the previous section, so we shall not go into more detail than necessary when applying the results from \cref{sec:main-res}.

The main part of the proof  consists in deriving Neumann-to-Dirichlet bounds for the ``damped'' Helmhotz equation
\begin{subequations}
\label{eq:damped-helm}
\begin{align}
&(- \i + k^2 + \Delta) w = 0 &&\mbox{in}~ (0, +\infty) \times \R^{d-1}, \\
&\partial_x w|_{x = 0} = - u &&\mbox{on}~\R^{d-1}.
\end{align}
\end{subequations}
that are uniform in $k \geq 0$ and $u \in L^2(\R^{d-1})$. Note that for all such $k$ and $u$ there exists a unique solution $w\in H^1(\Omega) = H^1((0, +\infty) \times \R^{d-1})$ to \cref{eq:damped-helm} (actually $w \in W$, where $W$ was defined in \cref{coro:elliptic-W}), which is given by $w = - (- \i + k^2 + \Delta_N)^{-1}\gamma^\ast u$. We shall temporarily deviate from our notation convention and denote by $\hat{w}$ (resp.\ $\hat{u}$) the Fourier transform of $w$ (resp.\ ${u}$) in the \emph{tangential} direction, that is, in the $y$-variable. Then $\hat{w} \in H^1((0, +\infty) \times \R^{d-1})$ as well and in particular the function $\tau \mapsto \hat{w}(\cdot, \tau)$ lies in $L^2(\R^{d-1}, H^1(0, +\infty))$. For a.e.\ $\tau \in \R^{d-1}$, $\hat{w}(\cdot, \tau)$ satisfies the following ordinary differential equation and initial condition:
\begin{equation}
  \label{eq:ode-w}
  \partial_x^2 \hat{w}(x, \tau) + (k^2  - \|\tau\|^2 - \i ) \hat{w}(x, \tau) = 0, \quad \partial_x \hat{w}(0, \tau) = - \hat{u}(\tau).
\end{equation}
The characteristic equation $r^2 + k^2 - \|\tau\|^2 - \i = 0$  has two distinct roots $r = \pm r(k, \tau)$, with $\re r(k, \tau) > 0$ and $|r(k, \tau)| = |\i + \|\tau\|^2 - k^2|^{1/2}$. Using \cref{eq:ode-w} together with the property that $\hat{w}(\cdot, \tau) \in L^2(0, +\infty)$ for a.e.\  $\tau \in \R^{d-1}$,
 we find that $\hat{w}$ is  given explicitely by
\begin{equation}
\hat{w}(x, \tau) =   \frac{\hat{u}(\tau)}{r(k, \tau)} e^{- xr(k, \tau)}, \quad x \geq 0.
\end{equation}
Therefore,
\begin{equation}
\int_{\R^{d-1}}|\hat{w}(0, \tau)|^2 \, \d \tau = \int_{\R^{d-1}}\frac{|\hat{u}(\tau)|^2}{1 + |\|\tau\|^2 - k^2|^2} \, \d \tau \leq \int_{\R^{d-1}} |\hat{u}(\tau)|^2 \, \d \tau.
\end{equation}
Finally, using Plancherel's theorem and reformulating the result in resolvent terms, we obtain that
\begin{equation}
\label{eq:res-schro-PDE}
\|\gamma (-\i + k^2 + \Delta_N)^{-1}\gamma^\ast\|^2_{\L(L^2(\partial \Omega))} \leq \frac{1}{(2\pi)^{d - 1}}, \quad k \geq 0.
\end{equation}

The rest of the proof relies primarily on our abstract results; in particular the special geometry does not play a role. Let $\lambda \in \R$,  $u \in L^2(\partial \Omega)$ and $w \triangleq - \i (1 + \i \lambda + \i \Delta_N)^{-1} \gamma^\ast u$.  Then, multiplying $(1 + \i \lambda + \i \Delta)w$ by $\overline{w}$ and using Green's formula give
\begin{equation}
\label{eq:pairing-schro-ell}
\int_\Omega |w|^2 + \i \lambda |w|^2 - \i \|\vnabla w\|^2 \, \d x + \i \int_{\partial \Omega} u \overline{w} \, \d \sigma = 0.
\end{equation}
After  taking the imaginary part of \cref{eq:pairing-schro-ell} and making use of the Cauchy--Schwarz and Young inequalities, we obtain
\begin{equation}
\label{eq:HF-schro-elliptic}
\|(1 + \i \lambda + \i \Delta_N)^{-1}\gamma^\ast\|_{\L(L^2(\partial \Omega), H^1( \Omega))} = O(1), \quad \lambda \to - \infty
\end{equation}
(the case $\lambda < 0$ can be seen as an elliptic regime). As a result, by continuity $H^1(\Omega) \to L^2(\partial \Omega)$ of the trace,
\begin{equation}
\label{eq:HF-schro-elliptic-trace}
\|\gamma(1 + \i \lambda + \i \Delta_N)^{-1}\gamma^\ast\|_{\L(L^2(\partial \Omega)))} = O(1), \quad \lambda \to - \infty
\end{equation}
\Cref{eq:res-schro-PDE} implies that the bound \cref{eq:HF-schro-elliptic} is also valid when $\lambda \to + \infty$. Taking the real part of the pairing identity \cref{eq:pairing-schro-ell} (much as in the proof of \cref{theo:collocated-SO}), we see that, as $|\lambda| \to +\infty$, $\lambda \in \R$,
\begin{equation}
\label{eq:adm-schro}
\|(1 + \i \lambda + \i \Delta_N)^{-1}\gamma^\ast\|^2_{\L(L^2(\partial \Omega), L^2(\Omega))} \leq \|\gamma(1 + \i \lambda + \i \Delta_N)^{-1}\gamma^\ast\|_{\L(L^2(\partial \Omega))} = O(1).
\end{equation}
\Cref{theo:input-group,theo:transfer-function} now directly yield the $L^2$-parts of the desired regularity properties \cref{eq:reg-schro}.

It remains to prove the $H^{-1/2}(0, T; H^1(\partial \Omega))$-regularity in \cref{eq:reg-schro}. To this end we shall establish the bound
\begin{equation}
\label{eq:schro-tang-TF}
\|\gamma(1 + \i \lambda + \i \Delta_N)^{-1}\gamma^\ast\|_{\L(L^2(\partial \Omega), H^1(\partial \Omega))} = O(|\lambda|^{1/2}), \quad \lambda \in \R, \quad |\lambda| \to + \infty.
\end{equation}
Again, in the case $\lambda \to - \infty$ we may use elliptic estimates. Let $\lambda < 0$, $u \in L^2(\partial \Omega)$ and $w \triangleq (1 + \i \lambda + \i \Delta_N)^{-1}\gamma^\ast u$. Multiply $(1 + \i \lambda + \i \Delta)w$ this time by $\Delta \overline{w}$ to obtain
\begin{equation}
\label{eq:elliptic-bis}
- \int_\Omega \|\vnabla w\|^2 \, \d x + \int_{\partial \Omega} w \overline{u} \, \d \sigma + \i  \int_\Omega |\lambda| \|\vnabla w\|^2 + |\Delta w|^2 \, \d x + \i \int_{\partial \Omega} w \overline{u} \, \d \sigma = 0.
\end{equation}
By taking the imaginary parts of \cref{eq:pairing-schro-ell,eq:elliptic-bis}, we deduce that the $W$-norm of $w$ is controlled by the $L^2(\partial\Omega)$-norm of $u$  uniformly in $\lambda \leq -1$. By \cref{coro:elliptic-W}, $\gamma$ is continuous from $W$ into $H^1(\partial \Omega)$, and thus \cref{eq:schro-tang-TF} is indeed valid for $\lambda \to - \infty$. For $\lambda \to + \infty$ we use the same technique as in the proof of \cref{lem:comp-H1-L2}. If $\lambda \geq 1$ (say) we let $k \triangleq \lambda^{1/2}$ so that, given $u \in L^2(\partial \Omega)$, $w$ defined just as above solves $(k^2 + \Delta)w = \i w$, $\partial_\vn w = u$, and by the \emph{a priori} estimate \cref{eq:apriori-helm} provided by \cref{lem:helmhotz} we obtain
\begin{equation}
\|w|_{\partial \Omega}\|_{H^1(\partial \Omega)} \leq K(\|w\|_{H^1(\Omega)} + k \|w\|_{L^2(\Omega)} + \|u\|_{L^2(\partial \Omega)} + k \|w|_{\partial \Omega}\|_{L^2(\partial \Omega)}).
\end{equation}
Thus, for all $\lambda \geq 1$,
\begin{multline}
\|\gamma (1 + \i \lambda + \i \Delta_N)^{-1}\gamma^\ast\|_{\L(L^2(\partial \Omega), H^1(\partial \Omega))} \leq K\big(1 + \lambda^{1/2} \|\gamma (1 + \i \lambda + \i \Delta_N)^{-1}\gamma^\ast\|_{\L(L^2(\partial \Omega))}  \\ + \lambda^{1/2} \|(1+\i \lambda +  \i \Delta_N)^{-1}\gamma^\ast\|_{\L(L^2(\partial \Omega), L^2(\Omega))}+ \|(1+\i \lambda +  \i \Delta_N)^{-1}\gamma^\ast\|_{\L(L^2(\partial \Omega), H^1(\Omega))}\big).
\end{multline}
With \cref{eq:adm-schro} already in hand, to get \cref{eq:schro-tang-TF} it suffices to prove that
\begin{equation}
\label{eq:square-lambda}
\|(1+\i \lambda +  \i \Delta_N)^{-1}\gamma^\ast\|_{\L(L^2(\partial \Omega), H^1(\Omega))} = O(\lambda^{1/2}), \quad \lambda \to +\infty.
\end{equation}
To do so, we return to the pairing identity \cref{eq:pairing-schro-ell}, from which we readily deduce that
\begin{equation}
\|(1 + \i \lambda + \i \Delta_N)^{-1}\gamma^\ast\|^2_{\L(L^2(\partial \Omega), H^1(\Omega))} \leq K(1 + \lambda \|(1+\i \lambda +  \i \Delta_N)^{-1}\gamma^\ast\|_{\L(L^2(\partial \Omega), L^2(\Omega))}^2), \quad \lambda \geq 1.
\end{equation}
This gives \cref{eq:square-lambda} and completes the proof of \cref{{eq:schro-tang-TF}}, which in turn gives the final part of \cref{eq:reg-schro} via \cref{theo:transfer-function}.
\end{proof}
\begin{rem}
Using interpolation between $L^2(\partial \Omega) \to L^2(\partial \Omega)$ and $L^2(\partial \Omega) \to H^1(\partial \Omega)$ at the transfer function level (or, alternatively, some variant of the intermediate derivative theorem \cite[Theorem 2.3, Chapter 1]{LioMag68book}), we also obtain $\Psi|_{\partial \Omega} \in H^{-\theta/2}(0, T; H^\theta(\partial \Omega))$ for all $0 < \theta < 1$.
\end{rem}

\begin{rem}
In \cite[Section 8.2]{LasTri03} it is shown that, for arbitrary $\eps > 0$, there exists $u \in L^2(\partial \Omega \times (0, T))$ such that $\Psi \not \in L^2(0, T; H^{\eps}(\Omega))$. With a slight modification of their argument, we also see that $\Psi|_{\partial \Omega} \not \in L^2(0, T; H^{\eps}(\partial \Omega))$ in general. Therefore the space regularity in \cref{theo:reg-schro} is optimal, at least in terms of the Sobolev exponent.
\end{rem}

\section{Energy decay rates under non-uniform Hautus test}

\label{sec:decay}

\subsection{Energy decay rates in the abstract setting}

\label{sec:stability}
In this section we study energy decay rates for second-order damped systems of the form
\begin{equation}
\label{eq:SO-feedback}
\ddot{w} + DD^\ast \dot{w} + Lw  = 0, \quad w(0) = w_0,  \quad \dot{w}(0) = w_1,
\end{equation}
{where $L$ is as in \cref{sec:second-order-adm} and $D \in \L(U, H_{-1})$.
 \Cref{eq:SO-feedback} gives rises to a strongly continuous semigroup $\sg{S^D}$ on the phase space $H_1 \times H$. Its infinitesimal generator $A_D$ is given by
\begin{equation}
\dom(A_D) = \{ (w_0, w_1) \in H_1 \times H_1 : DD^\ast w_1 + Lw_0 \in H \}, \quad A_D = \begin{pmatrix}
0 & 1 \\ -L & -DD^\ast 
\end{pmatrix}.
\end{equation}
Following \cite{AmmTuc01,ChiPau23,KleWan23arxiv}, we are interested in sufficient conditions for quantified decay of the \emph{energy}
\begin{equation}
\label{eq:energy-semi}
E(t; w) \triangleq \frac{1}{2}\|L^{1/2}w(t)\|^2_H + \frac{1}{2}\|\dot{w}(t)\|^2_H, \quad t \geq 0, \quad w~\mbox{solves \cref{eq:SO-feedback}},
\end{equation}
in terms of \emph{observability} properties of the uncontrolled dynamics \cref{eq:second-order-cauchy} with output $D^\ast \dot{w}$. 
In the case that $0 \in \rho(L)$, as in \cite[Section 2B]{ChiPau23},  the feedback semigroup $\sg{S^D}$ is contractive and $2E(t; w)$ is the squared norm of $(w(t), \dot{w}(t))$ in $H_1 \times H$.
We consider a non-uniform Hautus-type condition similar to that of \cite[Proposition C.2]{JolLau20} or \cite[Proposition 3.10]{ChiPau23}; see also \cite[Definition 1.1]{KleWan23arxiv}.

\begin{defi}[Non-uniform Hautus test] {The pair $(\i L, D^\ast)$} satisfies the non-uniform Hautus test if there exist functions $m, M : \R \to [1, +\infty)$ and a real number $\lambda_0 \geq 0$ such that 
\begin{equation}
\label{eq:hautus}
\|w\|_H  \leq {M(\lambda)}\|(L - \lambda^2)w\|_H + m(\lambda)\|D^\ast w\|_U, \quad \lambda \in \R,  \quad |\lambda| \geq \lambda_0, \quad w \in H_2.
\end{equation}
\end{defi}
\begin{rem}[High-frequency unique continuation]
An immediate consequence of the non-uniform Hautus test is the following unique continuation property: for all $\lambda \in \R$ with $|\lambda| \geq \lambda_0$, if $w \in H_1$, $w \not = 0$, satisfies $(L - \lambda^2)w = 0$, then $D^\ast w \not = 0$.
\end{rem}
\begin{rem}[Schrödinger group observability]
\label{rem:sch-obs}
When $D^\ast$ is an admissible observation operator for the Schr\"odinger group $\{e^{\i t L}\}_{t \in \R}$ on $H$, the non-uniform Hautus test holding with constant parameters $m, M$ for all frequencies $\lambda \in \R$ is \emph{equivalent} to exact observability of $\{e^{\i t L}\}_{t \in \R}$  via $D^\ast$; see \cite[Theorem 5.1]{Mil05}.
\end{rem}


The following proposition can be seen as a second-order version of \cite[Theorem 3.2]{ChiPau23} and will allow us to go from non-uniform Hautus test to a high-frequency second-order resolvent estimate for \cref{eq:SO-feedback}. Just as in \cite{ChiPau23}, in the absence of input-output well-posedness for \cref{eq:w-u-cauchy},
 this procedure comes at the cost of multiplicative  terms penalising transfer function growth.

\begin{prop}[Second-order resolvent estimate]\label{prop:SO-res} Suppose that
\begin{enumerate}[label=(\arabic*),series=energy]
  \item \label{it:hautus} $(\i L, D^\ast)$ satisfies the non-uniform Hautus test with parameters $m$, $M$ and $\lambda_0$;
  \item \emph{(Transfer function growth.)} \label{it:reg-eta-SO} There exists $0 \leq \eta \leq 1$ such that
\begin{equation}
  \label{eq:reg-eta-SO}
  \|D^\ast((1+\i \lambda)^2 + L)^{-1}D\|_{\L(U)} = O(|\lambda|^{2\eta -1}), \quad \lambda \in \R, \quad |\lambda| \to + \infty.
\end{equation}
\end{enumerate}
Then, for $\lambda \in \R$ with $|\lambda|$ sufficiently large, $(L + \i \lambda DD^\ast - \lambda^2)^{-1}$ is well-defined as a map on $H$ and
\begin{equation}
\|(L+\i \lambda DD^\ast - \lambda^2)^{-1}\|_{\L(H)} = O(|\lambda|^{1 + 4\eta}(M(\lambda)^2 + m(\lambda)^2)), \quad \lambda \in \R, \quad |\lambda|\to +\infty.
\end{equation}
\end{prop}
\begin{proof}
For brevity we will write
\begin{equation}
L(\lambda) \triangleq L + \i \lambda DD^\ast - \lambda^2, \quad
R(\lambda) \triangleq  ((1 + \i \lambda)^2 + L)^{-1}, \quad  \lambda \in \R.
\end{equation}
Observe that  $L(\lambda) \in \L(H_1, H_{-1})$, $L(\lambda)^\ast = L(- \lambda)$ and
 $R(\lambda) \in \L(H_{-1}, H_1)$ for all $\lambda \in \R$.
The first step of the proof consists in deriving an \emph{a priori} estimate.
Let $\lambda \in \R$ with $|\lambda| \geq \lambda_0$, and let $w \in H_1$ be such that $L(\lambda)w \in H$.
Define $v \triangleq w + \i \lambda R(\lambda)DD^\ast w$. Then
\begin{equation}
\begin{aligned}
\label{eq:L-lambda-v}
(L-\lambda^2)v &= (L-\lambda^2)w + \i \lambda (L + 1 + 2\i\lambda - \lambda^2)R(\lambda)DD^\ast w - \i \lambda (1 + 2 \i\lambda) R(\lambda)DD^\ast w \\ 
&= (L - \lambda^2)w + \i \lambda DD^\ast w - \i (1 + 2\i \lambda)\lambda R(\lambda) DD^\ast w = L(\lambda)w - \i (1 + 2\i\lambda) \lambda R(\lambda) DD^\ast w.
\end{aligned}
\end{equation}
A first consequence of \cref{eq:L-lambda-v} is that $Lv \in H$, i.e., $v \in H_2$.
Therefore, we may apply the Hautus test estimate \cref{eq:hautus} to $v$ and obtain, making further use of \cref{eq:L-lambda-v},
\begin{multline}
\label{eq:long-est-v}
\|v\|_H \leq M(\lambda) \|L(\lambda)w\|_H + M(\lambda)(1 + 2|\lambda|) \|\lambda R(\lambda)D\|_{\L(U, H)} \|D^\ast w\|_U \\ 
 + m(\lambda)\|D^\ast w\|_U + m(\lambda) \|\lambda D^\ast R(\lambda)D\|_{\L(U)} \|D^\ast w\|_{U}.
\end{multline}
Suppose now that $|\lambda| \geq 1$. We also have
$
\|w\|_H \leq \|v\|_H + \|\lambda R(\lambda)D\|_{\L(U, H)} \|D^\ast w\|_U
$,
thus
\begin{multline}
\label{eq:long-est-w}
\|w\|_H \leq M(\lambda) \|L(\lambda)w\|_H +  3|\lambda| M(\lambda) \|\lambda R(\lambda)D\|_{\L(U, H)} \|D^\ast w\|_U  \\ + m(\lambda)\big(1 + \|\lambda D^\ast R(\lambda)D\|_{\L(U)}\big) \|D^\ast w\|_U.
\end{multline}
On the other hand, 
taking the imaginary part of the identity
$
\langle L(\lambda) w, w \rangle_H = \|L^{1/2}w\|^2_H + \i \lambda \|D^\ast w\|^2_U - \lambda^2 \|w\|^2_H
$
and using the Cauchy--Schwarz inequality give
\begin{equation}
\label{eq:est-P-w}
|\lambda| \|D^\ast w\|_U^2\leq \|L(\lambda)w\|_{H}\|w\|_H,
 \qquad 
\|D^\ast w\|_U \leq |\lambda|^{-1/2} \|L(\lambda)w\|_H^{1/2} \|w\|_H^{1/2}. 
\end{equation}
Substituting \cref{eq:est-P-w} into \cref{eq:long-est-w} yields
\begin{multline}
\label{eq:long-est-w-bis}
\|w\|_H \leq M(\lambda) \|L(\lambda)w\|_H +3|\lambda|^{1/2}M(\lambda) \|\lambda R(\lambda)D\|_{\L(U, H)} \|L(w)\|^{1/2}_H \|w\|^{1/2}_H \\ + m(\lambda) \big(1 + \|\lambda D^\ast R(\lambda)D\|_{\L(U)}\big) \|L(\lambda)w\|_H^{1/2} \|w\|_H^{1/2}.
\end{multline}
We may use Young's inequality and deduce from \cref{eq:long-est-w-bis} that
\begin{equation}
\label{eq:lambda-apriori-bF}
 \|w\|_H \leq \left(
4M(\lambda) + 18 |\lambda| M(\lambda)^2 \|\lambda R(\lambda)D\|_{\L(U, H)}^2 + 2 m(\lambda)^2  \big(1 + \|\lambda D^\ast R(\lambda)D\|_{\L(U)}\big)^2
\right) \|L(w)\|_H.
\end{equation}
Now we shall prove that $L(\lambda)^{-1}$ is well-defined in $\L(H)$ for large $|\lambda|$ and give an estimate of its operator norm. Denote by $F(\lambda)$ the sum of terms in the outermost parentheses in \cref{eq:lambda-apriori-bF}. At this point we have proved that,
 for arbitrary $\lambda \in \R$ with $|\lambda| \geq \max\{\lambda_0, 1\}$ and $w \in H_1$ such that  $L(\lambda)w \in H$, $
\|w\|_H \leq F(\lambda) \|L(\lambda) w\|_H$. In particular,
 $\|L(-\lambda)w\|_H \geq  F(-\lambda)\|w\|_H$ for all $w \in H_1$ such that $L(-\lambda)w \in H$. Consider $L(\lambda)$ as an unbounded operator on $H$. One may check
that the $H$-realisation of $L(\lambda)$ is a (densely defined) closed operator; furthermore, its adjoint is the $H$-realisation of $L(\lambda)^\ast$. Since $L(-\lambda) = L(\lambda)^\ast$ and $F(- \lambda) > 0$,
we have also proved that $L(\lambda)^\ast$ is bounded below. Therefore, by \cite[Theorem 2.20]{Bre11book}, $L(\lambda) : \dom(L(\lambda))\to H$ is surjective. We moreover have $\|L(\lambda)w\|_H \geq  F(\lambda) \|w\|_H$ for all $w \in \dom(L(\lambda))$, which means that $L(\lambda)$ is also injective (as $F(\lambda) > 0$) and $\|L(\lambda)^{-1}\|_{\L(H)} \leq F(\lambda)$.
Finally we use \cref{it:reg-eta-SO} to give a more explicit estimate. Recalling \cref{eq:norm-P-Past,eq:passive-input-TF} from the proof of \cref{theo:collocated-SO}, we have
\begin{equation}
\| \lambda R(\lambda)D\|^2_{\L(U, H)} = O(|\lambda|^{2\eta}), \quad  \|\lambda D^\ast R(\lambda)D\|^2_{\L(U)} = O(|\lambda|^{4\eta}), \quad \lambda \in \R, \quad |\lambda| \to +\infty,
\end{equation}
which, with crude estimates, leads to $F(\lambda) = O(|\lambda|^{1 + 4\eta}(M(\lambda)^2 + m(\lambda)^2))$ and completes the proof.
\end{proof}


The second-order resolvent estimate of \cref{prop:SO-res} is the main ingredient for obtening quantified energy decay of solutions to \cref{eq:SO-feedback}, as shown in the next theorem.
At this stage, the technical machinery that we need in order to conclude is already contained in the literature \cite{Leb96proceedings,BorTom10,AnaLea14,ChiPau23,KleWan23arxiv},
so we will just sketch the arguments.
\begin{theo}[Non-uniform energy decay rates]
\label{theo:decay-rate-w}
In addition to the hypotheses \cref{it:hautus,it:reg-eta-SO} of \cref{prop:SO-res}, suppose that
\begin{enumerate}[label=(\arabic*),resume=energy]
   \item \label{it:compact}\emph{(Compact resolvent.)} For some (hence all) $\mu \in \rho(L)$, the operator $(L - \mu)^{-1}$ is compact on $H$;
   \item \emph{(Low-frequency unique continuation.)} \label{it:LFUCP} For all $\lambda \in \R$ with $0 < |\lambda| < \lambda_0$, if $w \in H_1$ satisfies $w \not = 0$ and $(L - \lambda^2)w = 0$, then $D^\ast w \not = 0$;
   \item \emph{(Polynomial growth.)} \label{it:pol-mm} The functions $m$ and $M$ are of the form $M(\lambda) = M_0(1 + |\lambda|^\alpha)$ and $m(\lambda) = m_0(1 + |\lambda|^\beta)$ for some $\alpha, \beta \geq 0$ and $m_0, M_0 \geq 1$.
\end{enumerate}
Let $(w_0, w_1) \in H_1 \times H_1$ be initial data such that $DD^\ast w_1 + Lw_0 \in H$, i.e., $(w_0, w_1) \in \dom(A_D)$.
Then, the corresponding (classical) solution $w$ to \cref{eq:SO-feedback} satisfies
\begin{equation}
\label{eq:decay-rate-w}
E(t; w) =
o(t^{-1/(1 +2\eta +\alpha + \beta)}), \quad t \to + \infty.
\end{equation}
Without assumptions \cref{it:compact,it:LFUCP}, the same conclusion is valid if we assume instead that $0 \in \rho(L)$ and the non-uniform Hautus test \cref{it:hautus} holds for all frequencies $\lambda \in \R$.
\end{theo}
\begin{proof}[Proof (sketch)]
If $0$ lies in the spectrum of $L$, assuming that $L$ has compact resolvent, one can show that $0$ is an isolated eigenvalue of $A_D$; see \cite[Lemma 3.4]{KleWan23arxiv}. Note that the argument is particularly delicate when $D$ has maximal unboundedness, i.e., when $D \not \in \L(U, H_{\eps-1})$ for any $\eps > 0$. This allows us to decompose, via Riesz projections, the phase space $X = H_1 \times H$ into the (in general, non-orthogonal) direct sum of $\ker(A_D)$ and a suitable (topological) complement $\mathring{X}$. The subspace $\mathring{X}$ can be renormalised with the energy seminorm (that is, the square root of the expression in \cref{eq:energy-semi}) and then
$\sg{S^D}$ restricted to $\mathring{X}$ becomes a contraction semigroup; see \cite[Lemma 3.6]{KleWan23arxiv} for more details. This decomposition is standard in the context of bounded control operator; see, e.g., \cite[Section II.4]{AnaLea14}. In short, polynomial energy decay of classical solutions to \cref{eq:SO-feedback} is  equivalent to polynomial stability in $\mathring{X}$ of the restricted semigroup. By \cite[Theorem 2.4]{BorTom10}, the latter is characterised by the growth in $\L(\mathring{X})$-norm of $(\i \lambda - \mathring{A}_D)^{-1}$  as $\lambda \in \R$, $\lambda \to + \infty$. Here the unique continuation property guarantees that $\i\lambda$ lies in the resolvent set of $\mathring{A}_D$ for all $\lambda \in \R$, where $\mathring{A}_D$ is the restriction of $A_D$ on $\mathring{X}$; see \cite[Lemma 3.11]{KleWan23arxiv}. Furthermore, there exist $\lambda_1, K, K' > 0$ such that, for all $\lambda \in \R$ with $|\lambda| \geq \lambda_1$,
\begin{equation}
\|(\i \lambda - \mathring{A}_D)^{-1}\|_{\L(\mathring{X})}  \leq K  \|(\i \lambda - {A}_D)^{-1}\|_{\L({X})} \leq K' + K' |\lambda|\|(L+\i \lambda DD^\ast - \lambda^2)^{-1}\|_{\L(H)};
\end{equation}
see for instance \cite[Lemmas 3.6 and 3.9]{KleWan23arxiv}.
In particular, by \cref{prop:SO-res},
\begin{equation}
\label{eq:res-growth-AP}
\|(\i \lambda - \mathring{A}_D)^{-1}\|_{\L(\mathring{X})} = O(|\lambda|^{2 + 4\eta + 2\alpha + 2\beta}), \quad \lambda \in \R, \quad |\lambda| \to + \infty,
\end{equation}
which in turn leads to the energy decay rate \cref{eq:decay-rate-w} for initial data $(w_0, w_1)$ in $\dom(A_D)$. 
In the case that $0 \in \rho(L)$, the energy and the squared norm in $X$ are proportional, and $0 \in \rho(A_D)$. When $L$ has compact resolvent, the above arguments directly apply to $A_D$ and the result follows. In the absence of compact resolvent for $L$ but under the additional assumption that the non-uniform Hautus test \cref{eq:hautus} holds for all frequencies $\lambda \in \R$,
\cite[Proposition 3.10]{ChiPau23}  gives $\i \R \subset \rho(A_D)$ and the resolvent growth \cref
{eq:res-growth-AP} for $A_D$.
Again, the energy decay rate \cref{eq:decay-rate-w}  follows from \cite[Theorem 2.4]{BorTom10}.
\end{proof}
\begin{rem}
The assumption \cref{it:pol-mm} that $m$ and $M$ are polynomials is included mostly for the sake of simplicity. 
Using \cite[Theorem 3.2]{RozSei19} instead of \cite[Theorem 2.4]{BorTom10}, one can consider functions $m$ and $M$ that are continuous, non-decreasing and \emph{of positive increase} (and obtain an energy decay rate involving suitable inverses of $m$ and $M$). Note however that, in this more general setting, decay rates are given in ``big-O'', as opposed to the ``small-o'' improvement of \cite{BorTom10} for the scale of polynomial growths.
\end{rem}


\subsection{The wave equation with Neumann boundary damping}

Let $\Omega$ be a bounded Lipschitz domain of $\R^d$, $d \geq 2$.
Fix some function  $b \in L^\infty(\partial \Omega)$ with $b \geq 0$.
We now consider the feedback problem
\begin{subequations}
\label{eq:neumann-wave-feedback}
\begin{align}
&(\partial^2_t - \Delta)w = 0&&\mbox{in}~\Omega \times (0, +\infty), \\
&\partial_{\vec{n}} w = - b^2 \partial_t w&&\mbox{on}~\partial \Omega \times (0, +\infty), \\
&(w, \partial_t w)|_{t =0} = (w_0, w_1) &&\mbox{in}~\Omega.
\end{align}
\end{subequations}
In what follows, the (continuous, self-adjoint) multiplication operator on $L^2(\partial \Omega)$ associated with $b$ is denoted by the same letter. We may  recast \cref{eq:neumann-wave-feedback} in the abstract framework of \cref{sec:stability}:
\begin{equation}
\ddot{w} + \gamma^\ast b^2 \gamma \dot{w} - \Delta_N w = 0, \quad w(0) = w_0, \quad \dot{w}(0) = w_1.
\end{equation}
 Solutions to \cref{eq:neumann-wave-feedback} give rise to a strongly continuous semigroup on the phase space $H^1(\Omega) \times L^2(\Omega)$. The energy is given by
\begin{equation}
E(t; w) =  \frac{1}{2} \int_\Omega |\partial_t w(x, t)|^2 + \|\vec{\nabla} w(x, t)\|^2 \, \d x, \quad \mbox{$w$ solves \cref{eq:neumann-wave-feedback}}.
\end{equation}

\begin{prop}[\emph{A priori}
 decay rates under non-uniform Hautus test] Let $\Omega$ and $\eta$ be as in \cref{eq:reminder-loss}.
\label{theo:apriori-decay-waves} Assume that there exist $m_0, M_0 > 0$ and $\alpha,\beta,\lambda_0 \geq 0$ such that
\begin{equation}
\label{eq:hautus-waves}
\|w\|_{L^2(\Omega)} \leq M_0(1 + |\lambda|^\alpha) \|(\Delta - \lambda^2)w\|_{L^2(\Omega)} + m_0(1 + |\lambda|^\beta) \|b w|_{\partial \Omega}\|_{L^2(\partial\Omega)}
\end{equation}
for all $w \in H^2(\Omega)$ and $\lambda \in \R$ with $|\lambda| \geq \lambda_0$. Then, for all initial data $(w_0, w_1) \in H^1(\Omega) \times H^1(\Omega)$ such that $\Delta w_0 \in L^2(\Omega)$ and $\partial_{\vec{n}}w_0 = - b w_1$ on $\partial \Omega$, the corresponding solution $w$ to \cref{eq:neumann-wave-feedback} satisfies
\begin{equation}
E(t; w) = o(t^{-1/(1 +2\eta +\alpha + \beta)}), \quad t \to + \infty.
\end{equation}
\end{prop}
\begin{proof}
This is mostly a direct application of \cref{theo:decay-rate-w} but we have some hypotheses to check.  Here, $\Omega$ is assumed to be bounded:
by, e.g., \cite[Theorem 1.4.3.2]{Gri85book}, $H^1(\Omega)$ is compactly embeddeded in $L^2(\Omega)$, which implies that
 the Neumann Laplacian $\Delta_N$ has  compact resolvent.
\Cref{it:LFUCP} in the theorem boils down to a standard application of the John--Holmgrem theorem on unique continuation across non-characteristic hypersurfaces for differential operators with analytic coefficients; see, e.g. \cite[Uniqueness Theorem, Section 5]{Joh82book}. Now, condition \cref{it:reg-eta-SO} in \cref{prop:SO-res} requires an estimate of the high-frequency growth of $\lambda \to b \gamma ((1 + \i \lambda)^2 - \Delta_N)^{-1}\gamma^\ast b$ in $\L(L^2(\partial\Omega))$-norm, but it is clear that, for all  $\lambda \in \R$,
\begin{equation}
\label{eq:bTF-TF}
\|b \gamma ((1 + \i \lambda)^2 - \Delta_N)^{-1}\gamma^\ast b\|_{\L(L^2(\partial \Omega))} \leq \|b\|^2_{\L(L^2(\partial \Omega))} \|\gamma ((1 + \i \lambda)^2 - \Delta_N)^{-1}\gamma^\ast\|_{\L(L^2(\partial \Omega))},
\end{equation}
and by \cref{coro:trace-TF} the right-hand side of \cref{eq:bTF-TF} grows like $O(|\lambda|^{\eta-1})$ as $|\lambda| \to + \infty$.
\end{proof}

The remainder of the section is devoted to the special case where $\Omega$ is a rectangle in $\R^2$. In that case we are able to derive explicit energy decay rates by combining boundary observability properties already established in \cite{RamTak05,TenTuc09}  and our transfer function estimate \cref{eq:trace-TF}.
\begin{lemma}[Non-uniform Hautus test]
\label{lem:hautus-rectangle} Assume that $\Omega$ is a rectangle of $\R^2$ and let $\Op$ be a nonempty open subset of $\partial \Omega$.
There exist $m_0, M_0 > 0$ such that
\begin{equation}
\label{eq:hautus-rectangle}
\|w\|_{L^2(\Omega)} \leq M_0\|(\Delta - \lambda^2)w\|_{L^2(\Omega)} + m_0 \|w|_{\partial \Omega}\|_{L^2(\Op)}, \quad \lambda \in \R.
\end{equation}
\end{lemma}
\begin{proof}
First, \cite[Theorem 1.1]{TenTuc09} states that the Schr\"odinger group generated by $- \i \Delta_N$ on $L^2(\Omega)$ is exactly observable with respect to the observation operator $\mathds{1}_\Op \gamma$: there exist $K, T > 0$ such that, letting $\Psi(t) \triangleq e^{-\i t \Delta_N} w_0$,
\begin{equation}
\label{eq:obs-shc-rect}
\int_\Omega |w_0(x)|^2 \, \d x  \leq K
\iint_{\Op \times (0, T)} | \Psi(x, t)|^2 \, \d \sigma \, \d t, \quad w_0 \in \dom(\Delta_N).
\end{equation}
Furthermore, according to \cite[Proposition 3.1]{RamTak05}, the observation operator $\mathds{1}_\Op \gamma$ is admissible for  $\{e^{-\i t\Delta_N}\}_{t \in \R}$ on $L^2(\Omega)$. Thus, as mentioned in \cref{rem:sch-obs}, we may apply \cite[Theorem 5.1]{Mil05} and directly deduce \cref{eq:hautus-rectangle} from 
 \cref{eq:obs-shc-rect}. 
\end{proof}

\begin{theo}[Energy decay rates for the rectangle] Assume that $\Omega$ is a rectangle in $\R^2$. Suppose also that there exist a positive constant $b_0$ and an open subset $\Op$ of $\partial \Omega$ such that $b \geq b_0$ a.e.\  on $\Op$.
Then, for all initial data $(w_0, w_1) \in H^1(\Omega) \times H^1(\Omega)$ such that $\Delta w_0 \in L^2(\Omega)$ and $\partial_{\vec{n}}w_0 = - b w_1$ on $\partial \Omega$, the corresponding solution $w$ to \cref{eq:neumann-wave-feedback} satisfies, for any fixed $\eps > 0$,
\begin{equation}
E(t; w) = o(t^{-2/3 + \eps}), \quad t \to + \infty.
\end{equation}
\end{theo}
\begin{proof}
 We  apply \cref{lem:hautus-rectangle} to $\Op$ and get \cref{eq:hautus-rectangle}, which we may rewrite in the form of \cref{eq:hautus-waves} since $b$ is bounded from below  on $\Op$. In the rectangle case, we may choose $\eta = 1/4 + \eps$, and the result  follows from \cref{theo:apriori-decay-waves}.
\end{proof}







\section{Concluding remarks}
\label{sec:conc}

In this section we give some comments and perspectives.

\begin{itemize}
  \item Our frequency-domain estimates \cref{eq:FQ-cond-C,eq:input-FQ,eq:FQ-IO} can  be refined with  explicit control over the real part of the Laplace variable. For instance, under the hypotheses of \cref{theo:inter-group} we actually have, for fixed (small) $\eps > 0$,
  \begin{equation}
  \label{eq:weiss-type}
  \|C(\sigma + \i \omega - A)^{-1}\|_{\L(X, Y)} \leq \frac{K(1 + |\omega|^\eta)}{(\sigma - \sigma_0 - \eps)^{1/2}}, \quad \sigma > \sigma_0 + \eps, \quad \omega \in \R.
  \end{equation}
  \Cref{eq:weiss-type} may be seen as a Weiss-type resolvent estimate with growth along vertical lines. On the other hand,
  via lifting to the space $X_\eta^\fra$ or $X_\eta^\inte$, techniques and problems related to the so-called Weiss property (see, e.g., \cite[Section 3.2]{JacPar04survey} or the literature review in \cite{PreSch24arxiv}) for observation operators extend to our framework.

  \item The question of admissibility is of course also relevant in the Banach space setting, and the exponent $p \geq 1$ when considering $L^p$-based control spaces constitutes an additional tuning parameter to obtain smoother solutions. For instance, for the heat equation with Neumann boundary condition, the space of solutions produced by boundary data $u \in L^p(0, T; L^2(\partial \Omega))$ lies in $L^2(\Omega)$ if and only if $p \geq 3/4$; see \cite{HaaKun07}. Extending our results to semigroups on Banach spaces would be an interesting but delicate endeavour. In the absence of a vector-valued Plancherel's theorem, many of the arguments in our proofs break down. What is more, in contrast with the quadratic, Hilbert space case, different interpolation techniques (e.g., real and complex interpolation) may give different (non-isomorphic) spaces, which adds a layer of complexity when building Sobolev scales for the semigroup data and inputs or outputs.

\end{itemize}


\appendix

\section{On Hilbert-valued Sobolev spaces}
\label{sec:add-hilb}

This section contains technical results which are, for the most part,
 generalisation to the vector-valued case of more or less classical facts about scalar-valued Sobolev spaces, as found for instance in \cite[Chapter 1]{LioMag68book}.
We do not claim any novelty here: closely related results may be found scattered across the literature, however, often in forms which do not meet our exact requirements. For example, \cite{Ama00,Ama19book} deal with more general (in particular, $L^p$-based, Banach-valued) function spaces but avoid the ``critical'' parameters for fractional order Sobolev spaces (in our case, $s \in 1/2 + \N$), which we need.
Note also that a number of simplifications are made owing to the one-dimensional variable in our setting. In what follows, given a Hilbert space $E$, we shall use without further comment that the trace map $u \mapsto u(0)$ is well-defined and continuous from $H^{1/2+\eps}(0, +\infty; E)$ into $E$ for any $\eps > 0$; see, e.g., \cite[Corollary 27]{Sim90}.
The first lemma is an implementation of the method of extension by reflection.
\begin{lemma}[Extension by reflection]
\label{lem:reflec-plus}
Let $E$ be a Hilbert space and $m \geq 1$ be an integer. There exists a linear map $\pi_m : L^2\hl{E} \to L^2(\R, E)$ that satisfies the following properties:
\begin{enumerate}[series=pi]
  \item\label{it:pi-ext} $\pi_m u = u$ a.e.\ on $(0, +\infty)$ for all $u \in L^2\hl{E}$;
  \item\label{it:pi-cont} $\pi_m \in \L(H^{s}\hl{E}, H^s(\R, E))$ for all $0 \leq s \leq m$;
  \item\label{it:pi-diff}There exists another linear map $\pi'_m : L^2\hl{E} \to L^2(\R, E)$ such that 
  \begin{subequations}
  \label{eq:pi-prime}
  \begin{align}
  \label{eq:cont-pi-prime}
  &\pi'_m \in  \L(H^{s}\hl{E}, H^s(\R, E)), \quad 0 \leq s \leq m-1, \\
  \label{eq:commut-ext}
  & (\d / \d t) \pi_m = \pi_m' (\d / \d t).
  \end{align}
  \end{subequations}
\end{enumerate}
\end{lemma}
\begin{proof}
The construction is similar to that of \cite[Lemma 12.2, Chapter 1]{LioMag68book}.
First, we claim that there exist real numbers $\alpha_1, \dots, \alpha_{2m}$ and $\beta_1, \dots, \beta_{2m}$ such that the $\beta_k$ are negative and
\begin{equation}
\label{eq:algebraic-alpha}
\sum_{k=1}^{2m} \alpha_k \beta_k^j = 1, \quad j = -m, \dots, m -1.
\end{equation}
{This follows from a simple linear algebra argument.} Then, for any $u \in L^2\hl{E}$, we define $\pi_m u$ a.e.\  in $\R$ by
\begin{equation}
[\pi_mu](t) \triangleq u(t) \quad \mbox{if}~t > 0, \quad [\pi_m u](t) \triangleq \sum_{k = 1}^{2m} \alpha_k u(\beta_k t) \quad \mbox{if}~t < 0.
\end{equation}
Clearly, $\pi_m$ is linear and continuous $L^2\hl{E} \to L^2(\R, E)$. Furthermore, when $u \in H^m\hl{E}$, we have
\begin{equation}
\frac{\d^j}{\d t^j}[\pi_mu](t) \triangleq \sum_{k=1}^{2m} \alpha_k \beta^j_k \frac{\d^j u}{\d t^j}(\beta_k t) \quad \mbox{for a.e.\ } t < 0, \quad j = 0, \dots, m.
\end{equation}
It follows from \cref{eq:algebraic-alpha} that
\begin{equation}
\frac{\d^j u}{\d t^j}(0) = \frac{\d^j}{\d t^j}[\pi_m u](0), \quad j = 0, \dots, m-1,
\end{equation}
which in turn allows us to prove that $\pi_m$ is in fact continuous $H^m\hl{E} \to H^m(\R, E)$. By interpolation, it must also be continuous $H^s\hl{E} \to H^s(\R, E)$ for all $0 \leq s \leq m$. Finally, for $u \in L^2\hl{E}$,  $\pi'_m u$ is defined a.e.\  by
\begin{equation}
\label{eq:def-pi-prime}
[\pi'_mu](t) \triangleq u(t) \quad \mbox{if}~t > 0, \quad [\pi'_m u](t) \triangleq  \sum_{k = 1}^{2m} \alpha_k \beta_k u(\beta_k t) \quad \mbox{if}~t < 0,
\end{equation}
so that the identity \cref{eq:commut-ext} holds. Again, continuity of $\pi_m' : L^2\hl{E}\to L^2(\R, E)$  is immediate, and
continuity $H^{m-1}\hl{E} \to H^{m-1}(\R, E)$
is obtained by making use of the algebraic condition \cref{eq:algebraic-alpha}, and finally \cref{eq:cont-pi-prime} follows by interpolation.
\end{proof}
In preparation for the characterisation of $H^s_0$-spaces, we establish the following lemma, which is a somewhat simplified version of \cite[Lemmas 11.1 and 11.2, Chapter 1]{LioMag68book}.
\begin{lemma}
\label{lem:D0-V}
Let $E$ be a Hilbert space and $s \geq 0$. If $0 \leq s \leq 1/2$, let $V \triangleq H^s(\R, E)$. If $s > 1/2$,
define $V$ as follows:
\begin{equation}
\label{eq:def-V-0}
V \triangleq \{ u \in H^{s}(\R, E) : u^{(j)}(0) = 0,~ j \in \N,~ j < s - 1/2\},
\end{equation}
where $u^{(j)} = (\d^j/ \d t^j) u$, equipped with the norm of $H^{s}(\R)$.  Denote by $\D_0$ the subspace of $\D(\R, E)$ comprised of all those $\varphi$ which vanish in a  neighbourhood of $0$. Then, $\D_0$ is dense in $V$.
\end{lemma}
\begin{proof}
Using continuity $H^{1/2+\eps}(\R, E) \to E$ of the trace map, we see that $V$ is well-defined and in fact a closed subspace of $H^s(\R, E)$, making it a Hilbert space. Recalling the Fourier characterisation \cref{eq:bessel-spaces} of $H^{s}(\R, E)$,  the linear map $u \mapsto (\F[u], |\cdot|^{s}\F[u])$ is continuous $H \to L^2(\R) \times L^2(\R)$ and injective. This allows us to regard $V$ as a closed subspace of $L^2(\R, E) \times L^2(\R, E)$. In view of that embedding and by the Hahn--Banach theorem, any 
$u \in V^\ast$ extends to a continuous antilinear form on $L^2(\R, E) \times L^2(\R, E)$. Having identified $L^2(\R, E) \times L^2(\R, E)$ with its antidual, this means that, for every $u \in V^\ast$, there exists a (not unique, in general) pair $(f, g) \in L^2(\R, E) \times L^2(\R, E)$ such that
\begin{equation}
\label{eq:id-V-HB}
\langle u, v \rangle_{V^\ast, V} = \int_\R \langle f(\omega) + |\omega|^{s}g(\omega), \F[v](\omega)\rangle_E \, \d \omega, \quad v \in V.
\end{equation}
To show that the subspace $\D_0$ is dense in $V$, we shall prove that its orthogonal is trivial. Let $u \in V^\ast$ be such that 
$\langle u, \varphi\rangle_{V^\ast, V} = 0$. We find $(f, g)$ so that $u$ can be written as in \cref{eq:id-V-HB} and let $h \triangleq \F^{-1}[f +  |\cdot|^s g]$. Note that \emph{a priori} $h$ is defined as a Schwartz distribution in  $\S'(\R, E)$. Moreover $\D_0 \subset \S(\R, E)$. It follows from \cref{eq:id-V-HB} together with the distributional Parseval--Plancherel formula \cite[Proposition 1.7.4, Appendix]{Ama19book} that 
$
0 = \langle u, \varphi\rangle_{V^\ast, V} = \langle h, \varphi\rangle_{\S'(\R, E), \S(\R, E)}$
 for all $\varphi \in \D_0$, i.e., for all Schwartz functions $\varphi$ that vanish in a neighbourhood of $0$.
It follows that the support of $h$ (as a distribution) is contained in $\{0\}$, and there must therefore exist a finite number of vectors $h_0, \dots, h_k \in E$ such that 
\begin{equation}
\label{eq:vec-h-sum}
h = \sum_{j=0}^{k} \delta^{(j)} \otimes h_j,
\end{equation}
where $\delta^{(0)} = \delta$ is the Dirac mass at $0$, $\delta^{(j)} = (\d / \d t)\delta^{(j - 1)}$ for $j \geq 1$, and $\otimes$ denotes the tensor product; see \cite[Example 1.8.6(d), Appendix]{Ama19book}. Taking the Fourier transform of \cref{eq:vec-h-sum} yields $f(\omega) + |\omega|^{s}g(\omega) = \sum_{j=0}^k(\i \omega)^j h_j$ for a.e.\  $\omega \in \R$. Now, by construction of $h$, we also have
\begin{equation}
\label{eq:L2-weight-cond}
\frac{1}{1 + |\cdot|^s} \F[h]   \in L^2(\R, E), \quad \mbox{i.e.}, \quad \int_{\omega \in \R} \left\|\sum_{j=0}^k \frac{(\i \omega)^j}{(1 + |\omega|)^{s}} h_j \right\|^2 \d \omega < + \infty.
\end{equation}
The $L^2$-condition of \cref{eq:L2-weight-cond} implies that  $h_j = 0$  whenever $j \geq s - 1/2$. In particular, in the case $0 \leq s \leq 1/2$, we immediately obtain $h = 0$, hence $u = 0$ as required. As for the case $s > 1/2$, we let $v \in V$ and write, using \cref{eq:id-V-HB,eq:vec-h-sum},
\begin{equation}
\label{eq:zero-u-v}
\langle u, v \rangle_{V^\ast, V} = \sum_{\substack{j \in \N \\ j < s - 1/2}} (-1)^j \langle v^{(j)}(0), h_k\rangle_E = 0
\end{equation}
by \cref{eq:def-V-0}. \Cref{eq:zero-u-v} holds for arbitrary $v \in V$,  thus $u = 0$ and the proof is complete. 
\end{proof}

The next proposition comes as a straightforward consequence of \cref{lem:D0-V}.
\begin{prop}[Characterisation of $H^s_0$-spaces]\label{coro:charac-Hs-0}
Let $E$ be a Hilbert space. Then
\begin{subequations}
\label{eq:charac-Hs-0}
\begin{align}
&H^{s}_0\hl{E} = H^s\hl{E}, && 0 \leq s \leq 1/2, \\
&H^{s}_0\hl{E} = \{u \in H^{s}\hl{E}: u^{(j)}(0) = 0,~j \in \N,~ j < s - 1/2\}, && s > 1/2.
\end{align}
\end{subequations}
\end{prop}
\begin{proof}
Let $s \geq 0$ and $m \in \N$ with $s \leq m$. The inclusions from left to right in \cref{eq:charac-Hs-0} are clear and the inclusions from right to left amount to density of $\D\hl{E}$ in the sets on the right-hand side (with respect to the $H^s\hl{E}$-norm).  We will use an  extension map $\pi_m$ from \cref{lem:reflec-plus}. Let $u \in H^{s}\hl{E}$ if $0 \leq s \leq 1/2$ or, if $s > 1/2$, $u \in H^{s}\hl{E}$ such that $u^{(j)}(0) = 0$ for $j \in \N$ with $j < s - 1/2$.
 By \cref{lem:reflec-plus}, $\pi_m u \in H^{s}(\R, E)$, and by \cref{lem:D0-V}, there exists a sequence of test functions $\varphi_n \in \D(\R, E)$, with each $\varphi_n$ vanishing in a neighborhood of $0$, such that $\varphi_n \to \pi_m u$ in $H^{s}(\R, E)$. Then, for all $n$,  $\varphi_n$ also belongs (after restriction) to $\D\hl{E}$ and
\begin{equation}
\|u - \varphi_n\|_{H^{s}\hl{E}} = \|\pi_m u - \varphi_n\|_{H^{s}\hl{E}} \leq K \|\pi_m u - \varphi_n\|_{H^{s}(\R, E)},
\end{equation}
which proves that $\varphi_n \to u$ in $H^s\hl{E}$, as required.
\end{proof}
\begin{rem}
In the case  $s \not \in \N + 1/2$ we refer the reader to \cite[Section 1.4, Chapter VIII]{Ama19book} for related results.
\end{rem}
With \cref{coro:charac-Hs-0} in hand, we can establish a stronger version of \cref{lem:reflec-plus} which guarantees that the extension map is also continuous in negative order Sobolev norms.

\begin{lemma}[Extension by reflection, continued]
\label{lem:reflec-gen}
Let $E$ be a Hilbert space and $m \geq 1$ be an integer. There exists a linear map $\pi_m : L^2\hl{E} \to L^2(\R, E)$ that satisfies, in addition to the properties \cref{it:pi-ext,it:pi-cont,it:pi-diff} of \cref{lem:reflec-plus}:
\begin{enumerate}[resume=pi]
  \item\label{it:pi-cont-neg} $\pi_m \in \L(H^{-s}\hl{E}, H^{-s}(\R, E))$ for all $0 \leq s \leq m$;
  \item\label{it:pi-diff-neg} The map $\pi'_m$ of \cref{it:pi-diff} can be chosen so that $\pi_m' \in \L(H^{-s}\hl{E}, H^{-s}(\R, E))$ for all $0 \leq s \leq m+1$.
\end{enumerate}
\end{lemma}
\begin{proof}
We construct $\pi_m \in \L(L^2\hl{E}, L^2(\R, E))$ exactly as in \cref{lem:reflec-plus} and we observe that its adjoint $\pi_m^\ast \in \L(L^2(\R, E), L^2\hl{E}$ is given by
\begin{equation}
\label{eq:pi-m-ast}
[\pi^\ast_m u](t) = u(t) - \sum_{k = 1}^{2m}  \alpha_k \beta_k^{-1} u(\beta_k^{-1} t)\quad \mbox{for a.e.\ }t > 0, \quad u \in L^2(\R, E).
\end{equation}
Let $0 \leq s \leq m$.
Using duality, we see that  $\pi_m \in \L(H^{-s}\hl{E}, H^{-s}(\R, E))$ is equivalent to $\pi_m^\ast \in \L(H^{s}(\R, E), H^s_0\hl{E})$. It is clear from \cref{eq:pi-m-ast} that $\pi_m^\ast$ is continuous $H^s(\R, E) \to H^s\hl{E}$ and it remains to prove that it maps $H^s(\R, E)$ into $H^s_0\hl{E}$. By \cref{coro:charac-Hs-0}, if $0 \leq s \leq 1/2$, there is nothing else to prove, while if $s > 1/2$, it suffices to show that $(\d^j / \d t^j) [\pi^\ast_m u](0) = 0$ for all $j \in \N$ with $j < s - 1/2$ and all $u \in H^{s}(\R, E)$. Assuming that $s > 1/2$, for all such $j$ we have
\begin{equation}
\label{eq:diff-star-pi}
\frac{\d^j}{\d t^j} [\pi^\ast_m u](t) = \frac{\d^j u}{\d t^j}(t)- \sum_{k=1}^{2m} \alpha_k \beta_k^{-j - 1} \frac{\d^j u}{\d t^j}(\beta_k^{-1}t) \quad \mbox{for a.e.\ }t > 0, \quad u \in H^{s}(\R, E),
\end{equation}
and the desired property follows again from the algebraic condition \cref{eq:algebraic-alpha}, concluding the proof of \cref{it:pi-cont-neg}. As for \cref{it:pi-diff-neg}, we define $\pi'_m$ as in \cref{eq:def-pi-prime} and notice that, for any $j = 0, \dots, m$, 
\begin{equation}
\frac{\d^j}{\d t^j}[(\pi_m')^\ast u](t) = \frac{\d^j u}{\d t^j}(t) - \sum_{k=0}^{2m} \alpha_k \beta_k^{-j} \frac{\d^j u}{\d t^j}(\beta_k^{-1}t) \quad\mbox{for a.e.\ }t > 0, \quad u \in H^{m}(\R, E)
\end{equation}
(compare with \cref{eq:diff-star-pi}). This allows us to prove, with the same arguments, that $\pi'$ is actually continous $H^{-s}\hl{E} \to H^{-s}(\R, E)$ for all $0 \leq s \leq m + 1$, and the proof is complete.
\end{proof}

We are finally ready to establish the following simple-looking but useful result.
\begin{prop} 
\label{lem:diff-shift}
 Let $E$ be a Hilbert space and $s \in \R$.
 There exists $K > 0$ such that
\begin{equation}
\label{eq:sandwich-Hs}
K^{-1} \|u\|_{H^{s}\hl{E}} \leq \|u\|_{H^{s-1}\hl{E}} + \|\dot{u}\|_{H^{s - 1}\hl{E}} \leq K \|u\|_{H^s\hl{E}}
\end{equation}
for all $u \in H^{s}\hl{E}$.
\end{prop}
\begin{proof}
First, pick $m \in \N$ such that $-m \leq s \leq m$ and consider an extension map $\pi_m$ as provided by \cref{lem:reflec-gen}.
 Using the Fourier transform and the Bessel potential characterisation \cref{eq:bessel-spaces} of Sobolev spaces on $\R$, it is easy to see that
\begin{equation}
\label{eq:full-line-Hs-pi}
K^{-1} \|\pi_m u \|_{H^s(\R, E)} \leq  \| \pi_m u \|_{H^{s-1}(\R, E)} +  \|(\d/ \d t) \pi_m u\|_{H^{s-1}(\R, E)} \leq \|\pi_m u\|_{H^{s}(\R, E)}
\end{equation}
for all $u \in H^s(\R, E)$. For such $u$, we also have $\|\pi_mu\|_{H^s(\R, E)} \leq K \|\pi_m u\|_{H^s\hl{E}} = K \|u\|_{H^s\hl{E}}$; on the other hand, making use of the various properties of \cref{lem:reflec-plus,lem:reflec-gen} depending on the value of $s$, we also have $\|\pi_m u\|_{H^{s-1}(\R, E)} \leq K \|u\|_{H^{s-1}\hl{E}}$, $\|\pi_m u\|_{H^{s}(\R, E)} \leq K \|u\|_{H^{s}\hl{E}}$, and finally
\begin{equation}
\|(\d/ \d t) \pi_m u\|_{H^{s-1}(\R, E)} = \|\pi'_m \dot{u}\|_{H^{s-1}(\R, E)} \leq \|\dot{u}\|_{H^{s-1}\hl{E}}.
\end{equation}
We may now deduce \cref{eq:sandwich-Hs} from \cref{eq:full-line-Hs-pi} to complete the proof. 
\end{proof}

In our main proofs, we make repeated use of the extension by zero operator. The following proposition gives some of its properties.
\begin{prop}[Extension by zero]
\label{coro:ext-zero}
Let $E$ be a Hilbert space. Let $\pi : L^2\hl{E} \to L^2(\R, E)$ be the map defined as follows: for $u \in L^2\hl{E}$, $[\pi u](t) = u(t)$ for a.e.\  $t > 0$, $[\pi u](t) = 0$ for a.e.\  $t < 0$.
It is continuous
\begin{subequations}
\begin{align}
&H^{s}\hl{E} \to H^s(\R, E), && -1/2 \leq s < 1/2; \\
&H^s_0\hl{E} \to H^{s}(\R, E), && 1/2 < s < 3/2.
\end{align}
\end{subequations}
\end{prop}
\begin{proof} The argument depends on the value of $s$.

\begin{itemize}[wide]
  \item \emph{First case: $-1/2 \leq s \leq 0$.} For the ease of notation we take $0 \leq s \leq 1/2$ and prove that $\pi$ is continuous $H^{-s}\hl{E} \to H^{-s}(\R, E)$. Recalling \cref{eq:chain-H-spaces,eq:dual-norm}, we let $u \in L^2\hl{E}$ and pair $\pi u$ with an arbitrary test function $\varphi \in \D(\R, E)$. Since $\varphi$ (restricted to $(0, +\infty)$) belongs to $H^s\hl{E}$, which coincides with $H^s_0\hl{E}$ by \cref{coro:charac-Hs-0}, there exists a sequence of elements $\varphi_n \in \D\hl{E}$ such that $\varphi_n \to \varphi$ in $H^s\hl{E}$. With this in mind, we can write
  \begin{equation}
  \label{eq:pair-piu-phi}
  |\langle\pi u, \varphi_n \rangle_{L^2(\R, E)}| = |\langle u, \varphi_n \rangle_{L^2\hl{E}}| \leq \|u\|_{H^{-s}\hl{E}} \|\varphi_n\|_{H^s\hl{E}}
  \end{equation}
  for all $n$, where we used that $\varphi_n \in \D\hl{E}$. Letting $n \to +\infty$ in \cref{eq:pair-piu-phi} leads to
  \begin{equation}
  |\langle\pi u, \varphi \rangle_{L^2(\R, E)}| \leq \|u\|_{H^{-s}\hl{E}} \|\varphi\|_{H^s\hl{E}} \leq K \|u\|_{H^{-s}\hl{E}}  \|\varphi\|_{H^s(\R, E)}, 
  \end{equation}
  which gives the desired continuity property.
  \item \emph{Second case: $0 < s < 1/2$.} Here (and in the subsequent cases) we mostly follow the proof of \cite[Theorem 11.4, Chapter 1]{LioMag68book}, which carries over to the vector-valued case. By \cref{coro:charac-Hs-0}, $\D\hl{E}$ is dense in $H^s\hl{E}$ and it suffices to show that there exists $K > 0$ such that
  $\|\varphi\|_{H^{s}(\R, E)} \leq K \|\varphi\|_{H^s\hl{E}}$ for all $\varphi \in \D\hl{E}$. To do so,  we use the Besov-type characterisation of $H^s$-spaces and observe that the required property will follow from the estimate
  \begin{equation}
  \int_0^{+\infty} \frac{1}{\tau^{2s+1}} \int_\R \|\varphi(t + \tau) - \varphi(t)\|^2_E \, \d t \, \d \tau \leq K \int_0^{+\infty} \frac{1}{\tau^{2s+1}} \int_0^{+\infty} \|\varphi(t + \tau) - \varphi(t)\|^2_E \, \d t \, \d \tau
  \end{equation}
  holding for all $\varphi \in \D\hl{E}$, or equivalently,
  \begin{equation}
  \label{eq:weight-Hs}
  \frac{1}{2s}
\int_0^{+\infty} \frac{1}{t^{2s}} \|\varphi(t)\|^2_E \, \d t \leq K' \int_0^{+\infty} \frac{1}{\tau^{2s+1}} \int_0^{+\infty} \|\varphi(t + \tau) - \varphi(t)\|^2_E \, \d t \, \d \tau,
  \end{equation}
  which is established in \cite[``Proof of (11.27)'', page 59]{LioMag68book}.
  \item \emph{Third case: $1/2 < s < 1$.} By definition, $\D\hl{E}$ is dense in $H_0^s\hl{E}$ and again the result follows if \cref{eq:weight-Hs} holds for all $\varphi \in \D\hl{E}$, which is established in the proof of 
 \cite[Theorem 11.3]{LioMag68book}.
  \item \emph{Fourth case: $s = 1$.} This case is immediate: use that $\varphi(0) = 0$ for all $\varphi \in H^1_0\hl{E}$.
  \item \emph{Fifth case: $1 < s < 3/2$.} We must show that $\|\varphi\|_{H^{s}(\R, E)} \leq K \| \varphi\|_{H^{s}\hl{E}}$ for all $\varphi \in \D\hl{E}$, or equivalently by \cref{lem:diff-shift},
  \begin{equation}
  \label{eq:shift-ineq-ext-0}
  \|\varphi\|_{H^{s-1}(\R, E)} + \|\dot{\varphi}\|_{H^{s-1}(\R, E)} \leq K \|\varphi\|_{H^{s-1}\hl{E}} + K \|\dot{\varphi}\|_{H^{s-1}\hl{E}}.
  \end{equation}
  But $\dot{\varphi} \in \D\hl{E}$ and $0 < s - 1 < 1/2$, so \cref{eq:shift-ineq-ext-0} follows from our previous analysis. \qedhere
\end{itemize}

\end{proof}
\begin{rem}
The extension by zero map is \emph{not} continuous $H^{1/2}(0, +\infty) \to H^{1/2}(\R)$; see \cite[Theorem 11.4, Chapter 1]{LioMag68book}.
\end{rem}

\begin{rem}[Finite intervals]
\label{rem:finite-interval}
The results of \cref{coro:charac-Hs-0,coro:ext-zero,lem:diff-shift} generalise without difficulty to the case where the half-line $(0, +\infty)$ is replaced by a finite interval of the form $(0, T)$, $T > 0$. To see this, build for instance smooth functions $\alpha$ and $\beta$ defined on $[0, T]$ such that $\alpha + \beta = 1$, $\alpha$ vanishes in a neighborhood of $T$ and $\beta$ vanishes in a neighborhood of $0$; see, e.g., the proof of \cite[Theorem 2.1, Chapter 1]{LioMag68book}. Then any vector-valued function $u$ defined on $(0, T)$ can be written as the sum of $\alpha u$ defined (after extension by zero) on $(0, +\infty)$ and $\beta u$ defined on $(- \infty, T)$. The properties of $\alpha$ and $\beta$ guarantee that the maps $u\mapsto \alpha u$ and $u \mapsto \beta u$ are bounded in suitable Sobolev norms. 
\end{rem}

The remainder of this section is devoted to characterising certain interpolation spaces. The first lemma in the series pertains to interpolation between $H^1_0$- and $L^2$-spaces.
\begin{lemma}
\label{lem:int-H10-L2}
Let $E$ be a Hilbert space. Then, for any $0 \leq s \leq 1$ with $s \not = 1/2$,
\begin{equation}
\label{eq:int-Hs0}
[H^1_0\hl{E}, L^2\hl{E}]_{1-s} = H^s_0\hl{E}.
\end{equation}
\end{lemma}
\begin{proof}
The proof is largely inspired by that of \cite[Theorem 11.6, Chapter 1]{LioMag68book}.  To obtain the inclusion from left to right in \cref{eq:int-Hs0}, we recall from \cref{coro:ext-zero} that the extension by zero map $\pi$ is continuous $L^2\hl{E} \to L^2(\R, E)$ and also $H^1_0\hl{E} \to H^1(\R, E)$. By interpolation, 
\begin{equation}
\label{eq:cont-pi-H0-int}
\pi ~\mbox{is continuous}~ [H^1_0\hl{E}, L^2\hl{E}]_{1 -s} \to H^{s}(\R, E).
\end{equation}
Let $u \in [H^1_0\hl{E}, L^2\hl{E}]_{1-s}$. By viewing $u$ as the restriction to $(0, +\infty)$ of $\pi u \in H^{s}(\R, E)$ by \cref{eq:cont-pi-H0-int}, we see that $u \in H^{s}\hl{E}$. If $0 \leq s < 1/2$, $H^{s}\hl{E} = H^s_0\hl{E}$ by \cref{coro:charac-Hs-0} and the desired inclusion is proved. If $1/2 < s \leq 1$, according to \cref{coro:charac-Hs-0}, it suffices to check that $u(0) = 0$, which holds because $u = [\pi u]|_{(0, +\infty)}$ with $\pi u \in H^s(\R, E)$, $s > 1/2$, and $\pi u = 0$ a.e.\  in $(-\infty, 0)$. For the inclusion from right to left, consider the map $R : L^2(\R, E) \to L^2\hl{E}$ defined by
\begin{equation}
[R u](t) = u(t) - u(-t) \quad \mbox{for a.e.\ }t > 0, \quad u \in L^2(\R, E).
\end{equation}
Then, $R$ is continuous $L^2(\R, E) \to L^2\hl{E}$ but also $H^1(\R, E) \to H^1_0\hl{E}$. By interpolation,
\begin{equation}
\label{eq:cont-int-r}
R~\mbox{is continuous}~H^{s}(\R, E) \to [H^1_0\hl{E}, L^2\hl{E}]_{1-s}.
\end{equation}
Furthermore, we see that $R \pi u = u$ for any $u \in L^2\hl{E}$.
Keeping in mind that $s \not = 1/2$,
\cref{coro:charac-Hs-0,coro:ext-zero} guarantee that $\pi$ (continuously) maps $H^s_0\hl{E}$ into $H^s(\R, E)$; thus, by \cref{eq:cont-int-r}, writing any $u \in H^s_0\hl{E}$ as $R \pi u$ reveals that $H^s_0\hl{E}$ is indeed contained in $[H^1_0\hl{E}, L^2\hl{E}]_{1-s}$. The proof is now complete.
\end{proof}
\begin{rem}
\Cref{lem:yet-another-interpolation-lemma} is false when $s = 1/2$; see \cite[Theorem 11.7, Chapter 1]{LioMag68book}. 
\end{rem}
With a similar proof, we also obtain the following result, which we need for the \emph{ad hoc} ``elliptic regularity'' argument of \cref{prop:h-eta-est-io}.
\begin{lemma}
\label{lem:yet-another-interpolation-lemma}
Let $E$ be a Hilbert space.
Then, for any $0 \leq s < 1/2$,
\begin{equation}
\label{eq:YAIL}
\begin{aligned}
[H^2\hl{E}\cap H^1_0\hl{E}, H^1_0\hl{E}]_{1-s} &= H^{1+s}\hl{E} \cap H^1_0\hl{E} \\& =H^{1+s}_0\hl{E}.
\end{aligned}
\end{equation}
\end{lemma}
\begin{proof}
First, recall that, given $1 \leq r \leq 2$, $H^r\hl{E} \cap H^1_0\hl{E}$ equipped with the $H^r$-norm is a closed subspace of $H^r\hl{E}$; in particular, it is also a Hilbert space. Thus, the interpolation spaces in \cref{eq:YAIL} are all well-defined. Moreover, the second equality in \cref{eq:YAIL} follows from \cref{coro:charac-Hs-0}.
 Now \cref{lem:reflec-plus} provides an extension map $\pi_2$ that is continuous $H^r\hl{E} \to H^r(\R, E)$ for all $0 \leq r \leq 2$. Since the norms on $H^2\hl{E} \cap H^1_0\hl{E}$ and $H^1_0\hl{E}$ are those induced by $H^2\hl{E}$ and $H^1\hl{E}$, respectively, it follows that $\pi_2$ is also continuous from the former spaces into $H^2(\R, E)$ and $H^1(\R, E)$, respectively. By interpolation,
\begin{equation}
\pi_2 ~\mbox{is continuous}~[H^2\hl{E}\cap H^1_0\hl{E}, H^1_0\hl{E}]_{1-s} \to H^{1+s}(\R, E).
\end{equation}
Therefore, any $u$ in $[H^2\hl{E}\cap H^1_0\hl{E}, H^1_0\hl{E}]_{1-s}$ is the restriction to $(0, +\infty)$ of some element of $H^{1+s}(\R, E)$ and, in particular, belongs to $H^{1+s}\hl{E}$. Of course, such functions $u$ also belong to $H^1_0\hl{E}$ and the inclusion from left to right in \cref{eq:YAIL} is now proved. As for the opposite inclusion, consider again the map $R$ introduced in the proof of \cref{lem:int-H10-L2}. It is clear that $R$ also continuously maps $H^2(\R, E)$ into $H^2\hl{E}\cap H^1_0\hl{E}$ and, by interpolation,
\begin{equation}
R ~\mbox{is continuous}~ H^{1+s}(\R, E) \to [H^2\hl{E}\cap H^1_0\hl{E}, H^1_0\hl{E}]_{1-s}.
\end{equation}
Recall that $R$ is a left inverse for the extension by zero map $\pi$. Furthermore, by \cref{coro:ext-zero}, $\pi$ is continuous $H^{1+s}_0\hl{E} \to H^{1+s}(\R, E)$. By writing $u = R \pi u$, we obtain that any $u \in H^{1+s}_0\hl{E}$ must belong to $[H^2\hl{E}\cap H^1_0\hl{E}, H^1_0\hl{E}]_{1-s}$, as required.
\end{proof}

 Our next lemma is essentially \cite[Proposition 2.1, Chapter 1]{LioMag68book}.
\begin{lemma}
\label{lem:int-V-V-ast}
Let $H$ and $V$ be Hilbert spaces such that $V \hookrightarrow H$ with continuous and dense embedding. Identifying  $H$ with its antidual, we have $[V, V^\ast]_{1/2} = H$ with equivalence of norms.
\begin{proof}
Replacing if the scalar product of $V$ by an equivalent one if necessary, we can assume that $\|u\|_V \geq \|u\|_H$ for all $u \in V$. As a first step, define  $L \in \L(V, V^\ast)$ by $\langle Lu, v\rangle_{V^\ast, V} = \langle u, v \rangle_V$ for all $u, v \in V$.
In view of the embedding chain $V \hookrightarrow H \hookrightarrow V^\ast$, we may also regard $L$ as an unbounded operator on $H$, with $\dom(L) = \{ u \in V : Lu \in H \}$. Then, $L$ is a strictly positive (i.e., coercive) self-adjoint operator on $H$. The spectral theorem and functional calculus of self-adjoint operators give fractional powers $L^s$ for all $s \in \R$. Then $\dom(L^{1/2}) = V$, moreover $V^\ast$ coincides with the completion of $H$ with respect to the norm $\|L^{-1/2}\cdot\|_H$. In particular, $L^{-1/2}$ extends as an isomorphism $V^\ast \to H$. Recalling that $L^{-1}$ to an isomorphism $V^\ast \to V$, this gives $[V, V^\ast]_{1/2} = H$ by virtue of the ``diagonalisation'' characterisation of quadratic interpolation spaces; see the discussion in \cref{sec:def} and also, more specifically, \cite{ChaHew15}.
\end{proof}
\end{lemma}
The final result of this section shows that we can recover $L^2$-spaces by interpolation between Sobolev spaces of positive and negative orders.
\begin{prop}
\label{lem:L2-rec}
Let $E$ be a Hilbert space and $0 \leq s \leq 1$. Then, up to equivalence of norms,
\begin{equation}
\label{eq:rec-L2}
[H^{1 - s}(0, +\infty; E), H^{-s}(0, +\infty; E)]_{1 - s} = L^2(0, +\infty; E).
\end{equation}
\end{prop}
\begin{proof}
The result is trivial when $s = 0$ or $s = 1$. The proof is split into different cases depending on the value of $s \in (0, 1)$.
\begin{itemize}[wide]
  \item 
 \emph{First case: $s = 1/2$.} This special case readily follows from \cref{lem:int-V-V-ast}, where we set $H = L^2\hl{E}$ and $V = H^{1/2}\hl{E}$ and  recall from \cref{coro:charac-Hs-0} that
\begin{equation}
 H^{1/2}\hl{E} = H^{1/2}_0\hl{E}, \quad\mbox{hence}\quad (H^{1/2}(0, +\infty))^\ast = H^{-1/2}(0, +\infty; E).
\end{equation}
\item
\emph{Second case: $0 < s < 1/2$.}
By \cref{coro:charac-Hs-0}, $H^s_0\hl{E}$ and $H^s\hl{E}$ coincide and
\begin{equation}
\label{eq:sp-int-Hs}
[H^{1 - s}(0, +\infty; E), H^{-s}(0, +\infty; E)]_{1 - s} = [H^{1-s}\hl{E}, (H^s\hl{E})^\ast]_{1 - s}.
\end{equation}
From now on we follow the proof of \cite[Theorem 12.5, Chapter 1]{LioMag68book}.
First, using for instance \cref{lem:int-V-V-ast}, we have $[H^1\hl{E}, H^1\hl{E}^\ast]_{1/2} = L^2\hl{E}$. Hence
\begin{equation}
\label{eq:int-H-pos}
\begin{aligned}
H^{1-s}\hl{E} &= [H^1\hl{E}, L^2\hl{E}]_s\\
&= [H^1\hl{E}, [H^1\hl{E}, (H^1\hl{E})^\ast]_{1/2}]_s \\
&= [H^1\hl{E}, (H^1\hl{E})^\ast]_{s/2},
\end{aligned}
\end{equation}
where we used the definition \cref{eq:def-inter} of fractional order Sobolev spaces and the reiteration property. Similarly, using also duality \cite[Theorem 3.7]{ChaHew15}, we obtain
\begin{equation}
\label{eq:int-H-neg}
(H^s\hl{E})^\ast = [H^1\hl{E}, (H^1\hl{E})^\ast]_{1 - s/2}.
\end{equation}
Plugging \cref{eq:int-H-pos,eq:int-H-neg} into \cref{eq:sp-int-Hs} and using the reiteration property lead to
\begin{equation}
[H^{1 - s}(0, +\infty; E), H^{-s}(0, +\infty; E)]_{1 - s} = [H^1\hl{E}, (H^1\hl{E})^\ast]_{1/2} = L^2\hl{E},
\end{equation}
as required. 
\item\emph{Third case: $1/2 < s < 1$.} 
In place of \cref{eq:sp-int-Hs}, we shall use the fact that 
\begin{equation}
\label{eq:sp-int-Hs-0}
[H^{1 - s}(0, +\infty; E), H^{-s}(0, +\infty; E)]_{1 - s} = [H^{1-s}_0\hl{E}, H^{-s}\hl{E}]_{1 - s},
\end{equation}
which is again a consequence of \cref{coro:charac-Hs-0}. Working this time in the pivot duality $H^1_0\hl{E} \hookrightarrow L^2\hl{E} \hookrightarrow H^{-1}\hl{E}$, \cref{lem:int-V-V-ast} gives
\begin{equation}
\label{eq:int-H10-H-1}
[H^1_0\hl{E}, H^{-1}\hl{E}]_{1/2} = L^2\hl{E}.
\end{equation}
Recall from \cref{lem:int-H10-L2} that
\begin{equation}
\label{eq:int-Hs0-bis}
H^{1 -s}_0\hl{E} = [H^1_0\hl{E}, L^2\hl{E}]_{s}.
\end{equation}
Starting from \cref{eq:sp-int-Hs-0}
and using \cref{eq:int-H10-H-1,eq:int-Hs0}, we readily complete the proofas in the previous case. \qedhere
\end{itemize}
\end{proof}

\section{Other technical results}
\label{sec:cutoff}
In our main  proofs, we make use of the following lemma.
\begin{lemma}
\label{lem:localize-shift}
Let $E$ be a Hilbert space. Let $T > 0$ and $0 \leq s \leq 2$. There exists $K > 0$ such that, for all $u \in H^{-s}\hl{E}$,
\begin{equation}
\label{eq:H-s-conc}
\|u\|_{H^{-s}\hl{E}}^2 \leq K \sum_{n=0}^{+\infty} \| u(nT + \cdot) \|^2_{H^{-s}(0, T; E)} + \|u((n+1/2)T + \cdot)\|^2_{H^{-s}(0, T; E)}.
\end{equation}
\end{lemma}
\Cref{lem:localize-shift} allows us to estimate negative order Sobolev norms of a function $u$ by summing estimates on finite time windows. Note that, without the shifted term in the sum in 
\cref{eq:H-s-conc}, the inequality would be false; indeed, in the scalar case define $u \in H^{-1}(0, +\infty)$ by $\langle u, \varphi \rangle \triangleq \overline{\varphi(T)}$, then $\sum_{n=0}^{+\infty} \| u(nT + \cdot) \|^2_{H^{-1}(0, T)} = 0$ but of course $u \not = 0$. The result follows from another lemma, which we state next.
\begin{lemma}
\label{lem:localize-cut}
Let $E$ be a Hilbert space. Let $T > 0$, $0 \leq s \leq 2$  and $\chi \in \D(0, T)$. There exists $K > 0$ such that 
\begin{equation}
\label{eq:localize-cut}
\sum_{n=0}^{+\infty} \|\chi(\cdot) \varphi(nT + \cdot)\|_{H^s(0, T; E)}^2 \leq K \|\varphi\|_{H^s\hl{E}}^2, \quad \varphi \in H^{s}\hl{E}.
\end{equation}
\end{lemma}

\begin{proof} By density arguments, it suffices to prove \cref{eq:localize-cut} for $\varphi \in \D\hl{E}$.
We omit the straightforward proofs of the cases $s = 0, 1, 2$ and
suppose, to begin with, that $0 < s < 1$. 
Define
$
\Chi(t) = \sum_{n = 0}^{+\infty} \chi(t - nT)
$
for $t \geq 0$ (here $\chi$ is extended by zero outside of $(0, T)$).
It is clear that $ \Chi \in \C^\infty(\R^+) \cap W^{1, \infty}(0, +\infty)$ and, as a result, $\Chi \varphi \in H^s\hl{E}$.     In order to estimate the left-hand side of \cref{eq:localize-cut} in terms of the $H^{s}$-norm of $\Chi \varphi$, we use the Besov characterisation of $H^{s}$-spaces of one variable \cite{LioMag68book,Sim90} and start by writing
\begin{equation}
\label{eq:Hs-integral}
 \int_0^{+\infty} \frac{1}{\tau^{2s +1}} \int_0^{+\infty} \|\Chi(t + \tau) \varphi(t + \tau) - \Chi(t) \varphi(t) \|^2_E \, \d t \, \d \tau \leq K \|\Chi \varphi\|^2_{H^s\hl{E}}.
\end{equation}
In what follows, we use the notation $\varphi_n \triangleq \varphi(nT + \cdot)$.
 Since $\chi$ has compact support, there exist $\tau_0, \eps > 0$  such that if $\chi(t) \not = 0$,  then $t + \tau_0 \in (\eps, T- \eps)$. Hence we may write
\begin{multline}
\label{eq:Hs-integral-sum}
\int_0^{\tau_0} \frac{1}{\tau^{2s +1}} \int_0^{+\infty} \|\Chi(t + \tau) \varphi(t + \tau) - \Chi(t) \varphi(t) \|^2_E \, \d t \, \d \tau  
\\=  \sum_{n=0}^{+\infty} \int_0^{\tau_0} \frac{1}{\tau^{2s +1}} \int_0^T \|\chi(t + \tau)\varphi_n (t + \tau) - \chi(t)\varphi_n (t)\|_E^2 \, \d t \, \d \tau.
\end{multline}
Using again  the equivalent translation-type  Sobolev norm,  this time  on the finite interval $(0, T)$, we see that there exists $ K> 0$ such that, for all $n \in \N$,
 \begin{equation}
\label{eq:finite-chi}
\|\chi \varphi_n\|^2_{H^{-s}(0, T; E)} \leq K 
\int_0^{\tau_0} \frac{1}{\tau^{2s +1}} \int_0^T \|\chi(t + \tau)\varphi_n (t + \tau) - \chi(t)\varphi_n (t)\|_E^2 \, \d t \, \d \tau + K \|\chi \varphi_n\|^2_{L^2(0, T; E)}.
\end{equation}
Summing \cref{eq:finite-chi} over $\mathbb{N}$, using \cref{eq:Hs-integral-sum} and then \cref{eq:Hs-integral}, and absorbing the left-over $L^2$-term, we obtain
\begin{equation}
\sum_{n=0}^{+\infty} \|\chi \varphi_n\|^2_{H^{-s}(0, T; E)} \leq K \left(\|\Chi \varphi\|^2_{L^2\hl{E}} +   ~\mbox{Left-hand side of \cref{eq:Hs-integral-sum}}\right) \leq K' \|\Chi \varphi\|^2_{H^s\hl{E}}.
\end{equation}
We recall that  multiplying by $\Chi$ is continuous on $H^s\hl{E}$ and immediately deduce \cref{eq:H-s-conc}. It remains to deal with the case $1 < s < 2$. To do so, we simply write $(\d/\d t) \chi \varphi_n = \dot{\chi}\varphi_n + \chi\dot{\varphi}_n$ and use \cref{lem:diff-shift}, resulting in
\begin{equation}
\label{eq:d-chi-phi}
\|\chi \varphi_n\|^2_{H^{s}(0, T; E)} \leq K \|{\chi}\varphi_n\|^2_{L^2(0, T; E)} + K \|\dot{\chi}\varphi_n\|^2_{H^{1-s}(0, T; E)} + K \|{\chi}\dot{\varphi}_n\|^2_{H^{1-s}(0, T; E)}, \quad n \in \N.
\end{equation}
Since $\dot{\chi} \in \D(0, T; E)$, \cref{eq:d-chi-phi} shows that the result readily follows from the previous case $0 \leq s < 1$.
\end{proof}

\begin{proof}[Proof of \cref{lem:localize-shift}] The first step of the proof consists in some simple manipulation with smooth cut-off functions.
We start by picking some $\chi \in \D(0, T)$ taking values in $[0, 1]$ and such that $\chi(t) = 1$ for $t \in [T/3, 2T/3]$. 
Then, we define  $\rho : [T/2, 3T/2] \to [0, 1]$ as follows:
\begin{equation}
\rho(t) \triangleq 1 - \chi(t), \quad T/2 \leq t \leq T; \qquad \rho(t) =  1 - \chi(t - T), \quad T \leq t \leq 3T/2.
\end{equation} 
Note that $\rho \in \D(T/2, 3T/2)$. In order to deal with the boundary $t = 0$,
we also pick $\chi_0 \in \C^\infty([0, T])$ such that $\chi_0(t) = 1$ for $t \in [0, T/2]$ and $\chi_0(t) = \chi(t)$ for $t \in [T/2, T]$. Again, $\chi$ and $\rho$ are extended with zero outside of their domains of definition. Set
$
\Chi \triangleq \sum_{n=0}^{+\infty} \chi(\cdot - nT)
$
and
$\Rho \triangleq \sum_{n=0}^{+\infty} \rho(\cdot - nT)$, with
$\Chi, \Rho \in \C^\infty(\R^+)$,
and observe that 
\begin{equation}
\label{eq:sum-chi}
\Chi(t) + \Rho(t) 
 = 1, \quad t \geq T/2; \qquad \chi_0(t)+ \rho(t) = 1, \quad 0 \leq t \leq T/2.
\end{equation}
We now turn our attention to \cref{eq:H-s-conc}, which we will establish for arbitrary $u \in L^2\hl{E}$. Let $\varphi \in \D\hl{E}$.
To simplify the notation, we write $\psi \triangleq \rho(\cdot - (1/2)T) \in \D(0, T)$ and also, for $n \in \N$,
\begin{equation}
\varphi_n \triangleq \varphi(nT + \cdot), \quad u_n \triangleq u(nT + \cdot), \quad \varphi_{1/2 + n} \triangleq \varphi((n+1/2)T + \cdot), \quad u_{n+1/2} \triangleq u((n+1/2)T + \cdot).
\end{equation}
With \cref{eq:chain-H-spaces,eq:dual-norm} in mind, we will estimate the $H^{-s}$-norm of $u$ by pairing it with
 $\varphi$. In view of \cref{eq:sum-chi}, we may write $\varphi = \Chi \varphi + \Rho \varphi$ on $(T/2, +\infty)$, and $\varphi = \chi_0 \varphi + \rho \varphi$ on $(0, T/2)$, which in turn leads to
\begin{multline}
\label{eq:dev-u-bumb}
\int_0^{+\infty} \langle u(t), \varphi(t) \rangle_E \, \d t = \int_0^{T}  \langle u(t), \chi_0(t)  \varphi(t) \rangle_E \\
+ \sum_{n=1}^{+\infty} \int_0^T \langle u_n(t), \chi(t) \varphi_n(t)\rangle_E \, \d t + \sum_{n=0}^{+\infty} \int_0^T \langle u_{n+1/2}(t), \psi(t) \varphi_{n+1/2}(t)\rangle_E \,  \d t.
\end{multline}
Furthermore, we see that
 $\chi_0 \varphi$ belongs to $\D(0, T; E)$,
and so do
 $\chi \varphi$ and $\psi \varphi_{n+1/2}$, $n \in \N$. 
This allows us to deduce from \cref{eq:dev-u-bumb} that
\begin{multline}
\left | \int_0^{+\infty} \langle u(t), \varphi(t) \rangle_E \, \d t  \right | \leq \|u\|_{H^{-s}(0, T; E)} \|\chi_0\varphi\|_{H^s(0, T; E)} \\
+ \sum_{n=1}^{+\infty} \|u_n\|_{H^{-s}(0, T; E)} \|\chi \varphi_n\|_{H^s(0, T; E)} + \sum_{n=0}^{+\infty} \|u_{n+1/2}\|_{H^{-s}(0, T; E)} \|\psi \varphi_{n+1/2} \|_{H^s(0, T; E)}.
\end{multline}
The Cauchy--Schwartz inequality gives
\begin{multline}
\label{eq:u-varphi-cs}
\left | \int_0^{+\infty} \langle u(t), \varphi(t) \rangle_E \, \d t \right |^2 \leq  4 \left ( \sum_{n=1}^{+\infty} \|u_{n+1/2}\|_{H^{-s}(0, T; E)}^2 \right ) \left ( \sum_{n=0}^{+\infty} \|\psi \varphi_{n+1/2}\|_{H^s(0, T; E)}^2 \right ) \\
+ 4 \left ( \sum_{n=1}^{+\infty} \|u_n\|_{H^{-s}(0, T; E)}^2 \right ) \left ( \sum_{n=1}^{+\infty} \|\chi \varphi_n \|_{H^s(0, T; E)}^2 \right ) 
 + 4  \|u\|_{H^{-s}(0, T; E)}^2 \|\chi_0 \varphi \|_{H^s(0, T; E)}^2.
\end{multline}
By \cref{lem:localize-cut},
\begin{equation}
\label{eq:est-cut-sum}
\sum_{n=1}^{+\infty} \|\chi \varphi_n\|_{H^s(0, T; E)}^2 \leq K \|\varphi\|_{H^s\hl{E}}^2, \quad \sum_{n=0}^{+\infty} \|\psi \varphi_{n+1/2}\|_{H^s(0, T; E)}^2 \leq K \|\varphi(T/2 + \cdot)\|^2_{H^s\hl{E}}.
\end{equation}
Left translations are continuous on $L^2\hl{E}$, on $H^1\hl{E}$ and thus also on $H^s\hl{E}$ by interpolation. This gives 
$\|\varphi(T/2 + \cdot)\|^2_{H^s\hl{E}} \leq K' \|\varphi\|_{H^s\hl{E}}$ and, coming back to \cref{eq:u-varphi-cs,eq:est-cut-sum},
\begin{multline}
\label{eq:almost}
|\langle u, \varphi\rangle_{L^2\hl{E}}|^2 \leq K \|\varphi\|_{H^s\hl{E}}^2 \sum_{n=0}^{+\infty} \|u_n\|^2_{H^{-s}(0, T; E)} + \|u_{n+1/2}\|^2_{H^{-s}(0, T; E)} \\ + K \|u\|_{H^{-s}(0, T; E)}^2 \|\chi_0 \varphi \|_{H^s(0, T; E)}^2.
\end{multline}
After absorbtion of the last term in \cref{eq:almost}, the proof is  complete.
\end{proof}


Next, we give a slightly modified version of an \emph{a priori} estimate for the Helmhotz equation from \cite{Spe14,BasSpe15}, with explicit dependence on the wavenumber. \Cref{lem:elliptic-tang}, which we used above, is a byproduct of the proof. In what follows, $\vec{\nabla}_\tau$ denotes the tangential gradient, that is, the projection of the gradient onto the tangent hyperplane.\footnote{Because the domain $\Omega$ is Lipchitz, the outward unit normal vector $\vn$ is defined for a.e.\  point of the boundary.}
\begin{lemma}
\label{lem:helmhotz}
Let $\Omega$ be a Lipschitz domain in $\R^d$, $d \geq 2$, which satisfies one of the following conditions:
\begin{itemize*}
\item $\Omega$ has bounded boundary; \item $\Omega$ is the half-plane. 
\end{itemize*}
There exists $K > 0$ such that for all $k \geq 0$, $w \in H^1(\Omega)$ and $f \in L^2(\Omega)$, if $w$ solves $\Delta w + k^2 w = - f$ and $\partial_\vn w \in L^2(\partial \Omega)$ then $w|_{\partial \Omega} \in H^1(\partial \Omega)$ and
\begin{equation}
\label{eq:a-priori-estimate-interior}
\|\vnabla_\tau w\|_{L^2(\partial \Omega)^{d-1}} \leq K( \|\partial_\vn w\|_{L^2(\partial \Omega)} + k \|w|_{\partial \Omega}\|_{L^2(\partial \Omega)} + \|\nabla w \|_{L^2(\Omega)^d} + k \|w\|_{L^2(\Omega)} + \|f\|_{L^2(\Omega)}).
\end{equation}
In particular, for $k \geq 1$ (say) and up to a change of constant $K$, the $H^1(\partial \Omega)$-norm of $w|_{\partial \Omega}$ is bounded by the right-hand side of \cref{eq:a-priori-estimate-interior}.
\end{lemma}
\begin{proof}
We will use the space $W$ defined in \cref{coro:elliptic-W}, which we recall is comprised of all elements $w \in H^1(\Omega)$ such that $\Delta w \in L^2$ and $\partial_\vn w \in L^2(\partial \Omega)$. We claim that the space $\tilde{W} = \{\varphi|_{\Omega} : \varphi \in \D(\R^{d})\}$
is dense in $W$; the proof is ommited but is similar to that given in \cite[Appendix A]{MoiSpe14survey}.
The following Rellich-type identity holds
for all real-valued vector fields \smash{$\vec{h} = (h_1, \dots, h_d) \in \C^1(\overline{\Omega})^d$, $w \in \tilde{W}$} and $k \geq 0$:
\begin{multline}
\label{eq:rellich}
\int_\Omega 2\re \left( \overline{\vec{h}\cdot \vnabla w} (\Delta w + k^2 w)  \right) - (\vnabla \cdot \vec{h})(\|\vnabla{w}\|^2 - k^2 |w|^2) + 2 \re  \sum_{i = 1}^{d}\sum_{j = 1}^d \dl{h_j}{x_i} \dl{w}{x_i} \overline{\dl{w}{x_j}} \, \d x \\ 
= \int_{\partial \Omega} 2 \re \left( \overline{\vec{h}\cdot \vnabla w} \partial_\vn w \right) + (k^2 |w|^2 - \|\vnabla w\|^2) (\vec{h} \cdot \vn) \, \d \sigma;
\end{multline}
see, e.g., \cite[Lemma 3.7]{Spe14}. We shall deduce trace identities from \cref{eq:rellich} by choosing an appropriate vector field \smash{$\vec{h}$}. If $\Omega$ is bounded, \cite[Lemma 1.5.1.9]{Gri85book} states that there exist \smash{$\vec{h} \in \C^{\infty}(\overline{\Omega})^d$} and $\delta > 0$ such that $\smash{\vec{h}} \cdot \vn \geq \delta$ a.e.\  on $\partial \Omega$. In fact, by following the proof of this lemma, one may construct $\smash{\vec{h}} \in \D(\R^{d})^d$ satisfying the same property even for domains $\Omega$ that are unbounded but have bounded boundary.
On the other hand, if $\Omega$ is the half-space $\{(x, y) \in (0, +\infty) \times \R^{d-1}\}$, we simply let $\smash{\vec{h}(x, y)} \triangleq - \chi(x)(1, 0)$ for all $(x, y) \in \overline{\Omega}$, where $\chi : \R^+ \to \R$ is any smooth function equal to $1$ near $0$ and $0$ away from $0$; such a vector field also satisfies the same property. In any case, with a series of applications of the Cauchy--Schwarz and Young inequalities we may deduce from \cref{eq:rellich} that, for all $w \in \tilde{W}$ and $k \geq 0$,
\begin{equation}
\label{eq:rellich-a}
\delta \int_{\partial \Omega} \| \vnabla_\tau{w}\|^2 \, \d \sigma \leq K_{\vec{h}} 
\left(
\int_{\partial \Omega} |\partial_\vn w|^2 + k^2 |w|^2 \, \d \sigma + \int_\Omega |\Delta w + k^2 w|^2 + \|\vnabla w\|^2 + k^2 |w|^2 \, \d x
 \right),
\end{equation}
where we also used that $\|\vnabla w\|^2 = |\partial_\vn w|^2 + \|\vnabla_\tau w\|^2$ a.e.\  on $\partial \Omega$; here,
$K_{\vec{h}}$ is a positive constant that depends only on our construction of $\smash{\vec{h}}$. For fixed $k \geq 0$, the right-hand side of \cref{eq:rellich-a} is controlled by $\|w\|^2_W$. Furthermore, since $w$ is assumed to be sufficiently smooth,
\begin{equation}
\|w|_{\partial \Omega}\|^2_{H^1(\partial \Omega)} \leq K \int_{\partial \Omega} \| \vnabla_\tau{w}\|^2 + |w|^2 \, \d \sigma; 
\end{equation}
this is immediate if $\Omega$ is the half-space, and we refer the reader to \cite[Appendix A]{ChaGra12survey} in the case where $\Omega$ has bounded boundary. By density of $\tilde{W}$ in $W$, we see that the trace map extends to a continuous linear operator from $W$ into $H^1(\partial \Omega)$, as required (in particular \cref{lem:elliptic-tang} is now proved). This also shows that \cref{eq:rellich-a} is valid for arbitrary $w \in W$. As for the \emph{a priori} estimate \cref{eq:a-priori-estimate-interior} which holds uniformly in $k \geq 0$, it suffices to observe that any $w \in H^1(\Omega)$ solving $\Delta w + k^2 w = f$, $f \in L^2(\Omega)$, with $\partial_\vn w \in L^2(\partial \Omega)$ belongs to $W$, and then use \cref{eq:rellich-a} with $|\Delta w + k^2 w|^2$ replaced by $|f|^2$. The final statement of \cref{lem:helmhotz} readily follows.
\end{proof}

Finally, we prove the claim of \cref{rem:reg-wave}. The key ingredients are Lions and Magenes' intermediate derivative theorem \cite{LioMag68book} and Strauss' results on continuity of solutions to abstract second-order evolution equations \cite{Str66}.

\begin{proof}[Complement to \cref{theo:sharp-wave}]
Let $u \in L^2(\partial \Omega \times (0, T))$. We assume that the corresponding solution $w$ to \cref{eq:neumann-wave-IBVP} satisfies $w \in H^{1-\eta}(\Omega \times (0, T))$.  Recall that on the one hand
 $H^{1-\eta}(\Omega \times (0, T))= L^2(0, T; H^{1-\eta}(\Omega)) \cap H^{1-\eta}(0, T; L^2(\Omega))$ and on the other hand, denoting by $\{H_s\}_{s\in \R}$ the abstract Sobolev scale associated with $-\Delta_N$, $H_{1-\eta} = H^{1-\eta}(\Omega)$ and (here, $0 \leq \eta < 1/2$) $H_{-\eta} = H^{-\eta}(\Omega)$. Thus $w \in L^2(0, T; H_{1-\eta})$ and  $\ddot{w} \in L^2(0, T; H_{-1 - \eta}) + L^2(0, T; H_{-1}) = L^2(0, T; H_{-1-\eta})$. We introduce the lifted variable $v = (1 - \Delta_N)^{-\eta}w$. Then
\begin{equation}
\label{eq:IDT}
v \in L^2(0, T; H_1) = L^2(0, T; H^1(\Omega)), \quad \ddot{v} \in L^2(0, T; H_{-1}) = L^2(0, T; (H^1(\Omega))^\ast).
\end{equation}
By \cite[Proposition 2.2, Chapter 1]{LioMag68book}, which is a corollary of the intermediate derivative theorem \cite[Theorem 2.3, Chapter 1]{LioMag68book}, \cref{eq:IDT} implies that
\begin{equation}
\label{eq:IDT-bis}
\dot{v} \in L^2(0, T; L^2(\Omega)).
\end{equation}
Furthermore, because $w$ solves $\ddot{w} - \Delta_Nw = \gamma^\ast u$, $v$ must solve $\ddot{v} - \Delta_Nv = (1 - \Delta_N)^{-\eta}\gamma^\ast u$ with $(1 - \Delta_N)^{-\eta}\gamma^\ast u \in L^2(0, T; (H_1(\Omega)^\ast))$. Together with \cref{eq:IDT}, \cref{eq:IDT-bis} and \cite[Theorems 4.2 and 4.3]{Str66}, we deduce that $v \in \C([0, T], H^1(\Omega)) \cap \C^1([0, T], L^2(\Omega))$. Returning to the $w$-variable, this gives
\begin{equation}
w \in \C([0, T], H^{1-\eta}(\Omega)) \cap \C^1([0, T], H^{-\eta}(\Omega)),
\end{equation}
as required.
\end{proof}

\bibliographystyle{alpha}
\bibliography{regularity-mu.bib}

\end{document}